\documentclass[a4paper,twosided,11pt]{article}
\usepackage{amssymb,amsmath,amsthm,verbatim}
\usepackage{thmtools,mathtools}
\usepackage[dvipsnames,svgnames,x11names,hyperref]{xcolor}
\usepackage[colorlinks=true,linkcolor=Maroon,citecolor=blue,urlcolor=blue,hypertexnames=false,linktocpage]{hyperref}
\usepackage{lipsum}
\usepackage{array}
\usepackage{bibspacing}

\newtheorem{theorem}{Theorem}[section]
\newtheorem*{theorem*}{Theorem}

\newtheorem{corollary}[theorem]{Corollary}

\newtheorem{lemma}[theorem]{Lemma}

\newtheorem*{prop*}{Proposition}
\newtheorem{proposition}[theorem]{Proposition}
\newtheorem{conjecture}[theorem]{Conjecture}

\newtheorem*{conjecture*}{Conjecture}
\newtheorem{definition}[theorem]{Definition}
\newtheorem*{definition*}{Definition}
\theoremstyle{definition}
\newtheorem{remark}[theorem]{\textbf{Remark}}
\newtheorem*{remark*}{Remark}
\newtheorem*{fact*}{Fact}

\numberwithin{equation}{section}

\newcommand{\mycorner}{\llcorner}
\newcommand{\SSS}{\mathrm{Sing}}
\newcommand{\Sym}{\mathcal{S}}

\newcommand{\vol}{\textrm{vol}}
\newcommand{\norm}[1]{\left\Vert#1\right\Vert}

\newcommand{\abs}[1]{\left\vert#1\right\vert}
\newcommand{\set}[1]{\left\{#1\right\}}
\newcommand{\brac}[1]{\left(#1\right)}
\newcommand{\scalar}[1]{\left \langle #1 \right \rangle}

\newcommand{\T}{\mathbf{T}}
\newcommand{\G}{\mathbf{G}}
\newcommand{\Q}{\mathbf{Q}}

\newcommand{\R}{\mathbb{R}}

\newcommand{\CD}{\mathrm{CD}}

\newcommand{\Z}{\mathbb{Z}}

\newcommand{\I}{\mathcal{I}}
\newcommand{\II}{\mathrm{I\!I}}
\newcommand{\Id}{\mathrm{Id}}

\renewcommand{\H}{\mathcal{H}}

\newcommand{\eps}{\epsilon}

\newcommand{\Ric}{\text{\rm Ric}}

\renewcommand{\S}{\mathbb{S}}

\newcommand{\n}{\mathfrak{n}}

\newcommand{\m}{\mathfrak{m}}

\DeclareMathOperator{\supp}{supp}

\DeclareMathOperator{\interior}{int}

\newcommand{\bmu}{\mu_0}
\newcommand{\bg}{g_0}
\newcommand{\bI}{\mathcal{I}_0}
\newcommand{\bM}{M_0}

\begin{document}

\title{Isoperimetric Inequalities on Slabs\\
with applications to Cubes and Gaussian Slabs} 
\date{}

\author{Emanuel Milman\textsuperscript{1}}

\footnotetext[1]{Department of Mathematics, Technion - Israel
Institute of Technology, Haifa, Israel. Email: emilman@tx.technion.ac.il.\\
The research leading to these results is part of a project that has received funding from the European Research Council (ERC) under the European Union's Horizon 2020 research and innovation programme (grant agreement No 101001677). }

\begingroup    \renewcommand{\thefootnote}{}    \footnotetext{2020 Mathematics Subject Classification: 49Q20, 49Q10, 53A10.}
    \footnotetext{Keywords: Isoperimetric inequalities, Slabs, Cube, Hypercube, Gaussian Measure.}
\endgroup

\maketitle

\begin{abstract}
We study isoperimetric inequalities on ``slabs", namely weighted Riemannian manifolds obtained as the product of the uniform measure on a finite length interval with a codimension-one base. As our two main applications, we consider the case when the base is the flat torus $\R^2 / 2 \Z^2$ and the standard Gaussian measure on $\R^{n-1}$. 

The isoperimetric conjecture on the three-dimensional cube predicts that minimizers are enclosed by spheres about a corner, cylinders about an edge and coordinate planes. This has only been established for relative volumes close to $0$, $1/2$ and $1$ by compactness arguments. Our analysis confirms the isoperimetric conjecture on the three-dimensional cube with side lengths $(\beta,1,1)$ in a new range of relative volumes $\bar v \in [0,1/2]$. In particular, we confirm the conjecture for the standard cube ($\beta=1$) for all $\bar v \leq 0.120582$, when $\beta \leq 0.919431$ for the entire range where spheres are conjectured to be minimizing, and also for all $\bar v \in [0,1/2] \setminus (\frac{1}{\pi} - \frac{\beta}{4},\frac{1}{\pi} + \frac{\beta}{4})$. When $\beta \leq 0.919431$ we reduce the validity of the full conjecture to establishing that the half-plane $\{ x \in [0,\beta] \times [0,1]^2 \; ; \; x_3 \leq \frac{1}{\pi} \}$ is an isoperimetric minimizer. We also show that the analogous conjecture on a high-dimensional cube $[0,1]^n$ is false for $n \geq 10$. 

In the case of a slab with a Gaussian base of width $T>0$, we identify a phase transition when $T = \sqrt{2 \pi}$ and when $T = \pi$. In particular, while products of half-planes with $[0,T]$ are always minimizing when $T \leq \sqrt{2 \pi}$, when $T > \pi$ they are never minimizing, being beaten by Gaussian unduloids. In the range $T \in (\sqrt{2 \pi},\pi]$, a potential trichotomy occurs. 
\end{abstract}

\section{Introduction}

Let $(\bM^{n-1},\bg)$ denote an $(n-1)$-dimensional smooth Riemannian manifold, endowed with a probability measure $\bmu$ having nice positive density with respect to the Riemannian volume measure $\vol_{\bg}$. A weighted Riemannian slab (or simply ``slab") of width $T > 0$ is an $n$-dimensional weighted Riemannian manifold of the form:
\[
 (M^n_T,g,\mu_T) :=  ([0,T],\abs{\cdot}^2,\frac{1}{T} \m\mycorner_{[0,T]}) \otimes (\bM^{n-1}, \bg, \bmu),
\]
where $M^n_T :=  [0,T] \times \bM^{n-1}$, $g$ is the Riemannian product metric, and $\mu_T := \frac{1}{T} \m\mycorner_{[0,T]} \otimes \bmu$ is the product probability measure. The weighted Riemannian manifold $(\bM^{n-1},\bg,\bmu)$ is called the (vertical) ``base" of the slab. Throughout this work we denote the Euclidean metric on $\R^n$ by $\abs{\cdot}^2$ and corresponding Lebesgue measure by $\m$.

Given a Borel subset $E \subset (M^n,g,\mu)$, we denote its (weighted) volume by:
\[
V(E) = V_{\mu}(E)  := \mu(E) , 
\]
and if $E$ is of locally finite perimeter, we denote its (weighted) perimeter by:
\[
A(E) = A_{\mu}(E) := \int_{\partial^* E \cap \interior M} \Psi_{\mu} d\H^{n-1} 
\]
where $\Psi_{\mu}$ denotes the density of $\mu$, $\interior M = M \setminus \partial M$ denotes the interior of $M$, $\partial^* E$ is the reduced boundary of $E$ and $\H^k$ denotes the $k$-dimensional Hausdorff measure (see Section \ref{sec:CMC} for more information). In this work, we are interested in obtaining sharp isoperimetric inequalities on slabs $(M_T,g,\mu_T)$, namely, best-possible lower bounds on the perimeter $A(E)$ for all subsets $E$ with prescribed volume $V(E) = \bar v \in (0,1)$. 

\smallskip

Many functional and concentration inequalities (such as the Poincar\'e and log-Sobolev inequalities) are known to tensorize, and so their analysis on a slab reduces to understanding these inequalities on the base. However, this is not the case with isoperimetric inequalities, for which there is no general formula yielding their sharp form after tensorization (even though up to constants, one may give \emph{essentially} sharp lower bounds, see e.g. \cite{Morgan-Products}). Note that the flat Lebesgue measure on the horizontal factor $[0,T]$ ruins any positive curvature possibly enjoyed by the base space, thereby precluding any easy ``strict convexity" arguments. 
Consequently, we are interested in developing a general framework for obtaining sharp isoperimetric inequalities on product spaces, in the simplest case when the horizontal factor is just the uniform Lebesgue measure on $[0,T]$. 

\smallskip

For various specific base spaces, such as $\R^{n-1}$, $\S^{n-1}$, $\mathbb{H}^{n-1}$, $[0,1]$ or $[0,1] \times \R$, endowed with their natural Haar measures, such a study has been undertaken in prior work \cite{BrezisBruckstein,HPRR-PeriodicIsoperimetricProblem,Howards-BScThesis,HHM-Surfaces,KoisoMiyamoto,Pedrosa-SphericalCylinders,PedrosaRitore-Products}, 
which we shall utilize and expand upon. 
Another variant is when the horizontal factor is not an interval but rather $(\R^k,\abs{\cdot}^2,\m)$: it was shown in \cite{Castro-Products,Gonzalo-Products,RitoreVernadakis-Products} that when $M^{n-1}$ is compact, a minimizer of large volume is always of the form $B^k \times M^{n-1}$ for some Euclidean ball $B^k \subset \R^k$, and in \cite{FuscoMaggiPratelli-GaussianProduct} the case where the base space is the standard Gaussian one $(\R^{n-1},\abs{\cdot}^2,\gamma^{n-1})$ was studied; however, the infinite mass of the horizontal factor leads to somewhat different phenomenology than the one we observe in this work.

\subsection{The main idea}

Recall that the isoperimetric profile $\I : [0,1] \rightarrow \R_+$ of a weighted Riemannian manifold $(M,g,\mu)$ is defined as
\[
 \I(\bar v) := \inf \{ A_\mu(E) \; ; \; V_\mu(E) = \bar v\} .
\]
We denote by $\I_0$ the isoperimetric profile of $(\bM^{n-1}, \bg, \bmu)$, and by $\I_T$ the isoperimetric profile of $(M^n_T,g,\mu_T)$. It can be shown that
\begin{equation} \label{eq:intro-lower-bound}
\I_T(\bar v) \geq \I^b_T(\bar v) \;\;\; \forall \bar v \in [0,1] ,
\end{equation}
where the base-induced isoperimetric profile $\I^b_T$ is defined as
\[
\I^b_T(\bar v) := \inf \set{ \frac{1}{T}  \brac{\int_0^T \sqrt{ v'(t)^2 + \bI(v(t))^2 } dt + \norm{D_S v}} \; ; \; \frac{1}{T} \int_0^T v(t) dt = \bar v } .
\]
Here the infimum is taken over all functions $v : [0,T] \rightarrow [0,1]$ of bounded variation, 
and $D_S v$ denotes the singular part of the distributional derivative of $v$ (see Proposition \ref{prop:area-formula}, Definition \ref{def:base-profile} and Corollary \ref{cor:base-induced-profile}). 

\smallskip

Instead of approaching the computation of $\I_T^b$ as a problem in the calculus of variations, our idea is to interpret it as the isoperimetric profile of a two-dimensional \emph{model slab}. Indeed, when $\I_0$ is concave (and symmetric), it is known that it coincides with the isoperimetric profile of a (uniquely determined) even log-concave density $\varphi_{\I_0}$ on an interval $(-R_{\I_0},R_{\I_0})$. The corresponding two-dimensional model slab is defined as:
\[
S_T(\I_0) := ([0,T],\abs{\cdot}^2,\frac{1}{T} \m\mycorner_{[0,T]}) \otimes ((-R_{\bI},R_{\bI}), \abs{\cdot}^2 , \varphi_{\bI}(s) ds) .
\]
It is then not hard to see that $\I^b_T = \I_{S_T(\I_0)}$ (see Definition \ref{def:model-slab} and Proposition \ref{prop:coincide}). 

\smallskip

Sections \ref{sec:CMC} and \ref{sec:stability} are devoted to the understanding of the isoperimetric profile of model slabs $\I_{S_T(\I_0)}$
by means of various tools, including symmetrization, the constant mean-curvature equation, stability analysis and differential estimates for the isoperimetric profile. We thus obtain a lower bound on the actual isoperimetric profile $\I_T$ via (\ref{eq:intro-lower-bound}), which can seen as a certain reduction to the two-dimensional case by means of symmetrization. Our main observation is that this reduction may sometimes yield sharp results, if the easy upper bound on $\I_T$ (obtained by testing a candidate isoperimetric set) coincides with the aforementioned lower bound. 

\smallskip

We apply our general framework to two cases which are of particular interest:
\begin{enumerate}
\item When the base is the square $[0,1]^2$ endowed with its uniform measure, yielding the cubical slab $[0,T] \times [0,1]^2$. By a well-known reflection argument, the study of isoperimetric minimizers on this cube is equivalent to that on the flat torus $\R^3 / ((2T) \Z \times 2 \Z^2)$. 
\item When the base is the standard Gaussian space $(\R^{n-1},\abs{\cdot}^2,\gamma^{n-1})$, yielding the Gaussian slab corresponding to the measure $\frac{1}{T} \m\mycorner_{[0,T]} \otimes \gamma^{n-1}$. 
\end{enumerate}

In both cases, we obtain several new isoperimetric results as described next, some of which are perhaps surprising.

\subsection{Three dimensional cube}

Given $\beta \in (0,1]$, let $\Q^3(\beta) := ([0,\beta] \times [0,1]^2, \abs{\cdot}^2, \frac{1}{\beta} \m\mycorner_{ [0,\beta] \times [0,1]^2})$ denote the $3$-dimensional cube with side lengths $(\beta,1,1)$, endowed with its uniform measure. Note that the shortest edge is of length $\beta$. The most natural case is when $\beta=1$, in which case we set $\Q^3 := \Q^3(1)$. The following conjecture is widely believed, and quoting A.~Ros \cite{RosIsoperimetricProblemNotes}, ``is one of the nicest open problems in classical geometry":

\begin{conjecture}[Isoperimetric conjecture on $3$-dimensional cube]  \label{conj:Q3}
For every $\bar v \in (0,1)$, (at least) one of the following sets or its complement is an isoperimetric minimizer in $\Q^3(\beta)$ of relative volume $\bar v$:
\begin{enumerate}
\item An eighth ball around a vertex, $\{ x \in \Q^3(\beta) \; ; \;  x_1^2 + x_2^2 + x_3^2 \leq r^2 \}$. 
\item A quarter cylinder around the short edge $[0,\beta]$,  $\{ x \in \Q^3(\beta) \; ; \;  x_2^2 + x_3^2 \leq r^2 \}$. 
\item A half plane, $\{ x \in \Q^3(\beta) \; ; \; x_3 \leq \bar v \}$. 
\end{enumerate}
\end{conjecture}

It is sometimes more convenient to specify the enclosing boundary of the minimizers, given by truncated spheres, cylinders and flat planes. A stronger variant of the conjecture asserts that these are the only possible minimizers (up to isometries of $\Q^3(\beta)$ and null-sets), but we will not insist on uniqueness here. By taking complements, it is enough to establish the conjecture for $\bar v \in (0,1/2]$.
As already mentioned, by a well-known reflection argument, the conjecture on $\Q^3(\beta)$ is equivalent to the analogous one on the flat torus $\T^3(\beta) = \R^3/ (2\beta \Z \times 2 \Z^2)$ (endowed with its uniform measure)  -- see Remark \ref{rem:reflect-slab}. On $\T^3(\beta)$, the conjecture is that the minimizers are enclosed by spheres, cylinders about a closed geodesic of length $2\beta$ and parallel pairs of totally geodesic tori $\T^2(\beta) = \R^2 / (2 \beta \Z \times 2 \Z)$.  
 \smallskip
 
 The easier doubly periodic case on  $\T^2(\beta) \times \R$, or equivalently, on $[0,\beta] \times [0,1] \times \R$ ($\beta \in (0,1]$), has been studied by Ritor\'e-Ros \cite{RitoreRos-CompactnessOfStableCMC} and Hauswirth--P\'erez--Romon--Ros \cite{HPRR-PeriodicIsoperimetricProblem} (in fact, for general two-dimensional flat tori bases), who fully established the corresponding conjecture including uniqueness of minimizers for small enough $\beta$ and $\beta \leq \frac{16}{9 \pi} \simeq 0.565884$, respectively. It was also shown in \cite{HPRR-PeriodicIsoperimetricProblem} that spheres are indeed minimizing in $\T^2(\beta) \times \R$ for all $\beta \in (0,1]$ in precisely the conjectured range. Furthermore, the authors of \cite{HPRR-PeriodicIsoperimetricProblem} studied $G$-invariant minimizers in $\T^3 = \R^3 / 2\mathbb{Z}^3$, where $G$ is any finite group of isometries fixing the origin and containing the diagonal rotations through the origin; in particular, they showed that $G$-symmetric surfaces in $\T^3$ besides spheres cannot be the actual minimizers for the non-symmetric problem. The single periodic case on $\T^1 \times \R^2$, or equivalently, on $[0,1] \times \R^2$, may be fully understood using the methods of Ritor\'e--Ros \cite{RitoreRos-RP3}, 
and the case of $\T^1 \times \R^{n-1}$ was resolved by Pedrosa--Ritor\'e \cite{PedrosaRitore-Products} for all $n \leq 8$ -- in that case, minimizers are necessarily enclosed by spheres or cylinders.

Back to the triply periodic setting on $\T^3(\beta)$, or equivalently, $\Q^3(\beta)$. For small enough volumes $\bar v \in (0,\eps_s]$, it follows from Morgan \cite[Remark 3.11]{Morgan-Polytopes} that balls around a vertex in $\Q^3(\beta)$ are indeed isoperimetric minimizers (alternatively, apply Morgan-Johnson \cite[Theorem 4.4]{MorganJohnson} on $\T^3(\beta)$). It was shown by Hadwiger \cite{HadwigerCube} 
and subsequently by Barthe and Maurey \cite{BartheMaureyIsoperimetricInqs} using a Gaussian contraction argument 
that the half-plane $\{x \in \Q^3(\beta) \; ; \; x_3 \leq 1/2\}$ is an isoperimetric minimizer of volume $1/2$, establishing the conjecture for $\bar v = 1/2$  (this was shown for the case $\beta=1$ but the proof carries through for all $\beta \in (0,1]$). Moreover, it was shown by Acerbi--Fusco--Morini \cite[Theorem 5.3]{AcerbiFuscoMorini} (see also \cite[Theorem 1.1]{Glaudo-IsoperimetryOnCube}) that there exists $\eps_p > 0$ so that for all $\bar v \in [1/2-\eps_p,1/2+\eps_p]$, the half-plane $\{ x \in \Q^3(\beta) \; ; \; x_3 \leq \bar v \}$ is a minimizer of relative volume $\bar v$, confirming the conjecture in that range as well (again, this was shown for the case $\beta=1$, but the contraction argument extends this to all $\beta \in (0,1]$). For a variant allowing only sets whose boundary lies in the union of a finite number of coordinate hyperplanes, see \cite{ChambersEtAl-IsoperimetryOnCube}.
To the best of our knowledge, Conjecture \ref{conj:Q3} has remained open in the range $(\eps_s,1/2-\eps_p)$, and there are no effective bounds on $\eps_s,\eps_p>0$, as they are obtained by compactness arguments. 

\smallskip

An elementary computation of the relative volumes and surface areas of the three types of conjectured minimizers shows that a minimizer on $\Q^3(\beta)$ is expected to be enclosed by a sphere for $\bar v \in (0, \frac{4 \pi}{81} \beta^2]$, a cylinder for $\bar v \in [\frac{4 \pi}{81} \beta^2, \frac{1}{\pi}]$, and a flat plane for $\bar v \in [\frac{1}{\pi},\frac{1}{2}]$. With this in mind, we can state our first main result regarding $\Q^3(\beta)$ as follows:

\begin{theorem}  \label{thm:Q3-main}
On the $3$-dimensional cube $\Q^3(\beta)$ with edge lengths $(\beta,1,1)$, $\beta \in (0,1]$, the following holds:
\begin{enumerate}
\item \label{it:Q3-1}
For all $\bar v \in (0,\min(\frac{4 \pi}{81} \beta^2, \frac{v_{\min}}{\beta})]$ where $v_{\min} \simeq 0.120582$ is an explicit constant (computed in Proposition \ref{prop:one-sided-volume}), spheres about a corner are minimizing. \\
When $\beta \leq \brac{\frac{81 v_{\min}}{4 \pi}}^{1/3} \simeq 0.919431$, this confirms the conjecture for the entire range $\bar v \in (0,\frac{4 \pi}{81} \beta^2]$ where spheres are expected to be minimizing. 
\item \label{it:Q3-2}
For all $\bar v \in [\frac{4 \pi}{81} \beta^2, \max( \frac{1}{\pi} - \frac{\beta}{4} , \min( \frac{4 \pi}{81} , \frac{v_{\min}}{\beta}))]$ (this interval is non-empty when $\beta \leq 0.919431$), cylinders about the short edge $[0,\beta]$ are minimizing (see Figure \ref{fig:cylinder-interval}).
\item \label{it:Q3-3}
For all $\bar v \in [\frac{1}{\pi} + \frac{\beta}{4} , \frac{1}{2}]$ (this interval is non-empty when $\beta \leq  2 - \frac{4}{\pi} \simeq 0.72676$), flat planes $[0,\beta]\times [0,1] \times \{\bar v\}$ are minimizing.
\end{enumerate}
In particular, for all $\beta \in (0,1]$, Conjecture \ref{conj:Q3} holds for all $\bar v \in (0,1/2] \setminus (\frac{1}{\pi} - \frac{\beta}{4} , \frac{1}{\pi} + \frac{\beta}{4})$.  
\end{theorem}

\begin{figure}
\begin{center}
        \includegraphics[scale=0.5]{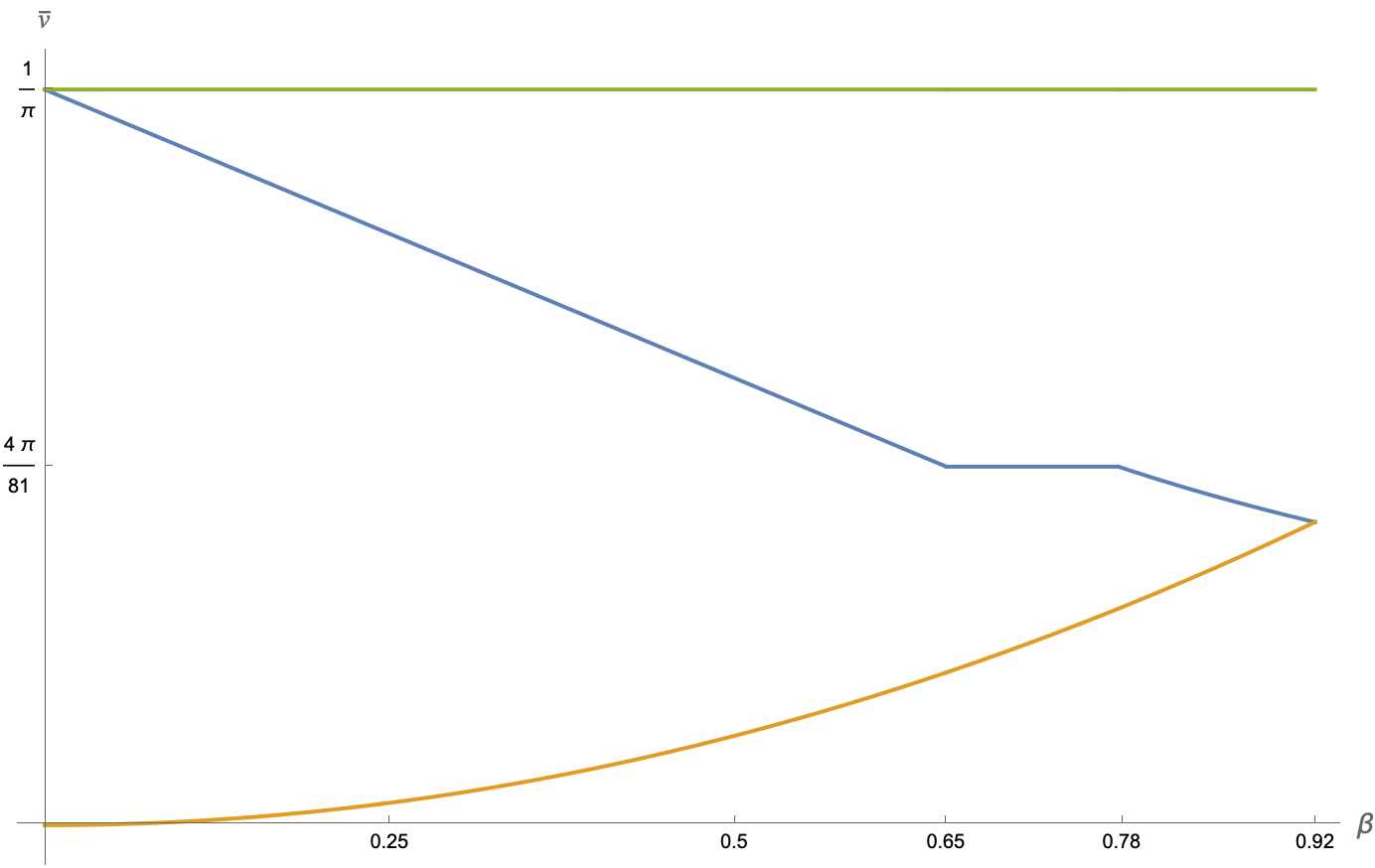}
     \end{center}
     \caption{
         \label{fig:cylinder-interval}
         Cylinders in $\Q^3(\beta)$ are conjectured to be minimizers for (weighted) volumes $\bar v$ between the yellow and green plots. We are able to show they are minimizers between the yellow and blue plots. 
     }
\end{figure}

Our proof yields a slightly smaller ``uncertainty" interval than $(\frac{1}{\pi} - \frac{\beta}{4} , \frac{1}{\pi} + \frac{\beta}{4})$ above, but we chose to state the simplest formulation. The above results suggest that $\bar v = \frac{1}{\pi}$ is the hardest case of Conjecture \ref{conj:Q3}. Indeed, according to Ros \cite{RosIsoperimetricProblemNotes}, there is a Lawson genus-2 surface in $\Q^3$ enclosing volume $\frac{1}{\pi}$ and having surface area $1.017$, just a little over the conjectured $1$. 
This is further corroborated by the following observation, which is implicitly contained in the work of Hauswirth--P\'erez--Romon--Ros \cite{HPRR-PeriodicIsoperimetricProblem}; it will be more convenient to formulate it on $\T^3(\beta)$. 

\begin{theorem}  \label{thm:T3-ODE}
On the $3$-dimensional torus $\T^3(\beta) = \R^3 / (2\beta \Z \times 2 \Z^2)$, $\beta \in (0,1]$, endowed with its uniform probability measure, the following holds:
\begin{enumerate}
\item \label{it:T3-ODE-1}
There exist $0 < v_s \leq \frac{4\pi}{81} \beta^2 \leq v_{c-} , v_{c+} \leq \frac{1}{\pi} \leq v_p < \frac{1}{2}$ so that if $E$ is an isoperimetric minimizer in $\T^3(\beta)$ of relative volume $\bar v \in (0,\frac{1}{2}]$ and $\Sigma = \overline{\partial^* E}$ then:
\begin{enumerate}
\item $\bar v \in (0,v_s) \Rightarrow $ $\Sigma$ is a sphere $\Rightarrow \bar v \in (0,v_s]$. 
\item $\bar v \in (v_{c-},v_{c+}) \Rightarrow$ $\Sigma$ is a cylinder about a shortest closed geodesic $\Rightarrow \bar v \in [v_{c-},v_{c+}]$.
\item $\bar v \in (v_p, \frac{1}{2}] \Rightarrow$ $\Sigma$ is the disjoint union of two parallel totally geodesic tori $\T^2(\beta)$ $\Rightarrow \bar v \in [v_p,\frac{1}{2}]$.
\end{enumerate}
\item \label{it:T3-ODE-2}
In particular, if Conjecture \ref{conj:Q3} holds at $\bar v = \frac{1}{\pi}$ and $\bar v = \frac{4 \pi}{81} \beta^2$, then it holds for all $\bar v \in (0,1)$. 
\end{enumerate}
\end{theorem}

Note that there is no guarantee that $v_{c-} < v_{c+}$, so that the interval $(v_{c-},v_{c+})$ may be empty. In any case, in view of Theorem \ref{thm:Q3-main}, we immediately deduce:
\begin{corollary}
When $\beta \leq 0.919431$, $v_s = v_{c-} = \frac{4\pi}{81} \beta^2$, and Conjecture \ref{conj:Q3} for $\Q^3(\beta)$ holds in its entirety iff it holds at $\bar v = \frac{1}{\pi}$, namely iff the half-plane $\{ x \in \Q^3(\beta) \; ; \; x_3 \leq \frac{1}{\pi} \}$ is an isoperimetric minimizer. 
\end{corollary}

In our opinion, this last observation is of particular interest, since it reduces the task of establishing the conjecture when $\beta \leq 0.919431$ to a sharp extension of the range where a half-plane is known to be minimizing, which in some sense is more of a ``linear" problem. Cases (\ref{it:Q3-1}) and (\ref{it:Q3-2}) of  Theorem \ref{thm:Q3-main} when $\bar v \leq \min( \frac{4 \pi}{81},\frac{v_{\min}}{\beta})$  are established in Subsection \ref{subsec:Q3.2}, whereas cases (\ref{it:Q3-2}) and (\ref{it:Q3-3}) in the remaining range, as well as Theorem \ref{thm:T3-ODE}, are established in Subsection \ref{subsec:Q3.3}. Some additional results are described in Section \ref{sec:Q3}. The proof involves a rigorous numerical estimation of a certain function involving elliptic integrals.

\subsection{High-dimensional cube}

Let $\Q^n := ([0,1]^n, \abs{\cdot}^2, \m\mycorner_{[0,1]^n})$ denote the $n$-dimensional unit cube endowed with its uniform measure. In view of the isoperimetric conjecture on $\Q^3$, it is natural to make the following:

\begin{conjecture}[Isoperimetric conjecture on $n$-dimensional cube]  \label{conj:Qn}
For every $\bar v \in (0,1)$, there exists $k \in \{0,1,\ldots,n-1\}$ and $r \in (0,1]$ so that the $r$-tubular neighborhood of a $k$-dimensional face of $\Q^n$ or its complement is an isoperimetric minimizer of volume $\bar v$. 
\end{conjecture}

In the two-dimensional case of $\Q^2$, or equivalently, for the analogous conjecture on the flat torus $\T^2 = \R^2 / (2 \Z^2)$, this is well-known (e.g.~\cite{BrezisBruckstein}, \cite[Theorem 3.1]{Howards-BScThesis}, \cite[Section 7]{HHM-Surfaces}, \cite[Section 1.5]{RosIsoperimetricProblemNotes}), but the general $n \geq 3$ case poses a much greater challenge. As in the case $n=3$, the conjecture is known to hold when $\min(\bar v , 1 - \bar v) \in (0,\eps_s(n)] \cup [1/2-\eps_p(n),1/2]$: small balls around corners of $\Q^n$ are minimizers, as are half-planes $\{ x \in \Q^n \; ; \; x_n \leq \bar v \}$ when $\bar v \in [1/2-\eps_p(n),1/2+\eps_p(n)]$. See also \cite[Theorem 1.2]{Glaudo-IsoperimetryOnCube} for an interesting dimension-independent lower bound on the isoperimetric profile of $\Q^n$ which is strictly better than the one obtained from the Gaussian contraction argument. 

\medskip

However, in Section \ref{sec:high-dim-cube} we observe that the above conjecture cannot be true in full generality in high-dimension:

\begin{theorem} \label{thm:Qn}
The isoperimetric conjecture on the $n$-dimensional cube is false for all $n \geq 10$. 
\end{theorem}

This confirms a prediction of Ros \cite{RosIsoperimetricProblemNotes}, who writes regarding the isoperimetric conjecture on $\Q^3$: ``In higher dimensions the corresponding conjecture is probably wrong". 
The argument for demonstrating that the conjecture is false is the same as the one used by Pedrosa--Ritor\'e in \cite{PedrosaRitore-Products} for showing that cylinders and half-balls fail to be isoperimetric minimizers in a slab $\R^{n-1} \times [0,1]$ for certain volumes and large $n$ (in fact, in precisely the same range $n \geq 10$). The only difference is that our computation is somewhat heavier since we need to disqualify the tubular neighborhoods of all $k$-dimensional faces. 
\smallskip
The falsehood of the conjecture in general dimension suggests that obtaining a positive answer in low-dimension would involve fortunate numeric coincidences. And indeed, as already mentioned, our progress in dimension $n=3$ is based on several numerical computations. 

\subsection{Gaussian slabs}

Before concluding this work, we apply our general slab framework to a second natural base space given by $(\R^{n-1} , \abs{\cdot}^2, \gamma^{n-1})$, where $\gamma^{n-1}$ denotes the standard Gaussian measure on $\R^{n-1}$. Let $\G^n_T$ denote the slab of width $T > 0$ over the latter Gaussian base space, namely:
\[
\G^n_T := ([0,T],\abs{\cdot}^2,\frac{1}{T} \m\mycorner_{[0,T]})  \otimes (\R^{n-1} , \abs{\cdot}^2, \gamma^{n-1}) .
\]

We denote by $\I_T$ the isoperimetric profile of $\G^n_T$. 
 It was shown by Sudakov--Tsirelson \cite{SudakovTsirelson} and independently Borell \cite{Borell-GaussianIsoperimetry} that half-planes are isoperimetric minimizers in $(\R^{n-1} , \abs{\cdot}^2, \gamma^{n-1})$; thanks to the product structure of the Gaussian measure, it follows that its isoperimetric profile $\I_{\gamma^{n-1}}$ is dimension-independent, coinciding with the profile of the one-dimensional $(\R,\abs{\cdot}^2,\gamma^1)$, which we denote by $\I_\gamma$. Since half-lines are minimizers in $(\R,\abs{\cdot}^2,\gamma^1)$, we have $\I_\gamma = \varphi_{\gamma} \circ \Phi_{\gamma}^{-1}$, where $\varphi_{\gamma}$ denotes the standard Gaussian density on $\R$ and $\Phi_{\gamma}(s) = \int_{-\infty}^s \varphi_\gamma(x) dx$. It is easy to check that $\I_{\gamma} \I_{\gamma}'' = -1$ on $(0,1)$. In particular, $\I_\gamma : [0,1] \rightarrow \R_+$ is concave and symmetric about $1/2$, and general results (see Proposition \ref{prop:I-properties}) imply that the same holds for $\I_T$. 

\medskip

The following is trivial:
\begin{lemma}
For all $T \leq \sqrt{2 \pi}$, $\I_T = \I_\gamma$. In other words, horizontal half-planes $\{ x_n \leq \Phi_\gamma^{-1}(\bar v) \}$ are isoperimetric minimizers of (weighted) volume $\bar v$ in $\G^n_T$ for all $\bar v \in (0,1)$. 
\end{lemma}
\begin{proof}
Since $\max_{x \in \R} \varphi_{\gamma}(x) = 1/\sqrt{2 \pi}$, it is immediately seen that the horizontal measure $\frac{1}{T} \m\mycorner_{[0,T]}$ is the push-forward of $\gamma^1$ via the map $P_1 = T \cdot \Phi_{\gamma}$ having Lipschitz constant $L = T/\sqrt{2 \pi}$. Defining $P(x) = (P_1(x_1), x_2,\ldots,x_n)$, it follows that $\G^n_T$ is the push-forward of $(\R^n, \abs{\cdot}^2, \gamma^n)$ via the map $P$, which is $\max(1,L)$-Lipschitz. Consequently, a standard transference principle of isoperimetry under Lipschitz maps (see e.g. \cite[Section 5.3]{EMilman-RoleOfConvexity})
 implies that $\I_T \geq \frac{1}{\max(1,L)} \I_{\gamma^n} =  \frac{1}{\max(1,L)} \I_{\gamma}$. On the other hand, by testing horizontal half-planes, we clearly have $\I_T \leq \I_\gamma$. The assertion when $T \leq \sqrt{2 \pi}$ now follows since in that case $\max(1,L) = 1$.
\end{proof}

However, it is already unclear what to expect when $T > \sqrt{2 \pi}$. The contraction argument above shows that $\I_T \geq \frac{\sqrt{2\pi}}{T} \I_{\gamma}$ when $T \geq \sqrt{2 \pi}$. On the other hand, by inspecting horizontal and vertical half-spaces in $\G^n_T$, we clearly have $\I_T \leq \min(\I_\gamma,\frac{1}{T})$. Since  $I_\gamma(1/2) = \varphi_{\gamma}(0) = 1/\sqrt{2 \pi}$, we conclude that $\I_T(1/2) = \frac{1}{T}$ for all $T \geq \sqrt{2 \pi}$. A natural question is thus whether $\I_T = \min(\I_\gamma,\frac{1}{T})$ not only when $T \in (0,\sqrt{2 \pi}]$ but also beyond. We will see that this is \textbf{not} the case, at least when $T > \pi$.

\medskip

A variant of the above setting was studied by Fusco--Maggi--Pratelli \cite{FuscoMaggiPratelli-GaussianProduct}, who considered the isoperimetric problem on the product space $(\R^k , \abs{\cdot}^2, \m) \otimes (\R^{n-1} , \abs{\cdot}^2, \gamma^{n-1})$. Specializing to the case $k=1$, note that the horizontal factor is of infinite mass and two-sided, but we may use reflection as in Remark \ref{rem:reflect-slab} to translate to the case that the horizontal factor is the one-sided $([0,\infty),\abs{\cdot}^2, \m)$. With this in mind, an equivalent reformulation of a result from \cite{FuscoMaggiPratelli-GaussianProduct} is that there exists a critical mass $v_m > 0$ so that when $\bar v > v_m$, a minimizer is a vertical half-plane $\{x_1 \leq \bar v\}$, whereas when $\bar v \in (0, v_m)$, a minimizer will be a certain ``one-sided Gaussian unduloid", given (up to vertical rotation) by $\{ x_1 \leq \tau(x_n) \}$ where $\tau : (-\infty,f_1] \rightarrow [0,\infty)$ is an explicit function strictly decreasing from $\infty$ to $\tau(f_1) = 0$ for some $f_1 \in \R$.
In our finite-mass setting, when $T < \infty$, we observe somewhat different phenomenology -- ``one-sided unduloids" (meeting only one side of the slab) can never be minimizing, but ``Gaussian unduloids" (meeting both sides of the slab) corresponding to strictly decreasing $\tau : [f_0,f_1] \rightarrow [0,T]$ for some finite $f_1 = \tau^{-1}(0), f_0 = \tau^{-1}(T) \in \R$ will necessarily occur as soon as $T > \pi$. What happens in the intermediate range $T \in (\sqrt{2 \pi} , \pi]$ is not clear to us and remains an interesting avenue of investigation. As usual, we do not pursue the question of uniqueness of minimizers, and restrict (by taking complements) to the range $\bar v \in (0,1/2]$. We show: 
\begin{theorem} \label{thm:main-Gn}
For all $T > \sqrt{2 \pi}$, there exist $v_v \in (0,\frac{\sqrt{2 \pi}}{2 T}] \subset (0,1/2)$ and $v_h \in [0,v_v]$ (both depending on $T$) so that:
\begin{enumerate}
\item The following trichotomy holds on $\G^n_T$ for any $n \geq 2$:
\begin{enumerate}
\item For all $\bar v \in (0,v_h]$, a horizontal half-plane $\{ x_n \leq \Phi_\gamma^{-1}(\bar v) \}$ of (weighted) volume $\bar v$ is minimizing and $\I_T(\bar v) = \I_\gamma(\bar v)$. In particular, $\I_T \I_T'' = -1$ on $(0,v_h)$.
\item For all $\bar v \in (v_h, v_v)$, a Gaussian unduloid $\{ x_1 \leq \tau(x_n) \}$ of (weighted) volume $\bar v$ is minimizing, $\I_T \I_T'' < -1$ in the viscosity sense, and $\I_T(\bar v) < \min(\I_{\gamma}(\bar v) , \frac{1}{T})$. Here $\tau : [f_0 , f_1] \rightarrow [0,T]$ is the strictly decreasing function given by:
\[
\tau(f) = \int_{f}^{f_1} \frac{ds}{\sqrt{ (\I_\gamma/\ell)^2(\Phi_{\gamma}(s)) - 1}} 
\]
for appropriate (finite) parameters $f_0 < f_1$ (depending on $\bar v$), where $\ell(v) = \frac{v_1 - v}{v_1 - v_0} \I_\gamma(v_0) + \frac{v-v_0}{v_1-v_0} \I_\gamma(v_1)$ is the chord of the graph of the concave $I_\gamma$ between $v_0 = \Phi_\gamma(f_0)$ and $v_1 = \Phi_{\gamma}(f_1)$. See Figure \ref{fig:Gaussian-unduloids}.
\item For all $\bar v \in [v_v,1/2]$, a vertical half-plane $\{ x_1 \leq T \bar v \}$ of (weighted) volume $\bar v$ is minimizing and $I_T(\bar v) = \frac{1}{T}$. 
\end{enumerate}
Note that while the interval $(0, v_v)$ is always non-empty, the individual intervals $(0,v_h]$ or $(v_h, v_v)$ may be empty. Also note that the interval $[v_v,1/2]$ always has strictly positive length of at least $\frac{1}{2} ( 1 - \frac{\sqrt{2 \pi}}{T})$. 
\item Both parameters $v_h, v_v$ are non-increasing in $T$. \item If $T > \pi$ then no horizontal half-planes are ever minimizing, and so $\I_T(\bar v) < \I_\gamma(\bar v)$ for all $\bar v \in (0,1)$. In particular:
\begin{enumerate}
\item $v_v > I_\gamma^{-1}(\frac{1}{T})$ (inverse taken in $(0,1/2]$), so $v_v \in (I_\gamma^{-1}(\frac{1}{T}) , \frac{\sqrt{2 \pi}}{2 T}]$.
\item $v_h = 0$ and Gaussian unduloids are minimizing for all $\bar v \in (0,v_v)$. 
\end{enumerate}
\end{enumerate}
\end{theorem}

Additional information regarding $\I_T$ is described in Section \ref{sec:Gn}. In our opinion, it would be interesting to further understand the behavior in the regime when $T \in (\sqrt{2\pi},\pi]$, as well as to determine the precise value of $v_v(T) = \min \{ \bar v \in (0,1/2] \; ; \; \I_T(\bar v) = \frac{1}{T} \}$. Our results imply in particular that
  for any fixed $\bar v \in (0,1)$, for all $T \geq \frac{\sqrt{2 \pi}}{2 \min(\bar v,1-\bar v)}$ and $n \geq 2$, a vertical half-plane  $[0,T \bar v] \times \R^{n-1}$ is a minimizer in $\G_T^n$, thereby confirming a conjecture of Hutchings \cite[Conjecture 3.13]{Morgan-Polytopes} for this particular setting.

\medskip

The rest of this work is organized as follows. In Sections \ref{sec:CMC} and \ref{sec:stability} we collect known results and further develop a general framework for obtaining isoperimetric inequalities on slabs. This is then applied to $\Q^3(\beta)$ and $\G^n_T$ in Sections \ref{sec:Q3} and \ref{sec:Gn}, respectively. A particularly heavy numerical computation is deferred to the Appendix. A refutation of Conjecture \ref{conj:Qn} on high-dimensional cubes is provided in Section \ref{sec:high-dim-cube}, which may be essentially read independently from the rest of this work. 

\medskip

\noindent
\textbf{Acknowledgments.} I thank Frank ``Chip" Morgan for his comments regarding a preliminary version of this work. I also thank the anonymous referees for their detailed reading of the manuscript and helpful comments.

\section{Isoperimetry on Slabs} \label{sec:CMC}

Recall we denote the Euclidean metric on $\R^n$ by $\abs{\cdot}^2$ and corresponding Lebesgue measure by $\m$. 
\begin{definition}[Weighted Riemannian Manifold] A triplet $(M^n,g,\mu)$ is called a weighted Riemannian manifold, if $(M^n,g)$ is a smooth $n$-dimensional connected Riemannian manifold-with-boundary $\partial M^{n-1}$ (possibly empty), and $\mu$ is a Borel probability measure on $(M^n,g)$ having density $\Psi_{\mu}$ with respect to the corresponding Riemannian volume measure $\vol_g$. $(M^n,g)$ is assumed to be either complete, or a (possibly incomplete) convex subset of Euclidean space $(\R^n,\abs{\cdot}^2)$. 
$\Psi_{\mu}$ is assumed positive, locally Lipschitz on $M^n$ and bounded on every geodesic ball. \end{definition}
\noindent Note that we shall only consider the case when $\mu$ is a probability measure in this work. 
Our applications require us to handle densities which are only locally Lipschitz regular. A property is called local if it holds on every compact subset. Note that a closed geodesic ball need not be compact due to incompleteness of $M^n \subset (\R^n,\abs{\cdot}^2)$. It should be possible to extend our setup to include incomplete manifolds having bounded geometry on every geodesic ball, but we refrain from this generality here.

\begin{definition}[Slab]
A Weighted Riemannian Slab (or simply ``slab") of width $T > 0$ is an $n$-dimensional weighted Riemannian manifold of the form:
\[
 (M^n_T,g,\mu_T) :=  ([0,T],\abs{\cdot}^2,\frac{1}{T} \m\mycorner_{[0,T]}) \otimes (\bM^{n-1}, \bg, \bmu),
\]
where $M^n_T :=  [0,T] \times \bM^{n-1}$, $g$ is the Riemannian product metric, and $\mu_T := \frac{1}{T} \m\mycorner_{[0,T]} \otimes \bmu$ is the product probability measure. The $(n-1)$-dimensional weighted Riemannian manifold $(\bM,\bg,\bmu)$ is assumed to be without boundary, and is called the (vertical) ``base" of the slab. \end{definition}

Given a Borel set $E \subset \R^n$ with locally-finite perimeter, its reduced boundary $\partial^* E$ is defined as the Borel subset of $\partial E$ for which there is a uniquely defined outer unit normal vector $\n_E$ to $E$ in a measure theoretic sense (see \cite[Chapter 15]{MaggiBook} for a precise definition). 
The definition of reduced boundary canonically extends to the Riemannian setting by using a local chart, as it is known that $T(\partial^* E) = \partial^* T(E)$ for any smooth diffeomorphism $T$ (see \cite[Lemma A.1]{KMS-LimitOfCapillarity}). It is known that $\partial^* E$ is a Borel subset of $\partial E$, and that modifying $E$ on a null-set does not alter $\partial^* E$. If $E$ is an open set with $C^1$ smooth boundary, it holds that $\partial^* E = \partial E$  (e.g. \cite[Remark 15.1]{MaggiBook}). 
See \cite{MaggiBook} and \cite[Section 2]{BarchiesiCagnettiFusco} for additional background on sets of finite perimeter. 

Given a Borel subset $E \subset (M^n,g,\mu)$, we denote its $\mu$-weighted volume by:
\[
V(E) = V_\mu(E)  := \mu(E) . 
\]
If $E$ is of locally finite perimeter, we denote its $\mu$-weighted perimeter by:
\[
A(E) = A_\mu(E) := \int_{\partial^* E \cap \interior M} \Psi_{\mu} d\H^{n-1} 
\]
(and set $A(E) = \infty$ otherwise), where $\interior M = M \setminus \partial M$ denotes the interior of $M$. 

 Denote by $\I = \I(M^n,g,\mu) : [0,1] \rightarrow \R_+$ the corresponding isoperimetric profile, given by:
\[
\I(\bar v) := \inf \{ A(E) \; ; \;  V(E) = \bar v \} ,
\]
where the infimum is over all Borel subsets $E \subset (M,g)$. A set $E$ realizing this infimum is called an isoperimetric minimizer. A standard compactness argument based on the finiteness of $\mu$ and the lower semi-continuity of (weighed) perimeter ensures that isoperimetric minimizers exist for all $\bar v \in [0,1]$ (see e.g.~\cite[Proposition 12.15]{MaggiBook} and \cite[Theorem 4.1 (i)]{EMilmanNeeman-GaussianMultiBubble}, and note that the argument carries through to the case that the density is locally Lipschitz and positive). 
Since modifying $E$ by null-sets does not change $\partial^* E$, thereby preserving $A(E)$ and $V(E)$, we may always modify $E$ by a null-set so that it is relatively closed in $(M^n,g)$, $\overline{\partial^* E} \cap \interior M = \partial E \cap \interior M$  and $\H^{n-1}((\partial E \setminus \partial^* E) \cap \interior M^n) = 0$ -- this is always possible by e.g.~\cite[Theorem 4.1]{EMilmanNeeman-GaussianMultiBubble}. 
The following are some well-known properties of the isoperimetric profile.

\begin{proposition} \label{prop:I-properties}
Let $\I  : [0,1] \rightarrow \R_+$ denote the isoperimetric profile of a weighted Riemannian manifold $(M^n,g,\mu = \exp(-W) d\vol_g)$. 
\begin{enumerate}
\item $\I$ is symmetric about $1/2$, namely $\I(1-\bar v) = \I(\bar v)$ for all $\bar v \in [0,1]$. 
\item Whenever $n\geq 2$, $\I$ is continuous (in fact, locally $\frac{n-1}{n}$-H\"older). 
\item $\I$ is concave if either of the following assumptions hold:
\begin{enumerate}
\item The boundary $\partial M$ is either empty or locally convex (in the sense that the second fundamental form $\II_{\partial M} \geq 0$), $W \in C^\infty_{loc}$ and $\Ric_g + \nabla^2_g W \geq 0$.  \item $M^n$ is a convex subset of Euclidean space $(\R^n,\abs{\cdot}^2)$ and $W : M^n \rightarrow \R$ is convex. 
\end{enumerate}
\end{enumerate}
\end{proposition}
\begin{proof}
The first assertion is immediate by taking complements. For the second, see e.g.~\cite[Lemma 6.9]{EMilman-RoleOfConvexity} (stated for complete oriented manifolds, but the only properties used in the proof were the bounded geometry and bounded density on geodesic balls). The third assertion follows from the work of Bavard--Pansu \cite{BavardPansu}, Sternberg--Zumbrun \cite{SternbergZumbrun}, Bayle \cite{BayleThesis} and Bayle--Rosales \cite{BayleRosales}, see \cite[Theorem A.3 and Corollary 6.12]{EMilman-RoleOfConvexity}; the condition $\Ric_g + \nabla^2_g W \geq 0$ is the celebrated Bakry--\'Emery Curvature-Dimension condition $\CD(0,\infty)$ \cite{BakryEmery}. 
\end{proof}

In particular, all of the above applies to the two main examples we shall study in this work, the three-dimensional torus $\T^3(\beta)$ and the slab over a Gaussian base $\G^n_T$ ($n \geq 2$), and we conclude that their isoperimetric profiles are symmetric, continuous and concave on $[0,1]$. 

\begin{remark} \label{rem:reflect-slab}
It will sometimes be convenient to consider the double-cover of a slab $M_T$, defined as $(\S^1(2T),\abs{\cdot}^2,\frac{1}{2T} \m\mycorner_{\S^1(2T)}) \otimes (\bM^{n-1}, \bg, \bmu)$, which has the convenience of not having any boundary. Here $[0,T]$ is embedded in $\S^1(2T)$ , where $\S^1(2T) = [-T,T]/\sim$ is obtained by identifying $\{-T\} \sim \{T\}$. We denote the double-cover by $2 M_T = 2 (M^n_T,g,\mu_T)$. 

By a well-known reflection argument, we have $\I(2 M_T) = \I(M_T)$. Indeed, if $E$ is a minimizer in $M_T$ then its reflection $2 E$ across $\{0\} \times \bM^{n-1}$ in $2 M_T$ has the same (weighted) volume and perimeter as those of $E$ in $M_T$, and hence $\I(2 M_T) \leq \I(M_T)$. On the other hand, if $\tilde E$ is a minimizer in $2 M_T$, consider $E[t] = \tau_t (\tilde E) \cap M_T$, where $\tau_t$ denotes horizontal rotation by $t \in \S^1(2 T)$. Since $\S^1(2 T) \ni t \mapsto V_{2 M_T}(E[t]) - V_{2 M_T}(E[t + T])$ is odd and continuous, it must vanish at some $t_0$, and hence $V_{M_T}(E[t_0]) = V_{M_T}(E[t_0+T]) = V_{2 M_T}(\tilde E)$. Selecting the half having the least relative perimeter in $M_T$, we have without loss of generality $A_{M_T}(E[t_0]) \leq A_{M_T}(E[t_0 + T])$. But since $A_{M_T}(E[t_0]) + A_{M_T}(E[t_0+T]) \leq 2 A_{2 M_{T}}(\tilde E)$, it follows that $A_{M_T}(E[t_0]) \leq A_{2 M_T}(\tilde E)$, and we deduce that $\I(M_T) \leq \I(2 M_T)$. We conclude that $\I(2 M_T) = \I(M_T)$, verifying in particular that if $E$ is a minimizer in $M_T$ then $2 E$ is a minimizer in $2 M_T$. 

Note that this argument does not require the base to be a manifold without boundary, and applies to manifolds with corners. Applying it successively on each coordinate, we see that $\I(\T^n) = \I(\Q^n)$ for all $n \geq 1$ and $\I(\T^3(\beta)) = \I(\Q^3(\beta))$. Uniqueness of minimizers on the double-cover $2 M_T$ immediately implies uniqueness on $M_T$, and typically this can also be reversed, but we shall not be concerned with uniqueness in this work. 
\end{remark}

Given a slab $(M_T,g,\mu_T)$, we denote by $\I_T = \I(M_T,g,\mu_T)$ its corresponding isoperimetric profile, and by $\bI = \I(\bM,\bg,\bmu)$ the isoperimetric profile of its base. Given $\bI$, we shall attempt to infer information on $\I_T$. By considering cylindrical sets (both horizontal and vertical), obviously
\begin{equation} \label{eq:T-min}
\I_T \leq  \min(\bI , \frac{1}{T}) . 
\end{equation}
However, it would be very naive to expect the converse inequality to hold (except, perhaps, in very special circumstances), to an extent we shall investigate in this work.

\subsection{Base-Induced Isoperimetric Profile}

Given a Borel subset $E$ of a slab $(M^n_T,g)$, we denote for all $t \in [0,T]$ by $E_t$ the (Borel) vertical section
\[
 E_t := \{ y \in \bM^{n-1} \; ; \; (t,y) \in E \}.
\]
and for all $y \in \bM^{n-1}$ by $E^y$ the (Borel) horizontal section:
\[
E^y := \{ t \in [0,T] \; ; \; (t,y) \in E \} . 
\]
We assume that $E$ is of locally finite perimeter in $(M^n_T,g)$, and set:
\[
V_E(t) := V_{\bmu}(E_t) ~,~ A_E(t) := A_{\bmu}(E_t) . 
\]
In that case, it is known that $V_E(t)$ is a function of bounded variation on $[0,T]$ (see \cite[Section 3]{BarchiesiCagnettiFusco}). We denote by $DV_E$ its distributional derivative, and decompose it into its absolutely continuous and singular parts as follows:
\[
DV_E = V'_E(t) dt + D_S V_E. 
\]
It is also known that $E_t$ is of locally finite perimeter in $(\bM^{n-1},\bg)$ for almost all $t \in [0,T]$, and that $A_E(t)$ is a Borel function \cite[Theorem 2.4]{BarchiesiCagnettiFusco}. The following is a known computation (see e.g. \cite[Lemma 2.4]{FuscoMaggiPratelli-GaussianProduct}, \cite[Proposition 3.4]{BarchiesiCagnettiFusco}, \cite[Section 4]{CFMP-GaussianIsoperimetricStability}); for completeness, we sketch a proof omitting the technical details:
\begin{proposition} \label{prop:area-formula}
Let $(M_T,g,\mu_T)$ denote a weighted Riemannian slab, and let $E$ be a set of locally finite perimeter in $(M_T,g)$. Then:
\begin{equation} \label{eq:area-formula}
A_{\mu_T}(E) \geq \frac{1}{T} \brac{\int_0^T \sqrt{ V'_E(t)^2 + A_E(t)^2 } dt + \norm{D_S V_E}} .
\end{equation}
Equality holds if for every $t,s \in [0,T]$, either $E_t \subset E_s$ or $E_s\subset E_t$, and for almost every $t \in [0,T]$, $\scalar{\n_E(t,y),e_1}$ is constant for $\H^{n-2}$-almost all $y \in \partial^* E_t$. 
\end{proposition}
Here $\norm{D_S V_E}$ denotes the total-variation norm of the singular measure $D_S V_E$.  
\begin{proof}[Sketch of Proof] We assume for simplicity that $E$ has smooth boundary. The perimeter of the vertical part of the boundary (where $\scalar{\n_E(x) , e_1}=1$) is lower-bounded by $\norm{D_S V_E}$ (times the horizontal density $\frac{1}{T}$) since $\bmu(E_{t+} \Delta E_{t-}) \geq \abs{\bmu(E_{t+}) - \bmu(E_{t-})}$, and equality holds if $E_{t+} \subset E_{t-}$ or $E_{t-} \subset E_{t+}$. Once the vertical part is taken care of, we may assume $V_E$ is absolutely continuous and calculate by the co-area formula:
\[
A_{\mu_T}(E) = \frac{1}{T} \int_0^T \int_{(\partial^* E)_t} \frac{\Psi_{\mu_0}(y) d\H^{n-2}(y)}{\abs{\sin \alpha(t,y)}} dt ,
\]
where $\alpha(t,y)$ is the angle between the normal $\n_E(t,y)$  
and the horizontal direction $e_1$,  so that $\cos \alpha(t,y) = \scalar{\n_E(t,y),e_1}$. 
On the other hand, the rate of vertical change in $E_t$ per horizontal movement is given by $\cot \alpha(t,y)$, 
and so
\[
V'_E(t) = \int_{(\partial^*E)_t} \cot(\alpha(t,y)) \Psi_{\mu_0}(y) d\H^{n-2}(y) .
\]
Since $1 / |\sin \alpha| = \Phi(\cot \alpha)$ for $\Phi(z) = \sqrt{1 + z^2}$ which is convex, we may apply Jensen's inequality to obtain:
\begin{align*}
A_{\mu_T}(E) & = \frac{1}{T} \int_0^T \int_{(\partial^* E)_t} \Phi(\cot \alpha(t,y)) \Psi_{\mu_0}(y) d\H^{n-2}(y) \\
& \geq \frac{1}{T} \int_0^T A_E(t) \Phi(V'_E(t) / A_E(t)) dt \\
& = \frac{1}{T} \int_0^T \sqrt{ V'_E(t)^2 + A_E(t)^2 } dt ,
\end{align*}
with equality if for almost every $t \in [0,T]$, $\alpha(t,y)$ is constant for $\H^{n-2}$-almost every $y \in (\partial^* E)_t$ (which for almost all $t$ coincides $\H^{n-2}$-almost-everywhere with $\partial^* E_t$ by \cite[Theorem 2.4]{BarchiesiCagnettiFusco}).
\end{proof}

Since $A_E(t) \geq \bI(V_E(t))$ and $\frac{1}{T} \int_0^T V_E(t) dt = V(E)$, we are naturally led to the following:
\begin{definition}[Base-Induced Isoperimetric Profile] \label{def:base-profile}
The base-induced isoperimetric profile $\I^b_T = \I^b_T(\bI)$ is defined as:
\begin{equation} \label{eq:Ib}
\I^b_T(\bar v) := \inf \set{ \frac{1}{T}  \brac{\int_0^T \sqrt{ v'(t)^2 + \bI(v(t))^2 } dt + \norm{D_S v}} \; ; \; \frac{1}{T} \int_0^T v(t) dt = \bar v } ,
\end{equation}
where the infimum is over all functions $v : [0,T] \rightarrow [0,1]$ of bounded variation. 
\end{definition}
In fact, whenever $\bI$ is continuous, a standard compactness argument for BV functions ensures that the above infimum is always attained. 
Note that by testing constant and heavyside functions, we trivially have $\I^b_T \leq \min(\bI,\frac{1}{T})$. 
Moreover, an immediate consequence of Proposition \ref{prop:area-formula} is:

\begin{corollary} \label{cor:base-induced-profile}
For every weighted Riemannian slab, $\I_T \geq \I^b_T$. 
\end{corollary}

Unfortunately, in general, the converse inequality need not hold, due to the fact that isoperimetric minimizers on the base which realize $\bI(\bar v)$ may not be nested as a function of $\bar v$, creating some extra vertical perimeter, or more generally, not satisfy equality in (\ref{eq:area-formula}). However, we can trivially state:

\begin{corollary}[Nested and Aligned Base Minimizers] \label{cor:nested}
Assume that there exists a nested family of isoperimetric minimizers for the base $(\bM,\bg,\bmu)$, namely a mapping $[0,1] \ni \bar v \mapsto E_0(\bar v)$ of Borel subsets of $(\bM,\bg)$ so that $A_{\bmu}(E_0(\bar v)) = \bI(\bar v)$ and
\[
 \bar v_1 \leq \bar v_2 \;\; \Rightarrow \;\; E_0(\bar v_1) \subset E_0(\bar v_2).
 \]
 Consider the set $E^1$ in the slab $(M_1,g,\mu_1)$ of width $1$ whose sections are given by $E^1_{\bar v} := E_0(\bar v)$, and assume that $E^1$ is a Borel set of locally finite perimeter and that for all $\bar v \in (0,1)$, $\scalar{\n_{E^1}(\bar v,y),e_1}$ is constant for almost all $y \in \partial^* E^1_{\bar v}$. \\
 Then $\I_T = \I^b_T$ for any slab with base $(\bM,\bg,\bmu)$ (and width $T > 0$). 
\end{corollary}

Note that the above assumption is satisfied for a Gaussian base $(\R^{n-1} , \abs{\cdot}^2, \gamma^{n-1})$, but not when the base is a $2$-dimensional Torus (or square), since the minimizers cannot be nested in the latter case. 

\begin{proof}[Proof of Corollary \ref{cor:nested}]
Given a function $v : [0,T] \rightarrow [0,1]$ of bounded variation, consider the Borel set $E \subset (M_T,g)$ with sections $E_t = E_0(v(t))$. By the definition of $\I^b_T$, it is enough to show that equality holds in (\ref{eq:area-formula}), and so it is enough to show that the sufficient condition for equality given by Proposition \ref{prop:area-formula} is satisfied. Indeed, the nestedness assumption guarantees the first part of the condition, and the second part (which need only hold outside the null-set of $t$'s where $v$ is non-differentiable) holds since to first order around a given slice $E_t$, $E$ coincides with $E^1$ around the slice $E^1_{v(t)}$ (after stretching $E^1$ appropriately to match the value of $v'(t)$). 
\end{proof}

\begin{remark}
The reader will note the similarity with the celebrated reformulation by Bobkov \cite{BobkovGaussianIsopInqViaCube} of the Gaussian isoperimetric inequality (say on $\R$) in the following functional form:
\begin{equation} \label{eq:Bobkov-inq}
\int_{\R} \sqrt{v'(t)^2 + \I_\gamma(v(t))^2} d\gamma^1(t) \geq \I_\gamma \brac{\int_{\R} v(t) d\gamma^1(t)} ,
\end{equation}
valid for all (say) locally Lipschitz functions $v : \R \rightarrow [0,1]$. 
\end{remark}

\subsection{Reduction to two-dimensional isoperimetric problem}

Computing the base-induced profile $\I^b_T$ given by (\ref{eq:Ib}) is a classical optimization problem in the calculus of variations. Instead of working on the functional level, it will be more convenient and insightful for us to reformulate it as a two-dimensional isoperimetric problem on a certain ``model slab". We draw our inspiration from the work of Bobkov, who observed in \cite{BobkovGaussianIsopInqViaCube} that his one-dimensional functional inequality (\ref{eq:Bobkov-inq}) is equivalent to the two-dimensional Gaussian isoperimetric inequality (in fact, the equivalence holds between the analogous $n$-dimensional Gaussian functional inequality and the $(n+1)$-dimensional Gaussian isoperimetric inequality). In Bobkov's case, one can use $I_\gamma$ in both sides of (\ref{eq:Bobkov-inq}), a fortunate coincidence which in some sense characterizes the Gaussian measure and is responsible for its tensorization and dimension-free properties, but the general situation is more complicated, as we shall see below. 

\smallskip

We will henceforth assume that:
\[
 \text{$\bI : [0,1] \rightarrow \R_+$ is concave and continuous}
\]
(of course, continuity just adds information at the end-points $\bar v \in \{0,1\}$).
By Proposition \ref{prop:I-properties}, this holds for the two main examples we investigate in this work -- when the base is the $2$-dimensional flat torus $\T^2$, and when it is the $(n-1)$-dimensional Gaussian space $\G^{n-1}_0$. 

\smallskip

It was shown by Bobkov \cite{BobkovExtremalHalfSpaces} that concave, symmetric about $1/2$, functions $\I : [0,1] \rightarrow \R_+$ are in one-to-one correspondence with log-concave, even, probability densities $\varphi : \R \rightarrow \R_+$. Recall that a function $\varphi : \R \rightarrow \R_+$ is called log-concave if $\log \varphi : \R \rightarrow \R \cup \{-\infty\}$ is concave. To ensure that the correspondence described below is one-to-one, we will always modify $\varphi$ on $\partial \supp(\varphi)$ to make it continuous on its (convex) support;  note that $\varphi$ is always locally Lipschitz on $M_{\varphi}$, the interior of $\supp(\varphi)$. Whenever $\varphi$ is a log-concave probability density, the isoperimetric profile $\I$ of $(M_\varphi,\abs{\cdot}^2,\varphi(s) ds)$ is concave and symmetric around $1/2$, and conversely, every such $\I$ is the isoperimetric profile of $(M_\varphi,\abs{\cdot}^2,\varphi(s) ds)$ for some log-concave probability density $\varphi$. The requirement that $\varphi$ be even ensures that $\varphi$ is determined uniquely by $\I$, via the relation:
\begin{equation} \label{eq:I-varphi}
\I = \varphi \circ \Phi^{-1} ~,~ \Phi^{-1}(v) = \int_{1/2}^v \frac{dw}{\I(w)} ,
\end{equation}
where $\Phi(s) = \int_{-\infty}^s \varphi(x) dx$. This correspondence is an immediate consequence of Bobkov's observation that the isoperimetric minimizers on $(M_\varphi,\abs{\cdot}^2,\varphi(s) ds)$ are given by half-lines (and the evenness ensures that it is enough to consider left half-lines). Note that when all expressions are smooth, indeed
\begin{equation} \label{eq:I''}
 \I' = (\log \varphi)' \circ \Phi^{-1} \text{ and } \I'' = \frac{(\log \varphi)''}{\varphi} \circ \Phi^{-1},
 \end{equation}
and the relation between the concavity of $\I$ and that of $\log \varphi$ becomes apparent. 

\smallskip

Now let $\varphi_{\bI}$ denote the unique even log-concave probability density satisfying (\ref{eq:I-varphi}) for $\I = \bI$, and set $\Phi_{\bI}(s) = \int_{-\infty}^s \varphi_{\bI}(x) dx$. We denote the interior of $\supp(\varphi_{\bI})$ by $(-R_{\bI},R_{\bI})$, $R_{\bI} \in (0,\infty]$. 
 In the two main examples we examine in this work, it is straightforward to write $\varphi_{\bI}$ explicitly:
\begin{itemize}
\item When the base is the $2$-Torus $\T^2$, $\I_0 = \I(\T^2) =: \I_{\T^2}$. We shall recall in Section \ref{sec:Q3} that $\I_{\T^2}(\bar v) = \min(\sqrt{ \pi \bar v} , 1 , \sqrt{\pi (1-\bar v)})$, and it follows that $\varphi_{\I_{\T^2}} = \varphi_{\T^2}$ is simply piecewise linear on $(-(\frac{1}{2} + \frac{1}{\pi}),\frac{1}{2} + \frac{1}{\pi})$ (explaining why we need to assume that our densities are only Lipschitz regular):
\begin{equation} \label{eq:varphi2}
\varphi_{\T^2}(s) := \begin{cases} \frac{\pi}{2} (\frac{1}{\pi} + \frac{1}{2} + s) &  s \in (-(\frac{1}{2} + \frac{1}{\pi}) ,-(\frac{1}{2} - \frac{1}{\pi})] \\ 
1 & s \in [-(\frac{1}{2} - \frac{1}{\pi}) , \frac{1}{2} - \frac{1}{\pi}] \\ 
\frac{\pi}{2} (\frac{1}{\pi} + \frac{1}{2} - s) & s \in [\frac{1}{2} - \frac{1}{\pi}, \frac{1}{2} + \frac{1}{\pi})
\end{cases} . 
\end{equation}
\item When the base is Gaussian space $\G^{n-1}$, $\bI = \I(\G^{n-1}) = \I_\gamma$. Therefore $\varphi_{\I_\gamma} = \varphi_\gamma$ is the standard Gaussian density on $\R$. 
\end{itemize}

\begin{definition}[Model Slab] \label{def:model-slab}
The Model Slab for the base profile $\bI$ is defined as the following two-dimensional slab of width $T > 0$ with vertical base $((-R_{\bI},R_{\bI}) , \abs{\cdot}^2 , \varphi_{\bI}(s) ds)$:
\[
S_T(\I_0) := ([0,T],\abs{\cdot}^2,\frac{1}{T} \m\mycorner_{[0,T]}) \otimes ((-R_{\bI},R_{\bI}), \abs{\cdot}^2 , \varphi_{\bI}(s) ds) . 
\]
\end{definition}

\begin{proposition} \label{prop:coincide}
The isoperimetric profile of the model slab $S_T(\I_0)$ coincides with the base-induced profile $\I^b_T$. 
\end{proposition}
\begin{proof}
The simplest way to see this is to invoke Corollary \ref{cor:nested}. Indeed, the isoperimetric minimizers of the vertical base $((-R_{\bI},R_{\bI}) , \abs{\cdot}^2 , \varphi_{\bI}(s) ds)$ are left half-lines and therefore trivially nested and satisfy the condition on constancy of the normal angle. Since the isoperimetric profile of the  base is $\I_0$ by construction, Corollary \ref{cor:nested} verifies that $\I(S_T(\I_0)) = \I^b_T$. 
\end{proof}

Since the density of $S_T(\I_0)$ is of the form $\exp(-W(t,s)) = \frac{1}{T} \varphi_{\bI}(s)$ with $W : [0,T] \times (-R_{\bI},R_{\bI}) \rightarrow \R$ convex, Proposition \ref{prop:I-properties} immediately yields:
\begin{corollary} \label{cor:Ib-concave}
$\I^b_T = \I(S_T(\I_0))$ is continuous, concave and symmetric about $1/2$. 
\end{corollary}

\begin{remark} \label{rem:BartheRos}
Combining Corollary \ref{cor:base-induced-profile} with Proposition \ref{prop:coincide}, we obtain, denoting $L = ([0,T],\abs{\cdot}^2,\frac{1}{T} \m\mycorner_{[0,T]})$:
\[
 \I(L \otimes (\bM^{n-1}, \bg, \bmu)) \geq \I(L \otimes ((-R_{\bI},R_{\bI}) , \abs{\cdot}^2 , \varphi_{\bI}(s) ds)) .
 \]
 This is in fact a direct consequence of an isoperimetric comparison theorem of Ros \cite[Theorem 22]{RosIsoperimetricProblemNotes} and Barthe \cite[Theorem 8]{BartheTensorizationGAFA}. When $(\bM^{n-1}, \bg, \bmu)$ is $(n-1)$-dimensional Gaussian space $\G^{n-1}$, we have equality above by Corollary \ref{cor:nested}, but this may also be shown directly by employing Ehrhard symmetrization \cite{EhrhardPhiConcavity}. More precisely, these statements are proved in the aforementioned references for the notion of outer Minkowski content, which is slightly weaker than the notion of weighted perimeter we employ above; however, by \cite[Corollary 3.8]{MorganRegularityOfMinimizers}, the boundary of an isoperimetric minimizer in our setting is of class $C^{1,\alpha}_{loc}$ (apart from a singular set of Hausdorff dimension at most $n-8$), and so these two notions a-posteriori coincide for minimizers, and hence the above references do apply. 
 The above inequality and equality for the Gaussian case are the only two statements we shall require in the sequel, but we have chosen to present things in a more pedagogical manner in the spirit of \cite{FuscoMaggiPratelli-GaussianProduct}, first without assuming that $\I_0$ is concave and then adding on this assumption. Along the way, we obtained another perspective on the above comparison theorem, interpreting it as a consequence of Proposition \ref{prop:area-formula}.
 \end{remark}

\begin{remark}
Since it is enough to test sets with smooth non-vertical boundary in the definition the isoperimetric profile of $S_T(\I_0)$, 
it is possible to show that it is enough to take infimum over absolutely continuous functions in the definition (\ref{eq:Ib}) of $\I^b_T$; in that case the $\norm{D_S v}$ term disappears, as in (\ref{eq:Bobkov-inq}). 
\end{remark}

A simple (and well-understood) corollary is the following:

\begin{corollary} \label{cor:at-half}
Assume that $\frac{\bI}{\bI(1/2)} \geq \frac{\I_\gamma}{\I_\gamma(1/2)}$ on $[0,1]$. Then:
\begin{equation} \label{eq:Gaussian-minorant}
\I_T \geq \I(S_T(\bI))  \geq \min\brac{\bI(1/2) , \frac{1}{T}} \frac{\I_\gamma}{\I_\gamma(1/2)} \text{ on $[0,1]$.}
\end{equation}
In particular:
\[
\I_T(1/2) = \I(S_T(\bI))(1/2) = \min\brac{\bI(1/2) , \frac{1}{T}} . 
\]
\end{corollary}
\begin{proof}
Whenever $\bI$ is concave and $\bI  \geq \frac{1}{L} \I_\gamma$ for some $L < \infty$, then it is easy to see that the natural monotone map $P_2$ pushing forward $\gamma$ onto $\varphi_{\bI}(s) ds$ is $L$-Lipschitz \cite[Subsection 4.3]{EMilman-SpectrumAndContractions}. In particular, since $\max \I_{\gamma} = \I_\gamma(1/2) = 1/\sqrt{2 \pi}$, the natural monotone map $P_1$ pushing forward $\gamma$ onto $\frac{1}{T} \m\mycorner_{[0,T]}$ is $\frac{T}{\sqrt{2 \pi}}$-Lipschitz. Consequently, the map $P = (P_1,P_2) : \R^2 \rightarrow S_T(\bI)$ pushes forward $\G^2 = (\R^2, \abs{\cdot}^2,\gamma^2)$ onto $S_T(\bI)$ and is $\max(L , \frac{T}{\sqrt{2 \pi}})$-Lipschitz. Since $\I(\G^2) = \I_\gamma$, by a standard transference principle for isoperimetric inequalities (e.g.~\cite[Section 5.3]{EMilman-RoleOfConvexity}), we obtain:
\[
\I(S_T(\bI)) \geq \frac{1}{\max(L , \frac{T}{\sqrt{2 \pi}})} \I_{\gamma} . 
\]
Our assumption is that this holds with $L = \frac{\I_{\gamma}(1/2)}{\bI(1/2)}$, and so recalling Corollary \ref{cor:base-induced-profile} and Proposition \ref{prop:coincide} (or Remark \ref{rem:BartheRos}), we deduce (\ref{eq:Gaussian-minorant}).
Applying this to $\bar v=1/2$ and recalling (\ref{eq:T-min}), we deduce:
\[
\min\brac{\bI(1/2) , \frac{1}{T}} \geq \I_T(1/2) \geq \I(S_T(\bI))(1/2) \geq \min\brac{\bI(1/2) , \frac{1}{T}}  .
\]
\end{proof}

\subsection{Monotone Minimizers in $S_T(\bI)$}

So far, we've only used that the density of the vertical factor of $S_T(\I_0)$ is log-concave, and thus its minimizers are (nested) half-lines. It is now time to use that the (uniform) density of the horizontal factor is also log-concave and thus enjoys the same property. In conjunction, this implies that minimizers are jointly monotone in both coordinates:

\begin{proposition}[Monotone Minimizers] \label{prop:monotone}
For every $\bar v \in (0,1)$, there is a ``downward monotone" closed isoperimetric minimizer $E$ in $S_T(\I_0)$ with $V(E) = \bar v$, namely of the form:
\begin{equation} \label{eq:downward-f}
E = \{ (t,s) \in [0,T] \times (-R_{\bI},R_{\bI}) \; ; \;  s \leq f(t) \} ,
\end{equation}
for some \emph{non-increasing} upper-semi-continuous function $f : [0,T] \rightarrow [-R_{\bI},R_{\bI}]$ which is continuous at $t \in \{0,T\}$.\\
Equivalently, in the definition (\ref{eq:Ib}) of $I^b_T(\bar v)$, the minimum is attained for a non-increasing upper-semi-continuous function $v : [0,T] \rightarrow [0,1]$ which is continuous at $v \in \{0,T\}$. These functions are related by $v(t) = \Phi_{\I_0}(f(t))$. 
\end{proposition}
\begin{remark}
With some more effort, one can show that \emph{any} minimizer in $S_T(\I_0)$ must be of the above form, up to horizontal or vertical reflections (or both), and of course up to null sets. But we will not require this here. 
\end{remark}

For the proof, it will be convenient to invoke a general symmetrization theorem of Ros \cite[Proposition 8]{RosIsoperimetricProblemNotes}, which for simplicity we state here in a more restrictive form; for completeness, we sketch a proof in the spirit of \cite{FuscoMaggiPratelli-GaussianProduct} using the results of the previous subsection.
\begin{theorem}[Ros]
Let $((-R_i,R_i),\abs{\cdot}^2,\mu_i = \varphi_i(t) dt)$, $i=1,2$, denote two (one-dimensional) weighted Riemannian manifolds, and assume that $\varphi_2(t)$ is an even log-concave density. Given a Borel set $E$ in the product space:
\[
((-R_1,R_1) \times (-R_2,R_2),\abs{\cdot}^2,\mu = \mu_1 \otimes \mu_2) ,
\]
denote by $\Sym_2 E$ the Borel set whose vertical sections $(\Sym_2 E)_t$ are closed left half-lines satisfying:
\[
\mu_2((\Sym_2 E)_t) = \mu_2(E_t) \;\;\; \forall t \in (-R_1,R_1) . 
\]
Then $V_\mu(\Sym_2 E) = V_\mu(E)$ and
\begin{equation} \label{eq:Sym2A}
A_{\mu}(\Sym_2 E) \leq A_{\mu}(E) . 
\end{equation}
\end{theorem}
\begin{proof}[Sketch of Proof]
We may assume that $E$ is of locally finite perimeter (otherwise there is nothing to prove). Applying a version of Proposition \ref{prop:area-formula} (proved in exactly the same manner):
\[
A_{\mu}(E) \geq \int_{-R_1}^{R_1} \sqrt{ V'_E(t)^2 + A_E(t)^2 } \varphi_1(t) dt + \int_{-R_1}^{R_1} \varphi_1(t) D_S V_E(dt) ,
\] 
where as usual $V_E(t) := \mu_2(E_t)$ and $A_E(t) := A_{\mu_2}(E_t)$. Clearly $V_E \equiv V_{\Sym_2 E}$, and since left half-lines are isoperimetric minimizers on $((-R_2,R_2),\abs{\cdot}^2,\mu_2)$, we also have $A_E(t) \geq A_{\Sym_2 E}(t)$ for all $t$. Consequently:
\[
A_{\mu}(E) \geq \int_{-R_1}^{R_1} \sqrt{ V'_{\Sym_2 E}(t)^2 + A_{\Sym_2 E}(t)^2 } \varphi_1(t) dt + \int_{-R_1}^{R_1} \varphi_1(t) D_S V_{\Sym_2 E}(dt) .
\]
Invoking now the equality case of Proposition \ref{prop:area-formula} (as in the proof of Corollary \ref{cor:nested}), it follows that the latter expression is equal to $A_{\mu}(\Sym_2 E)$, establishing (\ref{eq:Sym2A}). That $V_\mu(\Sym_2 E) = V_\mu(E)$ is trivial by Fubini's theorem, concluding the proof.
\end{proof}

\begin{proof}[Proof of Proposition \ref{prop:monotone}]
Instead of the slab $S_T(\I_0)$ on $[0,T]$, we may clearly consider the same slab on $(-T/2,T/2)$.
Denote $\mu_1 = \frac{1}{T} \m\mycorner_{(-T/2,T/2)}$ and $\mu_2 = \varphi_{\I_0} \m\mycorner_{(-R_{\bI},R_{\bI})}$, and note that both of the corresponding densities are log-concave and even. 
Given an isoperimetric minimizer $E$ in $S_T(\I_0)$ with $V(E) = \bar v$, we apply the previous theorem twice, firstly for the vertical sections and secondly for the horizontal ones. It follows that the set $\tilde E = \Sym_1 \Sym_2 E$ is also a minimizer (with the same weighed volume), where $\Sym_1 F$ denotes the analogous horizontal symmetrization of $F$. Note that $\Sym_i \tilde E = \tilde E$ for both $i =1,2$; this is trivial for $i=1$, but also holds for $i=2$ since:
\[
s_1 \leq s_2 \;\; \Rightarrow \;\; (\Sym_2 E)^{s_1} \supset (\Sym_2 E)^{s_2} \;\; \Rightarrow \;\; \mu_1( (\Sym_2 E)^{s_1}) \geq \mu_1( (\Sym_2 E)^{s_2}) .  
\] 
Consequently, $\tilde E$ is of the asserted ``downward monotone" form for some non-increasing $f : (-T/2,T/2) \rightarrow [-R_{\bI},R_{\bI}]$, which we may translate and extend by continuity and monotonicity to the entire closed interval $[0,T]$. By modifying $f$ on the countable set where it has jump discontinuities, we can always make it upper-semi-continuous and thus $\tilde E$ closed. As isoperimetric minimizers are closed under null-set modifications, $\tilde E$ is still a minimizer with $V(\tilde E) = \bar v$, thereby concluding the proof. 
\end{proof}

\subsection{Boundary regularity and CMC equation}

Recall that $\log \varphi_{\bI} : (-R_{\bI},R_{\bI}) \rightarrow \R$ is assumed concave, and hence locally Lipschitz and differentiable almost-everywhere.
Let us further assume that:

\medskip
\begin{tabular}{m{0.88 \textwidth} m{1em}}
\leftskip=0.3cm plus 0.3fil 
\rightskip=0.3cm plus -0.3fil 
\parfillskip=0cm plus 0.3\textwidth
The set $\SSS_{\bI} \subset (-R_{\bI},R_{\bI})$ where $\log \varphi_{\bI}$ (equivalently, $\varphi_{\bI}$) is non-differentiable consists of isolated points
 and $\varphi_{\bI}$ is $C^\infty_{loc}$ smooth outside $\SSS_{\bI}$. 
&
\end{tabular}

\begin{definition}[Zone]
For every maximal open interval $Z$ where $\varphi_{\I_0}$ is $C^\infty$ smooth, we will call $[0,T] \times Z$ a ``zone" in $S_T(\I_0)$ and the interval $Z$ a ``zonal interval". If $s \in \SSS_{\bI}$ is a point where $\varphi_{\I_0}$ is non-differentiable, we will call $[0,T] \times \{ s\}$ a ``zonal line". 
\end{definition}
 For example, recalling (\ref{eq:varphi2}), $S_T(\I_{\T^2})$ has 3 zones corresponding to the 3 piecewise linear parts of $\varphi_{\I_{\T^2}}$.

\begin{proposition} \label{prop:regularity}
Let $E$ be an isoperimetric minimizer in a weighted Riemannian manifold $(M^n,g,\mu)$. Denote $\Sigma = \overline{\partial ^* E}$, and assume (for simplicity) that $n \leq 7$.
\begin{enumerate}
\item $\Sigma$ is $C^\infty_{loc}$ smooth in every open set $\Omega \subset \interior M^n$ in which the (positive) density $\Psi_\mu$ is of class $C^{\infty}_{loc}$. 
\item Specializing to the case that $(M^n,g,\mu) = S_T(\I_0)$:
\begin{enumerate}
\item \label{it:regularity1} $\Sigma$ is of $C^\infty_{loc}$ smooth in every zone of $S_T(\I_0)$, and is moreover of class $C^{1,1}_{loc}$ on the entire  $S_T(\I_0)$.
In particular, the unit outward unit-normal $\n_{\Sigma}$ is well-defined and locally Lipschitz continuous on the entire $\Sigma$. 
\item \label{it:regularity2} If $\Sigma$ intersects $\partial S_T(\I_0) = \{0,T\} \times (-R_{\bI},R_{\bI})$, it does so perpendicularly. 
\end{enumerate}
\end{enumerate}
\end{proposition}
\begin{proof}
The interior $C^\infty_{loc}$ regularity of $\Sigma = \overline{\partial^* E}$ for an isoperimetric minimizer $E$ in an open subset of $\R^n$ when $n \leq 7$ is classical (see e.g.~\cite{MaggiBook}); for an extension to the weighted Riemannian setting $(M^n,g,\mu)$ in a neighborhood of a point $p \in \interior M^n$ where the density $\Psi_{\mu}$ is $C^\infty_{loc}$ see Morgan \cite[Corollary 3.7 and Subsection 3.10]{MorganRegularityOfMinimizers}. When the density is only locally Lipschitz and $3 \leq n \leq 7$, it is also shown there that $\Sigma \cap \interior M^n$ is $C^{1,\alpha}_{loc}$, but this improves to $C^{1,1}_{loc}$ when $n=2$. To extend this all the way up to the boundary $\partial S_T(\I_0)$, we use a standard reflection argument as in Remark \ref{rem:reflect-slab}: the set $2E$ in the double cover $2 S_T(\I_0)$ obtained by reflecting $E$ across $\{0\} \times (-R_{\bI},R_{\bI})$ must be an isoperimetric minimizer in $2 S_T(\I_0)$, and so the interior regularity applies on $\partial S_T(\I_0)$, establishing (\ref{it:regularity1}). Finally, (\ref{it:regularity2}) also follows since the reflected $\Sigma$ cannot be $C^1_{loc}$ in $2 S_T(\I_0)$ without meeting $\partial S_T(\I_0)$ perpendicularly.
\end{proof}

Being a minimizer of area under a volume constraint, an isoperimetric minimizer is necessarily stationary, meaning that the first variation of area of an isoperimetric minimizer must be equal to a constant multiple $\lambda \in \R$ of its first variation of volume (more background on the first and second variations will be given in the next section). The following is well-known (see e.g.~\cite{BavardPansu}, \cite[Section 18.3]{MorganBook4Ed}, \cite[Proposition 3.3]{MorganJohnson}, \cite[Lemma-3.4.12]{BayleThesis}, \cite[Proposition 3.7]{BayleRosales}):
\begin{proposition} \label{prop:lambda-I}
With the same assumptions and notation as in the previous proposition, there exists $\lambda \in \R$ with the following property:
\begin{enumerate}
\item \label{it:CMC} 
If the positive density $\Psi_\mu$ is of class $C^{\infty}_{loc}$ in some open set $\Omega \subset \interior M^n$, then the weighted mean-curvature $H_{\Sigma,\mu}$ of $\Sigma$ is constant (``constant mean-curvature" or ``CMC") and equal to $\lambda$ in $\Omega$, where:
\begin{equation} \label{eq:Hmu}
H_{\Sigma,\mu} = H_{\Sigma} + \scalar{\nabla \log \Psi_\mu,\n_{\Sigma}} ,
\end{equation}
and $H_{\Sigma}$ denotes the trace of the second fundamental form $\II_{\Sigma}$ of $\Sigma$ with respect to the inward normal $-\n_{\Sigma}$. 
\item 
If $\bar v = V(E) \in (0,1)$ and $\Omega$ as in (\ref{it:CMC}) is non-empty, then:
\[
\underline{\I}^{',-}(\bar v) \geq \lambda \geq \overline{\I}^{',+}(\bar v)  ,
\]
where $\underline{\I}^{',-}(\bar v) := \liminf_{\eps \rightarrow 0-} (\I(\bar v+\eps) - \I(\bar v)) / \eps$ and $\overline{\I}^{',+}(\bar v) :=  \limsup_{\eps \rightarrow 0+} (\I(\bar v+\eps) - \I(\bar v)) / \eps$ denote the lower left and upper right derivatives of $\I$ at $\bar v$, respectively. 
\end{enumerate}
\end{proposition}
\begin{remark}
Note that when $\I$ is concave, the corresponding left and right derivatives $\I^{',-}$ and $\I^{',+}$ always exist on $(0,1)$, and so there is no need to use $\liminf$ and $\limsup$ in the above limits. 
\end{remark}

In the two-dimensional setting, when $\Sigma$ is a curve in $S_T(\bI)$, the equation $H_{\Sigma,\mu} \equiv \lambda$ on the union of all zones is a second order ODE which we can write down explicitly. 
 As $\Sigma$ is of class $C^{1,1}_{loc}$ globally, in order for this second order ODE to uniquely determine $\Sigma$ when crossing a zonal line, we just need to note that $\Sigma$ can a-priori only intersect each zonal line in at most a single point:

\begin{lemma} \label{lem:transverse}
Let $E$ be a downward monotone isoperimetric minimizer in $S_T(\bI)$, and $\Sigma = \partial E$. Then either $\Sigma$ is a non-zonal horizontal line, or else $\Sigma$ intersects every horizontal line in at most a single point. 
 In particular, a (horizontal) zonal line is never an isoperimetric minimizer. 
\end{lemma}
\begin{proof}
Assume that $\Sigma$ intersects $[0,T] \times \{s_0 \}$ in more than a single point. As $E = E_f := \{ (t,s) \in [0,T] \times (-R_{\bI},R_{\bI}) \; ; \; s \leq f(t) \}$ with monotone $f$, it follows that $\Sigma$ coincides with $J \times \{s_0\}$ for some (closed) interval $J  \subset [0,T]$ with non-empty interior. If $\log \varphi_{\bI}$ is non-differentiable at $s_0$, it has distinct left and right derivatives $(\log \varphi_{\bI})^{',-}(s_0) > (\log \varphi_{\bI})^{',+}(s_0)$ by concavity. Hence, we may define a smooth function $u$ supported on $J$, so that $\int_J u(t) dt = 0$ but 
\[
\int_J ((\log \varphi_{\bI})^{',-}(s_0) u_-(t) + (\log \varphi_{\bI})^{',+}(s_0) u_+(t)) dt < 0,
\]
 where $u_{\pm}$ denote the positive and negative parts of $u$. Consequently, a perturbation $f_\eps := f + \eps u$ will preserve the (weighted) volume of $E_{f_\eps}$ to first order and yet strictly decrease its (weighted) perimeter to first order, contradicting stationarity and hence minimality of $E$. 

Therefore $\log \varphi_{\bI}$ must be differentiable at $s_0$, which means that $[0,T] \times \{s_0\}$ is contained in a zone, where we know $\Sigma$ is of class $C^\infty_{loc}$ and satisfies the second order ODE $H_{\Sigma,\mu} \equiv \lambda$. But since $\Sigma$ coincides with $J \times \{s_0\}$ and the ODE is invariant under horizontal translations, uniqueness of solutions to this ODE implies that $\Sigma$ must coincide with the entire horizontal line $[0,T] \times \{ s_0\}$. 
\end{proof}

\begin{proposition} \label{prop:CMC}
Let $E$ be a downward monotone isoperimetric minimizer in $S_T(\I_0)$ of the form (\ref{eq:downward-f}). Denote $\Sigma = \partial E$ and let $J = f^{-1}(-R_{\I_0},R_{\I_0})$ (noting that $J \neq \emptyset$ iff $\Sigma$ is non-empty nor a vertical line). Denote $Z = (-R_{\I_0},R_{\I_0}) \setminus \SSS_{\bI}$ (recall that $\SSS_{\bI}$ is the collection of points where $\varphi_{\bI}$ is non-differentiable), and let $J_\infty = f^{-1}(Z)$. 

\begin{enumerate}
\item On $J_\infty$, the function $f$ is $C^{\infty}_{loc}$, and satisfies the following second order ODE:
\begin{equation} \label{eq:CMC1}
\frac{-f''}{(1 + (f')^2)^{\frac{3}{2}}} + \frac{(\log \varphi_{\I_0})'(f)}{(1 + (f')^2)^{\frac{1}{2}}} = \lambda ,
\end{equation}
where $\lambda \in \R$  is the constant weighted mean-curvature of $\Sigma$. \\
On $J$, the function $f$ is $C^{1,1}_{loc}$, and satisfies the above ODE almost everywhere. 
\item \label{it:perp} If $t \in \{0,T\} \cap J$ then $f'(t) = 0$.  
\item The ODE (\ref{eq:CMC1}) has a first integral: there exists $c \in \R$ so that pointwise on $J$
\begin{equation} \label{eq:CMC2}
\frac{\varphi_{\bI}(f)}{(1 + (f')^2)^{\frac{1}{2}}} -\lambda \Phi_{\bI}(f) \equiv c .
\end{equation}
In particular, $0 \leq \lambda \Phi_{\bI}(f)  + c \leq \varphi_{\bI}(f)$ on $J$. 
\item Equivalently, in terms of $v = \Phi_{\bI}(f)$, $v$ is $C^{\infty}_{loc}$ on $J_\infty$ and $C^{1,1}_{loc}$ on $J$, and satisfies pointwise on $J$:
\begin{equation} \label{eq:CMC3}
\frac{dv}{dt} = - \bI(v) \sqrt{ \brac{\frac{\bI(v)}{\lambda v + c}}^2 - 1} ~ . 
\end{equation}
In particular, $0 \leq \lambda v + c \leq \bI(v)$ on $J$. 
\item Unless $\Sigma$ is a vertical line, $v$ (equivalently, $f$) is continuous on $[0,T]$. In particular, $J$ is relatively open in $[0,T]$, and either $\Sigma$ is a vertical line or $\Sigma = \{ (t,f(t)) \; ; \; t \in J \}$ is the graph of $f$ on $J$. 
\item \label{it:not-infinite} $\lambda v + c > 0$ on $J$. 
\end{enumerate}
Furthermore, unless $\Sigma$ is a (non-zonal) horizontal line:
\begin{enumerate}
\setcounter{enumi}{6}
\item \label{it:monotone} $f$ is strictly monotone decreasing on $J$; equivalently, $\lambda v + c < \bI(v)$ on $\interior J$. 
\end{enumerate}
\end{proposition}
\begin{proof} 
Recall that $E = \{ (t,s) \in [0,T] \times (-R_{\bI},R_{\bI}) \; ; \; s \leq f(t) \}$.
If $\Sigma$ is a (necessarily non-zonal) horizontal line, all of the assertions trivially hold. 
If $\Sigma$ is not a horizontal line, then it intersects every zonal line at isolated points and at most once by Lemma \ref{lem:transverse}, and hence the weighted mean-curvature $\lambda$ of $\Sigma$ is well-defined almost everywhere. In addition, $f$ must be strictly monotone decreasing on $J$, and so  $(\log \varphi_{\bI})'(f(t))$ is well-defined for almost all $t \in J$.  
\begin{enumerate}
\item
Proposition \ref{prop:regularity} implies that $f$ is $C^{\infty}_{loc}$ on $J_\infty$ and $C^{1,1}_{loc}$ on $J$, except where its derivative is $-\infty$ -- we will see that this cannot happen in the proof of assertion (\ref{it:not-infinite}); to avoid making a circular argument, we will not assume that $f' > -\infty$ until then. 

The unweighted mean curvature $H_{\Sigma}$ of the graph of $f$ is well-known and easily computable as the first term in (\ref{eq:CMC1}). The outer unit-normal to the graph is given by $\frac{(-f',1)}{\sqrt{1 + (f')^2}}$, and so recalling that $\Psi_\mu(t,s) = \frac{1}{T} \varphi_{\bI}(s)$, the computation of the second term in (\ref{eq:CMC1}) is complete. See also \cite[Subsection 2.1]{KoisoMiyamoto} for a self-contained computation. The constancy of the weighted mean-curvature where the density is smooth yields (\ref{eq:CMC1}) on $J_\infty$ and almost-everywhere on $J$.
\item This follows since $\Sigma$ meets $\partial S_T(\I_0)$ perpendicularly. 
\item The ODE (\ref{eq:CMC1}) has a first integral since it is invariant under horizontal translations (and so e.g. No\"ether's theorem applies, see e.g. \cite{PedrosaRitore-Products}). Alternatively and equivalently, this follows since (\ref{eq:CMC1}) is the Euler-Lagrange equation for the Lagrangian:
\[
L(f,f',t) := (1 + (f')^2)^{\frac{1}{2}} \varphi_{\bI}(f) - \lambda \Phi_{\bI}(f) ,
\]
and since $L$ does not depend on $t$, the Euler-Lagrange equation reduces to the Beltrami identity stating that the Hamiltonian $\frac{\partial L}{\partial f'} f' - L$ must be constant (see e.g. \cite[Appendix B]{KoisoMiyamoto}). In any case, the left-hand-side of (\ref{eq:CMC2}) is locally Lipschitz on $J$ since $f$ is $C^{1,1}_{loc}$, and since its derivative is equal to zero almost-everywhere by (\ref{eq:CMC1}), we conclude that it is constant, yielding (\ref{eq:CMC2}). 
\item 
Substituting $f = \Phi_{\bI}^{-1}(v)$ into (\ref{eq:CMC2}) and recalling that $\bI = \varphi_{\bI} \circ \Phi_{\bI}^{-1}$, (\ref{eq:CMC3}) easily follows. 
The negative sign of the square root is due to the downward monotonicity of $v$. 
Note that $v(t) \in (0,1)$ for $t \in J$, and hence $\bI(v) > 0$ on $J$ (e.g.~by concavity).
\item
Assume that $\Sigma \neq \emptyset$ is not a vertical line, hence there exists $t_0 \in J$ with $v(t_0) \in (0,1)$. Assume that $v$ has a jump discontinuity at some $t \in J$. This means that the mean-curvature and also the weighted mean-curvature of this vertical segment of $\Sigma$ is $0$, and hence $\lambda =0$.
Since $v'(t) = -\infty$ we must have by (\ref{eq:CMC3}) that $\lambda v(t) + c = 0$, 
and hence also $c=0$. By (\ref{eq:CMC2}), this is impossible unless $v \equiv 0$ or $v \equiv 1$ in $J$, contradicting that $v(t_0) \in (0,1)$.  
\item
By (\ref{eq:CMC3}), the claim is equivalent to showing that $v'(t) \neq -\infty$ (or equivalently $f'(t) \neq -\infty$) on $J$, a debt we still owe from the proof of the first assertion. 
If $\lambda v(t) + c = 0$ for some $t \in  \interior J$, then by (\ref{eq:CMC3}) necessarily $v'(t) = -\infty$, and so $v$ would be strictly decreasing at such $t$ and hence $\lambda v(t+\eps) + c < 0$ for small enough $\eps \lambda > 0$, contradicting $\lambda v + c \geq 0$ on $J$. It follows that we cannot have $v'(t)=-\infty$ (equivalently, $f'(t) = -\infty$) for any $t$ in the relatively open $J$, since this cannot happen in the interior of $J$ by the previous argument, nor can it happen for $t \in \{0,T\}$ by assertion (\ref{it:perp}). 
\item
Finally, in view of (\ref{eq:CMC3}), $f$ being strictly monotone decreasing is equivalent to having  $\lambda v + c < \bI(v)$ on $\interior J$. 
Indeed, since $\bI$ is concave and $\lambda v + c \leq \bI(v)$ on $J$, if there was equality for some $t_0 \in \interior J$ this would imply that $\lambda v + c = \bI(v)$ on the entire $J$, and so $v$ would be constant by (\ref{eq:CMC3}), implying that $\Sigma$ is a horizontal line. 
\end{enumerate}
\end{proof}

\begin{lemma} \label{lem:chord}
With the same assumptions and notation as in the previous Proposition, assume further that $\Sigma$ is non-empty nor a vertical line, and denote:
\[
1 \geq v_1 := v(0) \geq v(T) =: v_0 \geq 0 . 
\]
As $v$ is monotone and continuous on $[0,T]$ and $J$ is non-empty by our assumptions, we may equivalently define:
\[
v_1 := \sup_{t \in J} v(t) = \lim_{t \searrow \inf J} v(t) ~,~ v_0 := \inf_{t \in J} v(t) = \lim_{t \nearrow \sup J} v(t) . 
\]
  Then:
 \[
\lambda v_i + c = \bI(v_i)\;\;, \;\; i=0,1.
\]
 In other words, $\{ \lambda v + c \, ; \, v \in [v_0,v_1] \}$ is the (possibly degenerate) chord of $\bI$ between $v_0$ and $v_1$. In particular:
 \begin{enumerate}
 \item If $v_1 > v_0$ then $\lambda = \frac{\bI(v_1) - \bI(v_0)}{v_1 - v_0}$ (and $\lambda v + c < \bI(v)$ for all $v \in (v_0,v_1)$).  \item If $\lambda \geq 0$ then necessarily $0 \in J$ and $v_1 < 1$. Analogously, if $\lambda \leq 0$ then necessarily $T \in J$ and $v_0 > 0$.
 \item $c \geq 0$, and $c=0$ if $v_0 = 0$. Conversely, if $v_1 > v_0$ then $v_0=0$ if $c=0$. 
  \item We cannot have both $v_0=0$ and $v_1=1$. 
\end{enumerate}
\end{lemma}
\begin{proof}
Note that $v_1  > 0$ and $v_0 < 1$, otherwise $\Sigma$ would be empty. 
Also note that $0 \in J$ iff $v_1 < 1$ and $T \in J$ iff $v_0 > 0$. If $v_1 < 1$ then $v'(0) = 0$ by Proposition \ref{prop:CMC} (\ref{it:perp}), and so evaluating (\ref{eq:CMC3}) at $t=0$, since $\bI(v_1) > 0$ (e.g.~by concavity), we must have $\bI(v_1) = \lambda v_1 + c$. Alternatively, if $v_1 = 1$ then $\bI(v_1) = 0$, and since $0 \leq \lambda v(t) + c \leq \bI(v(t))$ for $t \in J$, it follows by taking the limit $J \ni t \searrow \inf J$ and continuity of $\bI$ and $v$ that $\lambda v_1 + c = 0$. An identical argument holds for $v_0$ using the endpoint $T$ instead of $0$. 

Consequently, $\{ \lambda v + c \, ; \,  v \in [v_0,v_1] \}$ is the asserted chord having slope $\lambda = \frac{\bI(v_1) - \bI(v_0)}{v_1 - v_0}$ if $v_0 < v_1$. We shall see below that we cannot have both $v_0 = 0$ and $v_1=1$. Consequently, if $v_1 = 1$ (and since $0 < v_0 < 1$) then $\lambda < 0$, and similarly if $v_0 =0$ then $\lambda > 0$.

Since $\bI$ is concave, it follows that its chord between $v_0$ and $v_1$ cannot lie below $\bI$ for $v \notin (v_0,v_1)$, and in particular at $v=0$, implying that $c \geq 0$. In addition, if this chord touches $\bI$ at some $v \in (v_0,v_1)$ then it must coincide with $\bI$ on the entire $[v_0,v_1]$; in that case, (\ref{eq:CMC3}) would imply that $v' \equiv 0$ on $J$, which is only possible if $v_0 = v_1$. Consequently, if $v_1 > v_0$ we see again as in Proposition \ref{prop:CMC} (\ref{it:monotone}) that we must have $\lambda v + c < \bI(v)$ for all $v \in (v_0,v_1)$, implying by (\ref{eq:CMC3}) that $v$ is strictly decreasing on $J$. 

If $v_0=0$ then clearly $c=0$, but also conversely, if $c=0$ and $v_1 > v_0$ then necessarily $v_0=0$, since if $v_0 > 0$, the secant line $\lambda v$ would meet $\bI$ at three distinct points $v \in \{0,v_0,v_1\}$, which by concavity implies that $\bI$ coincides with $\lambda v$ for all $v \in [0,v_1]$, contradicting that $\lambda v < \bI(v)$ for all $v \in (v_0,v_1)$. 

Finally, we cannot have both $v_0=0$ and $v_1=1$ since $\lambda v_0 + c = \bI(v_0) = 0$ and $\lambda v_1 + c = \bI(v_1) = 0$ would mean that $\lambda = c = 0$, in contradiction to the positivity $\lambda v + c > 0$ on $J$ from the previous proposition. Another way to see this is to note that such a $\Sigma$ would have weighted length strictly larger than that of a vertical line enclosing the same weighed volume, simply by projecting $\Sigma$ onto said line (and using that the density of the horizontal factor is constant). 
\end{proof}

\subsection{Generalized unduloids}

We can now finally summarize all of the previous information as follows. 

\begin{definition} \label{def:ell}
Given $0 \leq v_0 < v_1 \leq 1$, we denote:
\[
\lambda = \lambda(v_0,v_1) = \frac{\bI(v_1)-\bI(v_0)}{v_1-v_0} ~,~ c = c(v_0,v_1) = \frac{v_1 \bI(v_0) - v_0 \bI(v_1)}{v_1-v_0} ,
\]
and $\ell = \ell_{v_0,v_1} : [v_0,v_1] \rightarrow \R_+$ the chord of $\bI$ between $v_0$ and $v_1$, namely:
\[
\ell(v) = \lambda v + c = \frac{v_1-v}{v_1-v_0} \bI(v_0) + \frac{v-v_0}{v_1-v_0} \bI(v_1) . 
\]
\end{definition}

\begin{theorem} \label{thm:model-slab-main}
Let $E$ be a downward monotone isoperimetric minimizer in $S_T(\I_0)$ of the form (\ref{eq:downward-f}), and denote $\Sigma = \partial E$. 
Set $v = \Phi_{\bI}(f) : [0,T] \rightarrow [0,1]$ and $v_1 = v(0)$, $v_0 = v(T)$. Then $\Sigma$ (equivalently, $v$) is uniquely defined from $v_0,v_1$ as follows -- either:
\begin{enumerate}
\item $v_1 = v_0 = \bar v$. Then $\Sigma$ is either empty ($\bar v \in \{0,1\}$) or a horizontal line ($\bar v \in (0,1)$).
\item $v_1 = 1$ and $v_0=0$. Then $\Sigma$ is a vertical line.
\item $1 \geq v_1 > v_0 \geq 0$ and either $1 > v_1$ or $v_0 > 0$. Then $\ell_{v_0,v_1} < \bI$ on $(v_0,v_1)$ and
\begin{equation} \label{eq:T-formula}
T(v_0,v_1) := \int_{v_0}^{v_1} \frac{dv}{\bI(v) \sqrt{ (\bI(v)/\ell_{v_0,v_1}(v))^2 - 1}} \leq T ,
\end{equation}
with equality if $1 > v_1$ and $v_0 > 0$; in particular, $T(v_0,v_1)$ is finite. Denote $\bar J = [0,T(v_0,v_1)]$ if $1 > v_1$ and $\bar J = [T - T(v_0,v_1) , T]$ if $v_0 > 0$ (so that $\bar J = [0,T]$ if both $1 > v_1$ and $v_0 > 0$). Then $v|_{\bar J}$ is strictly monotone decreasing from $v_1$ to $v_0$, and its inverse $\tau(v) : [v_0,v_1] \rightarrow \bar J = [\bar J_-,\bar J_+]$ is given by:
\begin{equation} \label{eq:tau-formula}
\tau(v) = \int_{v}^{v_1} \frac{d\omega}{\bI(\omega) \sqrt{ (\bI(\omega)/\ell_{v_0,v_1}(\omega))^2 - 1}} + \bar J_- .
\end{equation}
Equivalently, the inverse of  $f|_{\bar J} : \bar J \rightarrow [f_0,f_1]$, $f_i = \Phi_{\bI}^{-1}(v_i)$, is given by:
\begin{equation} \label{eq:f-formula}
\tau \circ \Phi_{\bI} (f) = \int_{f}^{f_1} \frac{ds}{\sqrt{ ((\bI/\ell_{v_0,v_1}) \circ \Phi_{\bI}(s))^2 - 1}} + \bar J_- .
\end{equation}
$\Sigma$ is the graph of $f$ over the relative interior of $\bar J$ in $[0,T]$, and has constant weighted mean-curvature equal to $\lambda = \lambda(v_0,v_1) = \frac{\bI(v_1)-\bI(v_0)}{v_1-v_0}$. In particular, if $\lambda \geq 0$ then $1 > v_1$ and if $\lambda \leq 0$ then $v_0 > 0$. 
\end{enumerate}
\end{theorem}

\begin{proof}
Most of the statements are already contained in Proposition \ref{prop:CMC} and Lemma \ref{lem:chord}. As for the explicit formulas for $v(t)$ and $T(v_0,v_1)$ when $1 \geq v_1 > v_0 \geq 0$, recall the ODE satisfied by $v$ as a function of $t$ from (\ref{eq:CMC3}). Since $\tau(v)$ is its inverse, we have:
\[
\frac{d\tau}{dv} = \frac{1}{dv/dt} = -\frac{1}{\bI(v) \sqrt{ (\bI(v)/\ell_{v_0,v_1}(v))^2 - 1}} ,
\]
and so (\ref{eq:tau-formula}) follows, up to the value of $\tau(v_1)$. Clearly $\tau(v_1) \geq 0$ with equality when $v_1 < 1$ and $\tau(v_0) \leq T$ with equality when $v_0 > 0$, verifying (\ref{eq:T-formula}) with its equality case, the definition of $\bar J$ and the initial condition $\tau(v_1) = \bar J_-$. (\ref{eq:f-formula}) follows from (\ref{eq:tau-formula}) after a change of variables $\omega = \Phi_{\bI}(s)$, as $\frac{d\omega}{ds} = \bI(\omega)$. 
\end{proof}

\begin{definition}[Generalized Unduloids]
When $1 > v_1 > v_0 > 0$, the function $f = \Phi_{\bI}^{-1}(v) : [0,T(v_0,v_1)] \rightarrow [-R_{\bI},R_{\bI}]$ constructed in Theorem \ref{thm:model-slab-main} (by taking $v$ to be the inverse of $\tau$ in (\ref{eq:tau-formula})) will be called the ``generalized unduloid" corresponding to $\bI$ between volumes $v_0$ and $v_1$ (or values $s_0$ and $s_1$, $s_i = \Phi_{\bI}^{-1}(v_i)$). \\
When $v_1=1$ or $v_0 = 0$, we will refer to any $f = \Phi_{\bI}^{-1}(v) : [0,T] \rightarrow [-R_{\bI},R_{\bI}]$ with $T \geq T(v_0,v_1)$  as a ``one-sided generalized unduloid". 
\end{definition}

\begin{corollary} \label{cor:perp}
If $\bI(v)$ is super-linear at $v=0$, namely if $\lim_{v \rightarrow 0} \frac{\bI(v)}{v} = +\infty$, then a one-sided generalized unduloid meets the top or bottom of $S_T(\bI)$ perpendicularly, namely satisfies $\lim_{t \nearrow t_0} f'(t) = -\infty$ for $t_0 = \bar J_+$ if $v_0 = 0$ and $\lim_{t \searrow t_0} f'(t) = -\infty$ for $t_0 = \bar J_-$ if $v_1 = 1$. 
\end{corollary}
\begin{proof}
As $f = \Phi_{\bI}^{-1}(v)$, we have by (\ref{eq:CMC3}) (or directly by (\ref{eq:CMC2})):
\begin{equation} \label{eq:CMC2.5}
\frac{df}{dt} = \frac{1}{\bI(v)} \frac{dv}{dt} = - \sqrt{ (\bI(v)/\ell_{v_0,v_1}(v))^2 - 1} . 
\end{equation}
Since $\ell_{v_0,v_1}(v)$ is of the form $\lambda v$ if $v_0 =0$ and $-\lambda (1-v)$ if $v_1 = 1$, and since $\bI(1-v) = \bI(v)$, the superlinearity of $\bI(v)$ at $v=0$ implies that the right-hand-side above converges to $-\infty$ as $v$ tends to $0$ or $1$, as asserted. 
\end{proof}

\begin{corollary} \label{cor:no-one-sided}
If $\frac{v}{\bI(v)^2}$ is non-integrable at $v=0$, then there are no one-sided generalized unduloids on $S_T(\bI)$ for any $T > 0$. In other words, if $E$ is a downward monotone minimizer on $S_T(\bI)$ and $\Sigma = \partial E$ is non-empty and not a horizontal nor vertical line, then it is necessarily the graph of a generalized unduloid with $1 > v_1 > v_0 > 0$ (so that the cases $v_1=1$ or $v_0 =0$ cannot occur).
\end{corollary}
\begin{proof}
If $v_0=v_0(\Sigma) = 0$, then since $\ell_{v_0,v_1}(v) = \lambda v$ and $\bI(v)$ is concave, it would follow that $T(v_0,v_1)$ defined in (\ref{eq:T-formula}) is infinite (as the integral diverges at $v=0$ by our assumption), a contradiction. Similarly if $v_1 = v_1(\Sigma) = 1$ (as the integral diverges at $v=1$). 
\end{proof}

\subsection{The isoperimetric profile of $S_T(\bI)$}

\begin{lemma} \label{lem:VA}
With the same assumptions and notation as in Theorem \ref{thm:model-slab-main}, assume further that $v_1 > v_0$. Then the weighted volume and area (respectively) of the downward minimizer $E$ in $S_T(\bI)$ are given by:
\begin{align}
\label{eq:V-formula} V(E) & = V_T(v_0,v_1) := \frac{1}{T} \int_{v_0}^{v_1} \frac{v \, dv}{\bI(v) \sqrt{ (\bI(v)/\ell_{v_0,v_1}(v))^2 - 1}} , \\
\label{eq:A-formula} A(E) & = A_T(v_0,v_1) := \frac{1}{T} \int_{v_0}^{v_1} \frac{ dv}{\sqrt{ 1 - (\ell_{v_0,v_1}(v)/\bI(v))^2 }} .
\end{align}
Denoting $d\sigma_T := \frac{1}{T} \varphi_{\bI}(s)\H^1(dt,ds)$, we have when $1 > v_1 > v_0 > 0$:
\[
\int_{\Sigma} \scalar{\n_{\Sigma},e_2}^2 d\sigma_T = \lambda V(E) + c = \ell_{v_0,v_1}(V(E)),\]
and:
\[
\int_{\Sigma} \scalar{\n_{\Sigma},e_1} d\sigma_T = \frac{v_1-v_0}{T} .
\]
Consequently, when $1 > v_1 > v_0 > 0$, we have the following simple estimates:
\[
\frac{v_1-v_0}{T}  \geq A(E) - \ell_{v_0,v_1}(V(E)) \geq \frac{1}{A(E)} \brac{\frac{v_1-v_0}{T} }^2 ,
\]
and thus:
\[
A(E) \geq \frac{\ell_{v_0,v_1}(V(E)) + \sqrt{\ell_{v_0,v_1}(V(E))^2 + 4 \brac{\frac{v_1-v_0}{T}}^2}}{2} . 
\]
\end{lemma}
\begin{proof}
Abbreviate $\ell = \ell_{v_0,v_1}$. Recalling (\ref{eq:CMC3}), we have:
\[
V(E) = \frac{1}{T} \int_{0}^T \Phi_{\bI}(f(t)) dt = \frac{1}{T} \int_{v_1}^{v_0} v \frac{dt}{dv} dv = \frac{1}{T} \int_{v_0}^{v_1} \frac{ v \, dv}{\bI(v) \sqrt{ (\bI(v)/\ell(v))^2 - 1}} .
\]
To see the formula for $A(E)$, recall from (\ref{eq:CMC2.5}) that: \[
f'(t) = - \sqrt{ (\bI(v)/\ell(v))^2 - 1} . 
\]
Since the outer unit-normal $\n_{\Sigma}$ to $\Sigma$ is given by $\frac{(-f'(t),1)}{\sqrt{1 + (f')^2}}$, we deduce:
\begin{equation} \label{eq:n-formulas}
\scalar{\n_{\Sigma},e_2} = \frac{1}{\sqrt{1 + (f')^2}} = \frac{\ell(v)}{\bI(v)} ~,~ 
\scalar{\n_{\Sigma},e_1} = \sqrt{1 - \brac{\frac{\ell(v)}{\bI(v)}}^2} . 
\end{equation}
Therefore, projecting $\Sigma$ onto the vertical axis (as the projection is one-to-one), and changing variables $v = \Phi_{\bI}(s)$ (so that $dv = \varphi_{\bI}(s) ds$), we obtain:
\[
A(E) = \frac{1}{T} \int_{\Sigma} \varphi_{\bI}(s) \H^1(dt,ds) = \frac{1}{T} \int_{s_0}^{s_1} \frac{\varphi_{\bI}(s) ds}{\scalar{\n_{\Sigma},e_1}} = \frac{1}{T} \int_{v_0}^{v_1} \frac{dv}{\sqrt{1 - (\ell(v)/\bI(v))^2}} . 
\]

Now, since $\scalar{\n_{\Sigma},e_2} = \frac{\ell(v)}{\bI(v)}$ where $v = \Phi_{\bI}(s)$, we have $\scalar{\n_{\Sigma},e_2} \varphi_{\bI}(s) = \lambda \Phi_{\bI}(s) + c$, and so projecting $\Sigma$ onto the horizontal axis (the projection is one-to-one) and using that $\bar J = [0,T]$ as $1 > v_1 > v_0 > 0$, we obtain:
\begin{align*}
& \int_{\Sigma} \scalar{\n_{\Sigma},e_2}^2 \frac{\varphi_{\bI}(s)}{T} \H^1(dt,ds) = \frac{1}{T} \int_{\Sigma} \scalar{\n_{\Sigma},e_2} (\lambda \Phi_{\bI}(s) + c) \H^1(dt,ds) \\
& =  \frac{1}{T} \int_{0}^T (\lambda \Phi_{\bI}(f(t)) + c) dt = \lambda V(E) + c  . 
\end{align*}
Similarly, projecting $\Sigma$ onto the vertical axis and changing variables $v = \Phi_{\bI}(s)$:
\[
\int_{\Sigma} \scalar{\n_{\Sigma},e_1} d\sigma_T = \frac{1}{T} \int_{s_0}^{s_1} \varphi_{\bI}(s)  ds = \frac{1}{T} \int_{v_0}^{v_1} dv =  \frac{v_1-v_0}{T}. 
\]

Finally, since:
\[
A(E) - \ell(V(E)) = \int_{\Sigma} (1 - \scalar{\n_{\Sigma},e_2}^2) d\sigma_T = \int_{\Sigma} \scalar{\n_{\Sigma},e_1}^2 d\sigma_T ,
\]
the asserted inequalities follow by comparing to $\int_{\Sigma} \scalar{\n_{\Sigma},e_1} d\sigma_T$
using $\scalar{\n_{\Sigma},e_1} \geq \scalar{\n_{\Sigma},e_1}^2$ on one hand and Jensen's inequality on the other. 
\end{proof}

The following is an immediate corollary of Theorem \ref{thm:model-slab-main} (and Proposition \ref{prop:monotone}):
\begin{corollary}
The isoperimetric profile $\I^b_T = \I(S_T(\bI))$ is given by:
\begin{align*}
& \I(S_T(\bI))(\bar v) = \min   \left ( \frac{1}{T}   , \bI(\bar v) ,  \right .    \\
& \left .  \;\; \inf \set{ A_{T}(v_0,v_1) \; ; \; \begin{array}{l}  T(v_0,v_1) = T , \\ V_{T}(v_0,v_1) = \bar v , \\ 1 > v_1 > v_0 > 0  \end{array} \text{ or } 
\begin{array}{l}  T(v_0,v_1) \leq T , \\ V_{T}(v_0,v_1) = \bar v , \\ 1 > v_1 > v_0 = 0  \end{array}  \text{ or } 
\begin{array}{l}  T(v_0,v_1) \leq T , \\ V_{T}(v_0,v_1) = \bar v , \\ 1 = v_1 > v_0 > 0  \end{array} 
}  \right ) . 
\end{align*}
where $T(v_0,v_1), V_T(v_0,v_1), A_T(v_0,v_1)$ were defined in (\ref{eq:T-formula}), (\ref{eq:V-formula}) and (\ref{eq:A-formula}). 
\end{corollary}

In fact, one can replace the first column above (when $1 > v_1 > v_0 > 0$) by:
\[
\inf \set{ A_{T(v_0,v_1)}(v_0,v_1) \; ; \; \begin{array}{l}  T(v_0,v_1) \leq T , \\ V_{T(v_0,v_1)}(v_0,v_1) = \bar v , \\ 1 > v_1 > v_0 > 0  \end{array} },
\]
thanks to the second inequality in the following simple:

\begin{lemma} \label{lem:monotone}
For all $0 < T_1 \leq T_2$, we have (pointwise):
\[
\frac{T_1}{T_2} \tilde \I_{T_1} \leq \tilde \I_{T_2} \leq \tilde \I_{T_1} ,
\]
for both $\tilde \I_T = \I(M_T)$ and $\tilde \I_T = \I(S_T(\bI))$. 
\end{lemma}
\begin{proof}
Let $P$ be the map from $S_{T_2}(\bI)$ to $S_{T_1}(\bI)$ which pushes forward $\frac{1}{T_2} \m\mycorner_{[0,T_2]} \otimes \varphi(s) ds$ onto $\frac{1}{T_1} \m\mycorner_{[0,T_1]} \otimes \varphi(s) ds$ by simply linearly scaling in the horizontal axis. Clearly $P$ is $1$-Lipschitz and its inverse is $T_2/T_1$-Lipschitz. The assertion then follows by a standard isoperimetric transference principle for Lipschitz maps (see e.g.~\cite[Section 5.3]{EMilman-RoleOfConvexity}). 
The proof for $M_T$ is identical. 
\end{proof}

\section{Second order information - stability and ODI for isoperimetric profile} \label{sec:stability}

Let $(M^n,g,\mu)$ be a weighted Riemannian manifold, let $T_t : (M^n,g)  \rightarrow (M^n,g)$ denote a one-parameter smooth variation for $t \in (-\eps,\eps)$ with $T_0 = \Id$, and denote by $X := \frac{d T_t}{dt}|_{t=0}$ the associated vector-field at time $t=0$. Note that when $M$ has a boundary then $X$ must be tangential, i.e. $X(p) \in T_p \partial M$ for all $p \in \partial M$. For simplicity, we demand that $X$ be compactly supported in $M$. Let $E \subset (M^n,g)$ be a set whose boundary $\Sigma = \partial E$ is a $(n-1)$-dimensional submanifold (with possibly non-empty boundary $\Sigma \cap \partial M$). We initially assume that the density $\Psi_{\mu}$ and $\Sigma$ are of class $C^\infty_{loc}$, and then see how the well-known results below should be modified when $\Psi_{\mu}$ is only locally Lipschitz and $\Sigma$ is only $C^{1,1}_{loc}$, as in the context of our two-dimensional model slabs. Note that an isoperimetric minimizer $E$ may in general have singularities when $n \geq 8$, but when $n \leq 7$ as in our context this is not possible \cite{MorganRegularityOfMinimizers}, and so we do not treat geometric singularities in our discussion. 

Denote $V(t) = V_{\mu}(T_t(E))$ and $A(t) = A_{\mu}(T_t(E))$. If $A'(0) = 0$ for all variations for which $V'(0)=0$, $E$ is called \emph{stationary}. 
It is well-known (e.g. \cite[Lemma 2.4 and Proposition 2.7]{BarbosaDoCarmo-StabilityInRn}, \cite[Proposition 3.2]{RCBMIsopInqsForLogConvexDensities}, \cite[Appendix C]{EMilmanNeeman-GaussianMultiBubble}) that stationarity is equivalent to the existence of a Lagrange multiplier $\lambda \in \R$ so that $(V - \lambda A)'(0) = 0$ for all variations $T_t$. Geometrically, this means that $\Sigma$ must have constant weighted mean-curvature (CMC) $H_{\Sigma,\mu} \equiv \lambda$ (recall that $H_{\Sigma,\mu}$ was defined in (\ref{eq:Hmu})). Furthermore, since our variations are always tangential, stationarity also implies that wherever $\Sigma$ meets $\partial M$, it must do so perpendicularly \cite[p. 266]{Gruter}. 
Conversely, if $\Sigma$ is CMC and meets $\partial M$ perpendicularly then it is necessarily stationary (see e.g. \cite[Lemma 4.3 and Appendix C]{EMilmanNeeman-GaussianMultiBubble}).

The set $E$ is called \emph{stable} (under volume-preserving variations) if in addition to being stationary, it satisfies $(V - \lambda A)''(0) \geq 0$ for all variations so that $V'(0) = 0$. It is well-known \cite{BarbosaDoCarmo-StabilityInRn,RCBMIsopInqsForLogConvexDensities} that if $E$ is an isoperimetric minimizer then it must be stationary and stable.

Denote $u = \scalar{X,\n_{\Sigma}}$ the normal component of $X$ on $\Sigma$, where recall $\n_{\Sigma}$ denotes the outer unit-normal to $\Sigma$. Clearly $V'(0) = \int_{\Sigma} u \, d\mu^{n-1}$, where we set $\mu^{k} = \Psi_{\mu} \H^{k}$ if $\mu = \Psi_{\mu} \vol_g$. 
It turns out  that $(V - \lambda A)''(0)$ is a quadratic form which depends on $X$ only via its normal component $u$ as follows
(see \cite{BarbosaDoCarmo-StabilityInRn,RCBMIsopInqsForLogConvexDensities,Ritore-IsoperimetricBook} for the case that $\Sigma \cap \partial M = \emptyset$, \cite[Theorem 2.5]{SternbergZumbrun} or \cite[Formula (3.6)]{BayleRosales} for the general one, and also \cite{EMilmanNeeman-TripleAndQuadruple} for an extension to the multi-bubble setting):
\[
(V - \lambda A)''(0) = Q(u) := - \int_{\Sigma} (L_{Jac} u) u \, d\mu^{n-1} - \int_{\Sigma \cap \partial M} \II_{\partial M}(\n_{\Sigma},\n_{\Sigma}) u^2 d\mu^{n-2} ,
\]
where $\II_{\partial M}$ is the second fundamental form of $\partial M$ with respect to the inward pointing normal (recall that $\n_{\Sigma} \in T \partial M$ on $\partial M$), and $L_{Jac}$ is the associated Jacobi operator:
\[
L_{Jac} u := \Delta_{\Sigma,\mu} u + (\Ric_{g} + \nabla^2_g W)(\n_{\Sigma},\n_{\Sigma}) u + \norm{\II_{\Sigma}}^2 u .
\]
Here $\Psi_{\mu} = \exp(-W)$, $\nabla_g$ is the Levi-Civita on $(M^n,g)$, $\Ric_g$ is the Ricci curvature of $\nabla_g$, $\norm{\II_{\Sigma}}$ is the Hilbert-Schmidt norm of the second-fundamental form $\II_{\Sigma}$ of $\Sigma$, $\Delta_{\Sigma}$ is the Laplace-Beltrami operator on $\Sigma$ equipped with its induced metric and connection $\nabla_{\Sigma}$, and $\Delta_{\Sigma,\mu}$ is the corresponding weighted surface Laplacian: 
\[
 \Delta_{\Sigma,\mu} u := \Delta_{\Sigma} u - \scalar{\nabla_{\Sigma} W, \nabla_{\Sigma} u} .
\]
For slabs we have $\II_{\partial M} = 0$, and so stability boils down to the requirement that 
\begin{equation} \label{eq:stability}
\int_{\Sigma} u \, d\mu^{n-1} = 0 \;\; \Rightarrow \;\; Q(u) = -  \int_{\Sigma} (L_{Jac} u) u \, d\mu^{n-1} \geq 0 ,
\end{equation}
for all $u = \scalar{X,\n_{\Sigma}}$ where $X$ is a smooth compactly-supported tangential vector-field on $(M^n,g)$; we denote the family of such functions $u$ by $\scalar{C^\infty_c , \n_{\Sigma}}$. The quadratic form $Q = Q_{\Sigma}$ is called the index-form corresponding to $\Sigma$. 

\medskip
One useful interpretation of the Jacobi operator on a CMC hypersurface $\Sigma$ is that it is captures the first variation of the (constant) weighted mean-curvature in the normal direction (e.g. \cite[Remark 5.8]{EMilmanNeeman-TripleAndQuadruple}):
\[
L_{Jac} u = - \left . \frac{d}{dt}\right |_{t=0} H_{T_t(\Sigma),\mu} .
\]
In particular, at a point $p \in \Sigma$ where $(M^n,g)$ is locally isometric to flat space $(\R^n,\abs{\cdot}^2)$, by testing the locally constant vector-field $X = \theta$ which generates a translation $T_t(\Sigma) = \Sigma + t \theta$ in a neighborhood of $p$, since the unweighted mean-curvature $H_{\Sigma + t \theta} = H_{\Sigma}$ remains fixed, we only obtain a contribution from the variation of $-\scalar{\nabla W(p + t \theta),\n_{\Sigma + t \theta}(p + t \theta)} = -\scalar{\nabla W(p + t \theta),\n_{\Sigma}(p)}$ along the translation, yielding the useful:
\begin{equation} \label{eq:LJac-locally-flat}
L_{Jac} \scalar{\theta,\n_{\Sigma}} =  \nabla^2 W(\theta,\n_{\Sigma}) \;\;\; \forall \theta \in T_p M^n . 
\end{equation}
This continues to hold when $p \in \partial M$, $\theta \in T_p \partial M$, and $(M^n,g)$ is locally isometric to a half-space $(\R_+ \times \R^{n-1}, \abs{\cdot}^2)$.
\medskip

To apply these classical facts to our two-dimensional slab $S_T(\bI)$, when $W(t,s) = - \log \varphi(s)$ is only assumed locally Lipschitz and a minimizer's boundary $\Sigma$ is only of class $C^{1,1}_{loc}$, recall our assumption that $\SSS_{\bI}$ consists of isolated points and Lemma \ref{lem:transverse}, stating that $\Sigma$ cannot be a zonal horizontal line and that it intersects each zonal line in at most a single point. Consequently, $\nabla^2 W$ is well-defined except on a set of isolated points $S \subset \Sigma$, $\n_{\Sigma}$ and therefore $u = \scalar{X,\n_{\Sigma}}$ are smooth except at $S$ where they are only locally Lipschitz, and we interpret $L_{Jac} u$ in the distributional sense; in particular, we interpret $Q(u)$ via integration-by-parts. 

\subsection{Instability of generalized unduloids} \label{subsec:unduloids-unstable}

Determining the stability of a given $\Sigma$, even in the two-dimensional slab $S_T(\bI)$, is in general not very tractable. A precise criterion for stability was given by Koiso \cite{Koiso} (in fact, for more general two-dimensional surfaces $\Sigma$ in a three-dimensional manifold; see \cite{KoisoMiyamoto} for a specialization to the case of a generalized unduloid on a two-dimensional slab). 

Let $\Sigma$ be a generalized unduloid (i.e.~$1 > v_1 > v_0 > 0$) in $S_T(\bI)$. In that case, it is easy to check that $L_{Jac}$ is an essentially self-adjoint operator acting on $\scalar{C^\infty_c , \n_{\Sigma}}$ with vanishing Neumann boundary conditions on $\Sigma \cap \partial S_T(\bI)$, and its spectrum is discrete, consisting of a sequence of distinct eigenvalues (of multiplicity $1$) $\{\lambda^N_i\}_{i=1,2,\ldots}$ increasing to infinity. 
Recalling (\ref{eq:stability}) and considering a linear combination of the first two eigenfunctions, it is immediate to see that a sufficient condition for instability is the negativity of the second eigenvalue. An explicit criterion for this was found by Pedrosa--Ritor\'e \cite{PedrosaRitore-Products}; for completeness, we mention it here: Given $v_0,v_1$, recall Definition \ref{def:ell} of $\ell_{v_0,v_1}$, and consider a variation $\eps \mapsto v_i(\eps)$ so that the chord $\ell_{v_0(\eps),v_1(\eps)} = \ell_{v_0,v_1} + \eps$ moves up vertically while preserving its slope $\lambda$. Then $\frac{d}{d\eps} T(v_0(\eps),v_1(\eps)) > 0$ (where the unduloid's half-period $T(v_0,v_1)$ was defined in (\ref{eq:T-formula})) implies that $\lambda_2^N < 0$ and thus that $\Sigma$ is unstable. 

When $\bI(\bar v) = c \bar v^{\frac{n-2}{n-1}}$ for $\bar v \in [0,v_e]$, even when $n=3$, an explicit computation of $\frac{d}{d\eps} T(v_0(\eps),v_1(\eps))$ involves several elliptic integrals, and so is highly intractable. However, Pedrosa and Ritor\'e were able to express this as a contour integral over a Riemann surface, and after some clever manipulations, control its sign when $3 \leq n \leq 8$, thereby showing that all generalized unduloids with $v_e > v_1 > v_0 > 0$ are always unstable in that case (yielding the instability of unduloids on $\R^{n-1} \times [0,T]$ for those dimensions). Further numerical evidence for instability in a certain range of $v_i$'s and also stability in the complementary range when $n=9$ was obtained in \cite{KoisoMiyamoto}. When $n\geq 10$, it was shown by Pedrosa--Ritor\'e that there are isoperimetric minimizers which are (stable) generalized unduloids.

When $\bI = \I_\gamma$, we are not aware of any stability or instability results for generalized Gaussian unduloids.

\subsection{Instability of horizontal lines}

\begin{lemma} \label{lem:horizontal-stable}
Let $\Sigma = [0,T] \times \{ \bar s\}$ be a non-zonal horizontal line in $S_T(\bI)$. Then $\Sigma$ is stable if and only if:
\[
\sqrt{-(\bI \bI'')(\bar v)} \leq \frac{\pi}{T} ,
\]
where $\bar v = \Phi_{\bI}(\bar s)$ is the weighted volume of the set delineated by $\Sigma$. \\
In particular, if $\bI''(\bar v) < 0$, then $\Sigma$ will be unstable for $T$ large enough. 
\end{lemma}
\begin{proof}
Denote $P =  -(\log \varphi_{\bI})''(\bar s)$, and note that $P = -(\bI \bI'')(\bar v)$ in view of (\ref{eq:I''}).  
The Jacobi operator on $\Sigma$ is particularly simple:
\[
L_{Jac} u = u'' + P  u ,
\]
where we naturally parametrize $u$ on $[0,T]$. Applying vanishing Neumann boundary conditions on $[0,T]$, the
eigenvalues $\{ \lambda^N_k \}_{k=1,2,\ldots}$ of $-L_{Jac}$ are precisely given by $\lambda^N_k := (k-1)^2 \brac{\frac{\pi}{T}}^2 - P$, corresponding to the eigenfunctions $\xi_k(t) = \cos( (k-1) \frac{\pi}{T} t)$, which constitute an orthogonal basis for $L^2([0,T])$. Since $\xi_1$ is the constant function, it follows that (\ref{eq:stability}) holds iff $\lambda^N_2 \geq 0$, establishing the claim. 
\end{proof}

\subsection{Stability of vertical lines}

\begin{lemma}
Any vertical line $\Sigma$ is stable in $S_T(\bI)$.
\end{lemma}
\begin{proof}
The Jacobi operator on a vertical line $\Sigma$ is simply given by $L_{Jac} u = u'' + u' (\log \varphi_{\bI})'$, using the natural parametrization of $u$ on $M_{\bI} = (-R_{\bI},R_{\bI})$. 
Since
\[
- \int_{M_{\bI}} u L_{Jac} u \, \varphi_{\bI}(s) ds = \int_{M_{\bI}} (u')^2 \varphi_{\bI}(s) ds \geq 0 \;\;\; \forall u \in C_c^\infty(M_{\bI}),
\]
 (\ref{eq:stability}) is established (even without assuming $\int_{M_{\bI}} u \, \varphi_{\bI}(s) ds = 0$). 
\end{proof}

\subsection{ODI for isoperimetric profile}

The index-form $Q$ associated to an isoperimetric minimizer is useful not only for testing stability, but also to establish a second order ordinary differential inequality (ODI) for the isoperimetric profile $\I = \I(S_T(\bI))$. This observation has its origins in the work of Bavard--Pansu \cite{BavardPansu}, and has been further developed by Sternberg--Zumbrun \cite{SternbergZumbrun}, Kuwert \cite{Kuwert}, Bayle \cite{BayleThesis} and Bayle--Rosales \cite{BayleRosales} for normal variations; the usefulness of non-normal variations was exploited in \cite{EMilmanNeeman-GaussianMultiBubble,EMilmanNeeman-QuintupleBubble}. 
Recall that $d\sigma_T = \frac{1}{T} \varphi_{\bI}(s) \H^1(dt,ds)$ denotes the perimeter measure on $S_T(\bI)$. 
\begin{lemma} \label{lem:ODI}
Let $E_{\bar v}$ be an isoperimetric minimizer of volume $\bar v \in (0,1)$ on $S_T(\bI)$. Set $\Sigma_{\bar v} = \overline{\partial^* E_{\bar v}}$ and let $Q_{\Sigma_{\bar v}}(u)$ be the corresponding index-form. Then for any $u \in \scalar{C^\infty_c,\n_{\Sigma_{\bar v}}}$, we have:  
\begin{equation} \label{eq:ODI-gen}
\I''(\bar v) (\int_{\Sigma_{\bar v}} u \, d\sigma_T)^2 \leq Q_{\Sigma_{\bar v}}(u) = - \int_{\Sigma_{\bar v}} (L_{Jac} u) u \, d\sigma_T 
\end{equation}
in the viscosity sense (see Definition \ref{def:viscosity} below). In particular, if for all $\bar v \in (\bar v_0, \bar v_1) \subset (0,1)$, there exists $u = u_{\bar v} \in \scalar{C^\infty_c,\n_{\Sigma_{\bar v}}}$ so that $\int_{\Sigma_{\bar v}} u \, d\sigma_T \neq 0$,
\begin{equation} \label{eq:ODI-assumption}
\frac{- \int_{\Sigma_{\bar v}} (L_{Jac} u) u \, d\sigma_T \;  \int_{\Sigma_{\bar v}} d\sigma_T}{ \brac{\int_{\Sigma_{\bar v}} u \, d\sigma_T}^2 } \leq \I_m(\bar v) \I''_m(\bar v) ,
\end{equation}
and $\I(\bar v_i) = \I_m(\bar v_i)$, $i=0,1$, for some smooth function $\I_m : [\bar v_0, \bar v_1] \rightarrow \R_+$ with $\I_m'' < 0$ on $(\bar v_0, \bar v_1)$, then $\I \geq \I_m$ pointwise on $[\bar v_0,\bar v_1]$. 
\end{lemma}

\begin{definition}[Viscosity sense] \label{def:viscosity}
A continuous function $\I : [0,1] \rightarrow \R_+$ is said to satisfy the second order differential inequality $F(\I'',\I',\I) \leq 0$ (with $\frac{\partial F}{\partial \I''} > 0$) at $\bar v$ in the viscosity sense, if there exists a $C^2_{loc}([0,1])$ function $\iota$ so that $\iota(v) \geq \I(v)$ in a neighborhood of $\bar v$, $\iota(\bar v) = \I(\bar v)$, and $F(\iota'',\iota',\iota) \leq 0$ at $\bar v$ (see \cite{HPRR-PeriodicIsoperimetricProblem,EMilmanNeeman-QuintupleBubble} for further details).
\end{definition}

\begin{proof}[Proof of Lemma \ref{lem:ODI}]
For a proof of (\ref{eq:ODI-gen}), see e.g.~\cite[Theorem 2.5]{SternbergZumbrun}, 
\cite[Proof of Theorem 3.2]{BayleRosales} or \cite[Proof of Theorem 1.1 in Section 9]{EMilmanNeeman-GaussianMultiBubble}. Multiplying (\ref{eq:ODI-gen}) by $\I(\bar v) = \int_{\Sigma_{\bar v}} d\sigma_T$ and dividing by $(\int_{\Sigma_{\bar v}} u \, d\sigma_T)^2  > 0$, if (\ref{eq:ODI-assumption}) holds then we have $\I \I'' \leq \I_m \I''_m$ on $(\bar v_0,\bar v_1)$ in the viscosity sense, with $\I(\bar v_i) = \I_m(\bar v_i)$, $i=0,1$. Since $\I_m'' < 0$, an application of the maximum principle as in \cite[Proof of Theorem 1.1 in Section 9]{EMilmanNeeman-GaussianMultiBubble} or \cite[Lemma 8]{HPRR-PeriodicIsoperimetricProblem} yields the asserted $\I \geq \I_m$ on $[\bar v_0,\bar v_1]$. 
\end{proof}

\section{Three dimensional cube} \label{sec:Q3}

Given $\beta \in (0,1]$, recall our notation $\Q^3(\beta) := ([0,\beta] \times [0,1]^2, \abs{\cdot}^2, \frac{1}{\beta} \m\mycorner_{ [0,\beta] \times [0,1]^2})$ for the $3$-dimensional cube with side lengths $(\beta,1,1)$, endowed with its uniform measure. Note that the shortest edge is of length $\beta$. 

\subsection{Conjectured isoperimetric profile}

Recall that a minimizer in $\Q^3(\beta)$ is expected to be enclosed by an eighth sphere about a corner, a quarter cylinder about the short edge $[0,\beta]$, or a flat plane $[0,\beta] \times [0,1]$.

Also note that we have normalized our measure to have total mass $1$, affecting both the weighted volume and surface-area. 
Denoting by $P_k(r)$ the weighted volume ($P=V$) and surface-area ($P=A$) enclosed by a sphere ($k=3$) and cylinder ($k=2$) of radius $r$ ($r \leq \beta$ and $r \leq 1$, respectively), and expressing the corresponding surface-area as a function of volume (in the appropriate range) by $\I_m^{(k)}(\bar v)$, we compute:
\begin{center}
\begin{tabular}{cccc}
$V_3(r) = \frac{\frac{4}{3} \pi r^3}{8 \beta}$ & $A_3(r) = \frac{4 \pi r^2}{8 \beta}$ & $\I_m^{(3)}(\bar v) = \brac{\frac{9 \pi}{2 \beta}}^{\frac{1}{3}} {\bar v}^{\frac{2}{3}}$ & $\bar v \leq \frac{\pi \beta^2}{6}$ \\
$V_2(r) = \frac{\pi r^2 \beta}{4 \beta}$ & $A_2(r) = \frac{2 \pi r \beta}{4 \beta}$ & $\I_m^{(2)}(\bar v) = \sqrt{\pi \bar v}$ & $\bar v \leq \frac{\pi}{4}$ .
\end{tabular}
\end{center}
Of course the (weighted) surface-area of a flat plane $[0,\beta] \times [0,1]$ is $\I_m^{(1)}(\bar v) = 1$ for all $\bar v \in (0,1)$. Setting $\I_m := \min(\I_m^{(3)},\I_m^{(2)},\I_m^{(1)})$ on $[0,1/2]$ and extending by symmetry to the entire $[0,1]$, we obtain:
\[
\I_m(\bar v) = \begin{cases}  \brac{\frac{9 \pi}{2 \beta}}^{\frac{1}{3}} {\bar v}^{\frac{2}{3}} & \bar v \in [0,\frac{4 \pi}{81} \beta^2] \\
\sqrt{ \pi \bar v} & \bar v \in [\frac{4 \pi}{81} \beta^2, \frac{1}{\pi}] \\
1 & \bar v \in [\frac{1}{\pi},\frac{1}{2}] \\
 \end{cases} ~,~ \I_m(1-\bar v) = \I_m(\bar v).
\]
We see that a minimizer on $\Q^3(\beta)$ is expected to be enclosed by a sphere for $\bar v \in (0, \frac{4 \pi}{81} \beta^2]$, a cylinder for $\bar v \in [\frac{4 \pi}{81} \beta^2, \frac{1}{\pi}]$, and a flat plane for $\bar v \in [\frac{1}{\pi},\frac{1}{2}]$; see Figure \ref{fig:I-cube}.
Note that when $\bar v \leq 1/2$, there is no need to test the complements of the above tubular neighborhoods, since a minimizer of volume $\bar v \leq  1/2$ will necessarily have non-negative mean-curvature by the concavity and symmetry of the isoperimetric profile in conjunction with Proposition \ref{prop:lambda-I}.

\begin{figure}
    \begin{center}
        \includegraphics[scale=0.4]{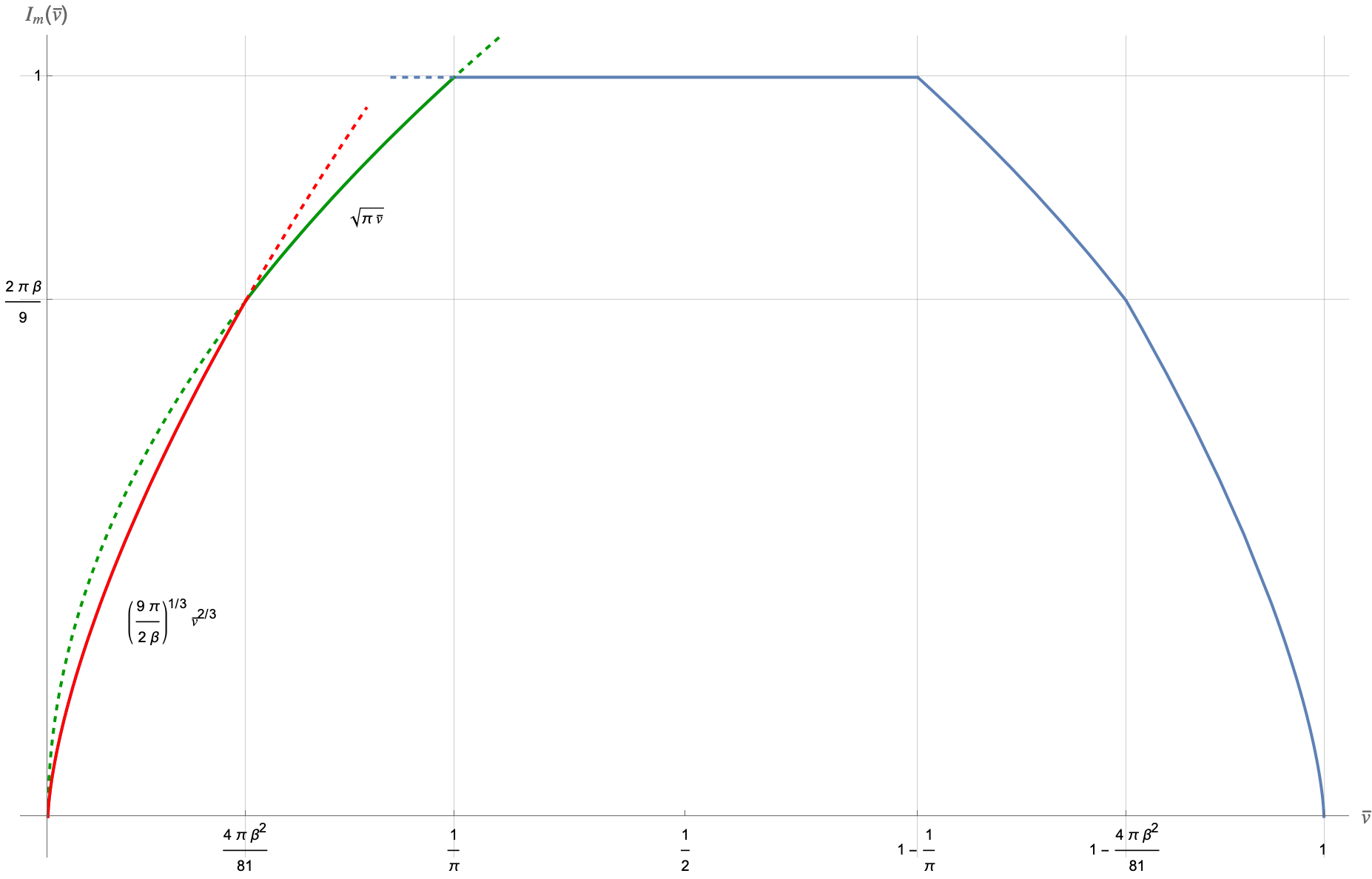}
     \end{center}
     \caption{
         \label{fig:I-cube}
         Plot of conjectured isoperimetric profile $\I_m$ of the three-dimensional cube $\Q^3(\beta)$ with side lengths $(\beta,1,1)$. The regimes when the interior of an eighth sphere about a corner and a quarter cylinder about the short edge are minimizing are highlighted in red and green, respectively. The true isoperimetric profile of the base $\Q^2$ (on $[0,1/2]$) is obtained by replacing the red curve by the dashed-green one.
     }
\end{figure}

The isoperimetric profile $\I(\Q^3(\beta))$ is thus conjectured to coincide with $\I_m$ above. 
Clearly $\I(\Q^3(\beta)) \leq \I_m$, so our goal is to show the converse inequality. 

\subsection{Unduloid Analysis} \label{subsec:Q3.2}

We would like to consider $\Q^3(\beta)$ as a slab of width $\beta$ over the base $\Q^2$. Since we assumed in Section \ref{sec:CMC} that the base is a manifold without boundary, let us pass to the equivalent flat torus $\T^3(\beta) = \R^3 / (2 \beta \Z \times 2 \Z^2)$, and consider it as a double-cover of the slab $\T^3_\beta$ of width $\beta$ over the base $\T^2 = \R^2 / (2 \Z^2)$ (all manifolds are equipped with their standard metrics and uniform probability measures). In view of Remark \ref{rem:reflect-slab}, we know that $\I(\Q^3(\beta)) = \I(\T^3(\beta)) = \I(\T^3_\beta)$, and so we will study the isoperimetric profile of $\T^3_\beta$. The isoperimetric minimizers of the two-dimensional base $\T^2$ are known to be enclosed by either a circle or two parallel horizontal or vertical lines \cite[Theorem 3.1]{Howards-BScThesis}, \cite[Section 7]{HHM-Surfaces} (or recall Remark \ref{rem:reflect-slab} and \cite{BrezisBruckstein}), and a calculation yields:
\[
\I_{\T^2}(\bar v) := \I(\T^2)(\bar v) = \min(\sqrt{\pi \bar v} , 1 , \sqrt{\pi (1-\bar v)}) .
\]
Our goal will be to calculate (at least, partially) the base-induced isoperimetric profile $\I^b_\beta(\I_{\T^2})$, or equivalently, the isoperimetric profile $\I(S_\beta(\I_{\T^2}))$ of the two-dimensional model slab
\[
 S_\beta(\I_{\T^2}) =  ([0,\beta],\abs{\cdot}^2,\frac{1}{\beta} \m\mycorner_{[0,\beta]}) \otimes ((-R_{\T^2},R_{\T^2}), \abs{\cdot}^2 , \varphi_{\T^2}(s) ds) ,
\]
where the density $\varphi_{\T^2}$ was already calculated in (\ref{eq:varphi2}) (with $R_{\T^2} = \frac{1}{2} + \frac{1}{\pi}$). 
Since $\I(\T^3_\beta) \geq \I(S_\beta(\I_{\T^2}))$ by Corollary \ref{cor:base-induced-profile} and Proposition \ref{prop:coincide}, we would like to show that $\I(S_\beta(\I_{\T^2}))(\bar v) = \I_m(\bar v)$ for some range of $\bar v$'s we can get a handle on. As explained in Section \ref{sec:CMC}, since the minimizers of $\T^2$ aren't nested, there is no expectation that $\I(\T^3_\beta)(\bar v) = \I(S_\beta(\I_{\T^2}))(\bar v)$ for all $\bar v \in (0,1)$, and so our strategy will inherently be confined to some sub-range of $\bar v$'s. As the isoperimetric profile $\I(S_\beta(\I_{\T^2}))$ is symmetric around $1/2$, it is enough to consider $\bar v \in (0,1/2)$. 

For example, since a zonal horizontal line cannot be minimizing in $S_\beta(\I_{\T^2})$ by Lemma \ref{lem:transverse}, we necessarily have $\I(S_\beta(\I_{\T^2}))(\bar v) < \I_{\T^2}(\bar v) = \I_m(\bar v) = 1$ for $\bar v = \frac{1}{\pi}$ (where $\I_{\T^2}$ is non-differentiable), and by continuity of both sides (recall Corollary \ref{cor:Ib-concave}), 
a strict inequality must also hold for all $\bar v$ in a neighborhood of $\frac{1}{\pi}$, and we conclude that any hope to show that $\I(S_\beta(\I_{\T^2}))(\bar v) = \I_m(\bar v)$ must be confined to $\bar v$ well separated from $\frac{1}{\pi}$. 

On the other hand, we can easily establish that:
\begin{equation} \label{eq:at-half}
\I(\Q^3(\beta))(1/2) = \I(\T^3_\beta)(1/2) = \I(S_\beta(\I_{\T^2}))(1/2) = 1 = \I_m(1/2) .
\end{equation}
As explained in the Introduction, the equality between the left and right hand sides is well-known, and follows by constructing a Lipschitz map pushing forward the Gaussian measure $(\R^3,\abs{\cdot}^2,\gamma^3)$ onto the uniform measure on $\Q^3(\beta)$. A slightly less obvious observation is that $\I(S_\beta(\I_{\T^2}))(1/2)= 1$. This follows by Corollary \ref{cor:at-half}, which applies since $\I_{\T^2} = \I(\Q^2) \geq \I_{\gamma}/\I_{\gamma}(1/2)$ by the above contraction argument from $(\R^2,\abs{\cdot}^2,\gamma^2)$ onto $\Q^2$.

\medskip

Our first main observation is the following:

\begin{proposition} \label{prop:main-computation}
For all $\beta > 0$, if $\Sigma$ is a generalized unduloid in $S_\beta(\I_{\T^2})$ with parameters $0 < v_0 < v_1 < 1$, having weighted mean-curvature $\lambda \geq 0.8$ and enclosing weighted volume $\bar v \leq \frac{4 \pi}{81}$, then necessarily $v_1 \leq \frac{1}{\pi}$. 
\end{proposition}

Proposition \ref{prop:main-computation} boils down to estimating the sign of a certain complicated (yet explicit) function involving elliptic integrals, and we do not know how to establish it without relying on numeric computation. Consequently, its proof is deferred to the Appendix. We will in addition require the following analogous proposition regarding one-sided generalized unduloids. Here the parameter space is only one-dimensional, and the expressions do not involve elliptic integrals but rather algebraic and trigonometric functions, and so we provide a proof here to give the reader a taste of what this type of argument entails.

\begin{proposition} \label{prop:one-sided-volume}
If $\Sigma$ is a one-sided generalized unduloid in $S_\beta(\I_{\T^2})$ with parameters $0 = v_0 < v_1 < 1$, having weighted mean-curvature $\lambda \geq 0.6$ and enclosing weighted volume $\bar v < \frac{v_{\min}}{\beta}$ for an explicit $v_{\min}\simeq 0.120582$, then necessarily $v_1 < \frac{1}{\pi}$ and $\Sigma$ is a quarter circle of radius $\frac{2}{\lambda}$ centered at the bottom-left corner of $S_\beta(\I_{\T^2})$.
\end{proposition}
\begin{remark}
As will be apparent from the proof, both restrictions above are necessary: if $\lambda \rightarrow 0$, or equivalently, $v_1 \rightarrow 1$, one may check that the enclosed weighted volume $\bar v$ tends to zero;  and if $\beta > 0$ is not too small, there are one-sided generalized unduloids enclosing weighted volume $\frac{v_{\min}}{\beta}$ with $v_1 > \frac{1}{\pi}$. 
\end{remark}
\begin{proof}[Proof of Proposition \ref{prop:one-sided-volume}]
Recalling Theorem \ref{thm:model-slab-main}, $\Sigma$ is the graph of $f$ over $[0,T(0,v_1)]$, where $T(0,v_1)$ is given by (\ref{eq:T-formula}). Also recall that $\ell(v) = \lambda v$ (with $c=0$), where $\lambda = \I_{\T^2}(v_1)/v_1 \geq 0$ is the constant weighted mean-curvature of $\Sigma$. Note that by Corollary \ref{cor:perp}, $\Sigma$ meets the bottom of the slab $S_\beta(\I_{\T^2})$ perpendicularly. See Figure \ref{fig:all-one-sided-unduloids}.
 
\begin{figure}
    \begin{center}
        \includegraphics[scale=0.4]{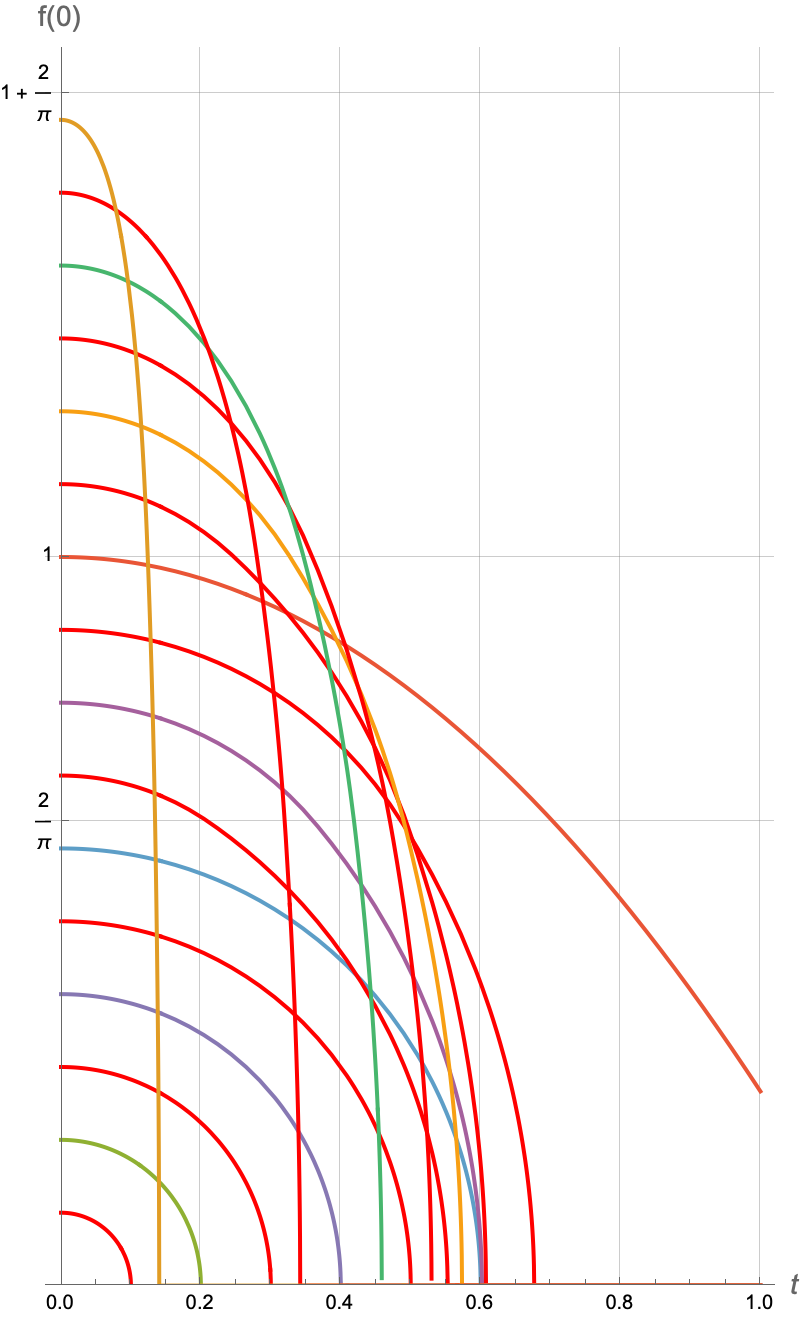}
     \end{center}
     \caption{
         \label{fig:all-one-sided-unduloids}
         Various one-sided generalized unduloids $\{f(t)\}$ with different initial values $f(0) = s_1 = \Phi_{\T^2}^{-1}(v_1)$. For convenience, the $s$-axis is parametrized on $(0,1+\frac{2}{\pi})$ instead of $(-\frac{1}{2} - \frac{1}{\pi},\frac{1}{2} + \frac{1}{\pi})$. Note that when $s_1 \leq \frac{2}{\pi}$ (equivalently, $v_1 \leq \frac{1}{\pi}$), $\Sigma$ is a quarter circle in the model slab, corresponding to an eighth sphere in $\Q^3(\beta)$. When $s_1 > \frac{2}{\pi}$ the one-sided unduloids are no longer nested (and some require time $> \beta$ to reach the bottom).      }
\end{figure}

It will be more convenient to parametrize things according to $\xi = 1/\lambda$; since we assume $\lambda \geq 0.6$ then $\xi \leq 5/3$. We write $v_1 = v_1(\xi)$ (which is determined uniquely since $\I_{\T^2}$ is strictly concave at the origin), and note that $v_1(\xi)$ is strictly increasing. Denote $V(\xi) = V(0,v_1(\xi))$, the weighted volume enclosed by $\Sigma$, where $V(v_0,v_1)$ is given by (\ref{eq:V-formula}). 

When $v_1 \in (0,\frac{1}{\pi}]$, or equivalently $\xi \in (0,\frac{1}{\pi}]$, note that $\Sigma$ is precisely a quarter circle in the model slab $S_{\beta}(\I_{\T^2})$ of radius $\frac{2}{\lambda} = 2 \xi$, corresponding to an eighth sphere around a corner of $\Q^3(\beta)$ having mean-curvature $\lambda$. Indeed, recalling (\ref{eq:CMC2.5}) and using that $v = \Phi_{\T^2}(f) = \frac{\pi}{4} f^2$ whenever $v \leq \frac{1}{\pi}$, we know that $f$ satisfies the following ODE when $v_1 \in (0,\frac{1}{\pi}]$:
\[
\frac{df}{dt} = - \sqrt{(\I_{\T^2}(v)/\ell(v))^2 - 1} = - \sqrt{ \frac{\pi}{\lambda^2 v} - 1} = - \frac{\sqrt{\brac{\frac{2}{\lambda}}^2 - f^2}}{f} .
\]
Denoting $g = \brac{\frac{2}{\lambda}}^2 - f^2$, we see that $\frac{dg}{dt} = 2 \sqrt{g}$. Since $g$ is strictly increasing on $[0,T(0,v_1)]$, it follows that $g(t) = (t+t_0)^2$, and as $f'(0) = 0$, we deduce that $t_0=0$ and $f(t) = \sqrt{ \brac{\frac{2}{\lambda}}^2 - t^2}$. Consequently, $\Sigma$ encloses (weighted) volume $V(\xi) = \frac{1}{8 \beta} \frac{4}{3} \pi (2 \xi)^3 = \frac{4 \pi}{3 \beta} \xi^3$; this can also be verified by direct computation using (\ref{eq:V-formula}) and $v_1(\xi) = \pi \xi^2$:
\[
V(\xi) = \frac{1}{\beta} \int_0^{v_1(\xi)} \frac{v dv}{\sqrt{\pi^2 \xi^2 - \pi v}} = - \left . \frac{2}{3 \beta} \sqrt{\xi^2 - v/\pi} (v + 2 \pi \xi^2) \right |_0^{\pi \xi^2} = \frac{4 \pi}{3 \beta} \xi^3 =: \frac{1}{\beta} F_1(\xi) . 
\]

When $v_1 \in [\frac{1}{\pi}, 1 - \frac{1}{\pi}]$, or equivalently, $\xi \in  [\frac{1}{\pi}, 1 - \frac{1}{\pi}]$, $v_1(\xi) = \xi$ and (\ref{eq:V-formula}) yields:
\begin{align*}
V(\xi) &  = \frac{1}{\beta} \brac{\int_0^{\frac{1}{\pi}} \frac{v dv}{\sqrt{\pi^2 \xi^2 - \pi v}} + \int_{\frac{1}{\pi}}^{\xi} \frac{v^2 dv}{\sqrt{\xi^2 - v^2}} } \\
& = \frac{1}{\beta} \brac{ \frac{4}{3} \pi \xi^3 - \brac{\frac{4}{3} \xi^2 + \frac{1}{6 \pi^2}} \sqrt{\pi^2 \xi^2 - 1} + \frac{\xi^2}{2} \sec^{-1}(\pi \xi) } =: \frac{1}{\beta} F_2(\xi) .
\end{align*}

Finally, when $v_1 \in [1 - \frac{1}{\pi},1)$, or equivalently, $\xi \in [1 - \frac{1}{\pi},\infty)$, we use $\I_{\T^2} \leq 1$ to lower bound the integral over $v \in [1-\frac{1}{\pi},v_1]$, yielding:
\[
V(\xi) \geq \frac{1}{\beta} F_2(\xi) - \frac{1}{\beta}  \int_{v_1(\xi)}^{\xi} \frac{v^2 dv}{\sqrt{\xi^2 - v^2}} . 
\]
Since $v_1(\xi)$ is the intersection point in $[1-1/\pi , 1)$ of $v/\xi$ with the concave $\sqrt{\pi (1 - v)}$, we can lower bound it by the intersection with the latter function's chord $\pi (1-v)$ between $1 - \frac{1}{\pi}$ and $1$, namely $v_1(\xi) \geq \frac{\pi}{\pi + 1/\xi}$. Using this lower bound above and integrating, we obtain for all $\xi \in [1 - \frac{1}{\pi},\infty)$:
\[
V(\xi) \geq \frac{1}{\beta} F_2(\xi) - \frac{1}{\beta} \frac{\xi^2}{2}  \brac{ \frac{\sqrt{(1/\pi+\xi)^2 - 1}}{(1/\pi+\xi)^2} +  \sec^{-1}\brac{1/\pi +\xi}}   =: \frac{1}{\beta} F_3(\xi) .
\]
We conclude that $V(\xi) \geq \frac{1}{\beta} F(\xi)$ with equality when $\xi \in (0,1-1/\pi]$, where $F(\xi)$ is defined as $F_1(\xi)$ if $\xi \in (0,1/\pi]$, $F_2(\xi)$ if $\xi \in [1/\pi , 1 - 1/\pi]$, and $F_3(\xi)$ if $\xi \in [1-1/\pi,\infty)$; see Figure \ref{fig:one-sided-unduloid-volume}.

\begin{figure}
    \begin{center}
        \includegraphics[scale=0.5]{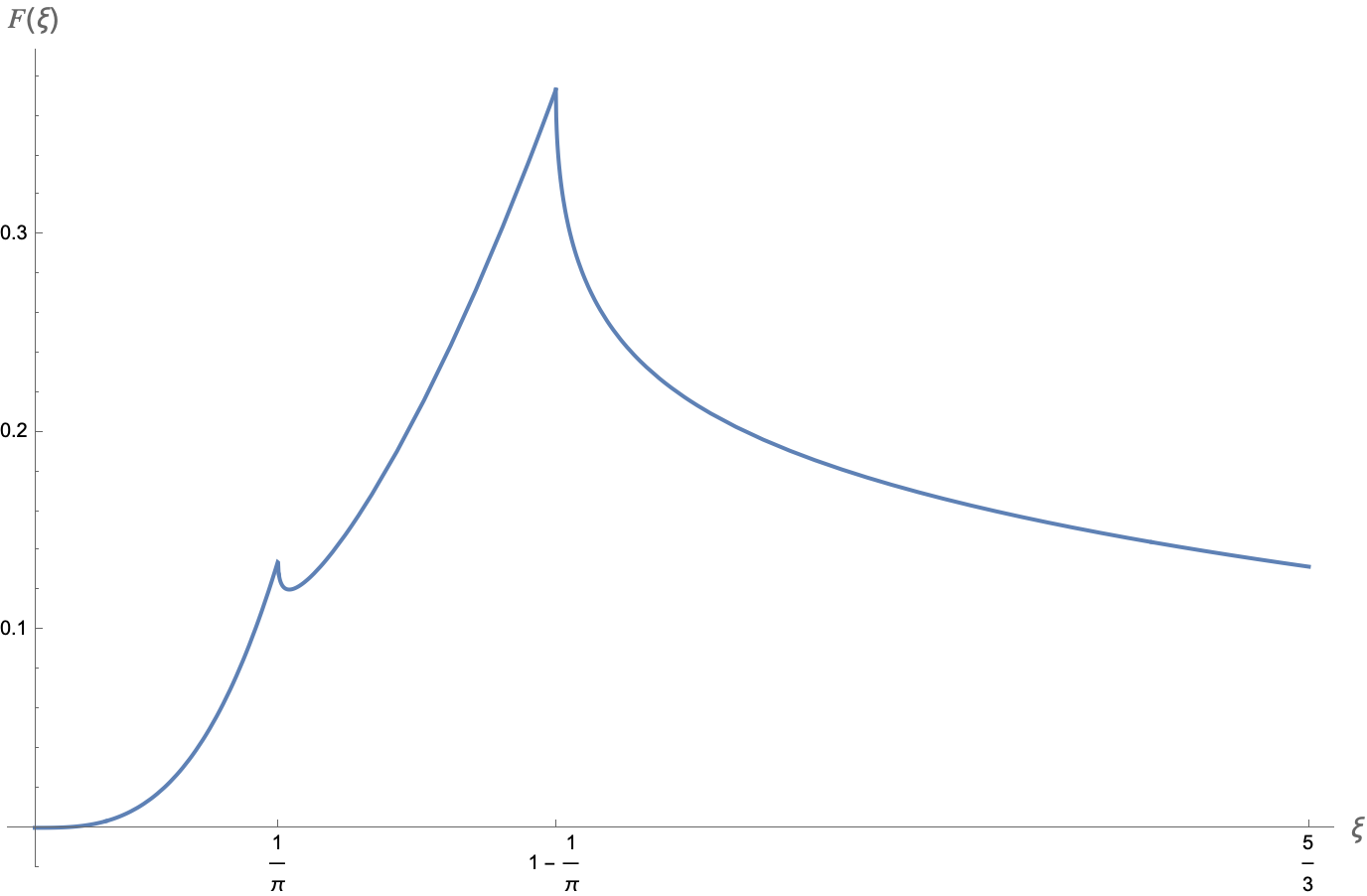}
     \end{center}
     \caption{
         \label{fig:one-sided-unduloid-volume}
         Plot of $F(\xi)$, lowerbounding $\beta V(\xi)$. 
     }
\end{figure} 

Direct calculation confirms that:
\[
F_2''(\xi) - \sec^{-1}(\pi \xi) = 8 R(\pi \xi) ~,~ R(x) = x - \frac{8 x^4 - 13 x^2 + 4}{(x^2-1)^{3/2}} ,
\]
and it is elementary to check that $R(x) > 0$ for all $x > 1$. Consequently, we see that $F_2$ is strictly convex on $[1/\pi,1-1/\pi]$, and as it first decreases and then increases (see Figure \ref{fig:one-sided-unduloid-volume}), it attains a unique minimum in that interval at $\xi_0$ where $F_2'(\xi_0) = 0$. Direct calculation of $F_2'(\xi_0)$ verifies that $x_0 = \pi \xi_0 > 1$ is the unique root of the equation:
\[
4 x + \sec^{-1}(x) = \frac{4 x^2 - 3}{\sqrt{x^2 - 1}} .
\] 
Numerically, this occurs at $x_0 \simeq 1.046172$ and thus $\xi_0 \simeq 0.333007$. We deduce that 
\[
v_{\min} := \beta \min_{\xi \in [1/\pi , 1-1/\pi]} V(\xi) = \beta V(\xi_0) = F_2(\xi_0) \simeq 0.120582 . 
\]

It remains to treat the range $\xi \in (1 - \frac{1}{\pi},5/3]$. It is not very hard to check that $F_3(\xi)$ is strictly decreasing in that range (we omit the verification), and so it is enough to numerically verify that $0.132149 \simeq F_3(5/3) > v_{\min}$ to conclude that $\min_{\xi \in [1/\pi , 5/3]} V(\xi) = \frac{v_{\min}}{\beta}$. It follows that if $\bar v < \frac{v_{\min}}{\beta}$ then necessarily $\xi < \frac{1}{\pi}$, or equivalently, $v_1 < \frac{1}{\pi}$, as asserted. 
\end{proof}

Assuming the validity of Proposition \ref{prop:main-computation}, we can now easily establish parts (\ref{it:Q3-1}) and (\ref{it:Q3-2}) of Theorem \ref{thm:Q3-main} when $\bar v \leq \min( \frac{4 \pi}{81} , \frac{v_{\min}}{\beta})$ as follows. 

\begin{proof}[Proof of Theorem \ref{thm:Q3-main}, parts (\ref{it:Q3-1}) and (\ref{it:Q3-2}) when $\bar v \leq \min(\frac{4 \pi}{81} , \frac{v_{\min}}{\beta})$.]
Let $\Sigma$ be a downward monotone isoperimetric minimizer in $S_\beta(\I_{\T^2})$ of weighted volume $\bar v \in (0, \frac{4 \pi}{81}]$. Let $0 \leq v_0 \leq v_1 \leq 1$ be the associated parameters to $\Sigma$ from Lemma \ref{lem:chord}. 
According to Theorem \ref{thm:model-slab-main}, $\Sigma$ is either a vertical line, a horizontal line, a generalized unduloid or a one-sided generalized unduloid, according to the values of $v_0,v_1$. 
A vertical line has weighted area $\frac{1}{\beta}$, while a horizontal line has weighted area $\I_{\T^2}(\bar v) < 1 \leq \frac{1}{\beta}$ (as $\bar v < \frac{1}{\pi}$), ruling out the former as a minimizer. 

Let us now lower bound the weighted mean-curvature $\lambda$ of $\Sigma$. Recall from Corollary \ref{cor:Ib-concave} that $\I_b := \I(S_\beta(\I_{\T^2}))$ is concave, and that by Proposition \ref{prop:lambda-I}
\[
\lambda \geq \I^{',+}_b(\bar v)  ,
\]
where $\I^{',+}_b(\bar v)$ denotes the right-derivative of $\I_b$ at $\bar v$ (which always exists by concavity). Concavity of $\I_b$ and the assumption that $\bar v \leq \frac{4 \pi}{81}$ implies that:
\[
\I^{',+}_b(\bar v) \geq \frac{\I_b(\frac{1}{2}) - \I_b(\bar v)}{\frac{1}{2} - \bar v} \geq \frac{\I_b(\frac{1}{2}) - \I_b(\frac{4 \pi}{81})}{\frac{1}{2} - \frac{4 \pi}{81}} .
\]
Since $\I_b(\frac{1}{2}) = 1$ by (\ref{eq:at-half}) and $\I_b(\frac{4 \pi}{81}) \leq\I_m(\frac{4 \pi}{81}) = \frac{2 \pi}{9}$, we conclude that:
\[
\lambda \geq \frac{1 - \frac{2 \pi}{9}}{\frac{1}{2} - \frac{4 \pi}{81}} \simeq  0.87533 > 0.8 > 0.6 . 
\]

Consequently, if $\Sigma$ is a generalized unduloid, then necessarily $0 < v_0 < v_1 \leq \frac{1}{\pi}$ by Proposition \ref{prop:main-computation}. In that range, the density $\varphi_{\T^2}(s)$ is linear, corresponding to the perimeter of a sphere of radius $s$ in $\R^2$, and so by the results of Pedrosa and Ritor\'e \cite[Proposition 3.2]{PedrosaRitore-Products} already mentioned in Subsection \ref{subsec:unduloids-unstable}, the generalized unduloid $\Sigma$ is unstable, and thus cannot be a minimizer. 
On the other hand, if we assume in addition that $\bar v < \frac{v_{\min}}{\beta}$ as in Proposition \ref{prop:one-sided-volume}, we are ensured that if $\Sigma$ is a one-sided generalized unduloid then it is necessarily a quarter circle around the bottom left corner of $S_\beta(\I_{\T^2})$.

Summarizing, we have shown that if $\bar v \leq \frac{4 \pi}{81}$ and $\bar v < \frac{v_{\min}}{\beta}$, then $\Sigma$ is either a horizontal line (corresponding to the case that a cylinder in $\T^3(\beta)$ is minimizing), or a quarter circle (corresponding to the case that a sphere in $\T^3(\beta)$ is minimizing). The transition from the case when spheres have less perimeter than cylinders for a given volume in $T^3(\beta)$ occurs exactly at $\frac{4 \pi}{81} \beta^2$, concluding the proof of parts (1) and (2) of Theorem \ref{thm:Q3-main} when $\bar v < \frac{v_{\min}}{\beta}$. The case when $\bar v = \frac{v_{\min}}{\beta}$ follows by continuity of the surface area of cylinders as a function of their volume and the continuity of the isoperimetric profile $\I(\Q^3(\beta))$. 
\end{proof}

\subsection{ODE Analysis} \label{subsec:Q3.3}

To obtain the other cases of Theorem \ref{thm:Q3-main} as well as Theorem \ref{thm:T3-ODE}, we will employ an ODE argument for the isoperimetric profile $\I = \I(\T^3(\beta)) = \I(\Q^3(\beta))$. Such an argument was already used by Hauswirth--P\'erez--Romon--Ros in \cite{HPRR-PeriodicIsoperimetricProblem} to study the doubly periodic isoperimetric problem on $\T^2(\beta) \times \R$. Theorem \ref{thm:T3-ODE} is in essence already implicitly contained in \cite{HPRR-PeriodicIsoperimetricProblem}, building upon the work of Ritor\'e--Ros \cite{RitoreRos-RP3} on $3$-dimensional manifolds, but we provide a proof for completeness. 

\medskip

It was shown in \cite{RitoreRos-RP3} (see \cite[Theorems 5 and 7]{HPRR-PeriodicIsoperimetricProblem}) that if $E$ is an isoperimetric minimizer in an orientable, flat $3$-manifold $M^3$, bounded by a closed surface $\Sigma = \overline{\partial^* E}$, then either $\Sigma$ is the disjoint union of two parallel totally geodesic (flat) $2$-tori, or else $\Sigma$ is a connected oriented surface of genus $g = g(\Sigma)$ so that:
\begin{itemize}
\item If $g=0$ then necessarily $\Sigma$ is a round sphere. 
\item If $g=1$ then necessarily $\Sigma$ is a flat torus, obtained as the quotient of either a plane or a circular cylinder. 
\end{itemize}
In fact, it was shown in \cite{Ritore-Genus4} that $g \leq 4$ (and that the case $g=4$ can only occur if $\Sigma$ is a minimal surface), and this was subsequently improved in \cite{Ros-StableGenus3} to $g \leq 3$. 
For a self-contained presentation with the most recent state-of-the-art information on $\Sigma$, we refer to Sections 6.3 and 6.5, and in particular, Subsection 6.5.2, in the excellent monograph by Ritor\'e \cite{Ritore-IsoperimetricBook} (see Theorem 6.49, as well as the more general Theorems 6.17 and 6.33). 

\medskip

Applying this to the compact $M^3 = \T^3(\beta)$, note that when the minimizing surface $\Sigma$ is a flat connected torus, then necessarily its lift to $\R^3$ would be a cylinder around the lift of a shortest closed geodesic in $\T^3(\beta)$. Indeed, its lift cannot be a plane (which bounds an entire half-plane), while a cylinder (of radius $r$) around a non-shortest closed geodesic (of length $\ell$) would have strictly greater surface area $A = 2 \pi r \ell$ than its counterpart around a shortest closed geodesic enclosing the same volume $V = \pi \ell r^2$. We summarize these results as follows:

\begin{proposition}[Ritor\'e--Ros] \label{prop:genus}
Let $\Sigma = \overline{\partial^* E}$ denote the boundary of a minimizer $E$ in $T^3(\beta)$ of (weighted) volume $\bar v \in (0,1)$. If $g(\Sigma) \leq 1$ then necessarily $\Sigma$ is a round sphere, a round cylinder around a shortest closed geodesic or the disjoint union of two parallel totally geodesic tori $\T^2(\beta)$, and hence $\I(\bar v) = \I_m(\bar v)$. In particular, if $\I(\bar v) < \I_m(\bar v)$ then necessarily $g(\Sigma) \geq 2$. 
\end{proposition}

Denote $p_0 = 0$, $p_1 = \frac{4\pi}{81} \beta^2$, $p_2 = \frac{1}{\pi}$, $p_3 = 1 - p_2$, $p_4 = 1 - p_1$, $p_5 = 1$, the points where $\I_m$ is non-differentiable. We denote by $g_i$ the genus of the conjectured minimizers in the range $v \in (p_i,p_{i+1})$, namely $g_0 = 0$, $g_1 = 1$, $g_2 = 1$, $g_3=1$, $g_4=0$, and set $\chi_i = 2 - 2 g_i$ to denote the corresponding Euler characteristics. The Euler characteristic of a closed surface $\Sigma$ is denoted by $\chi(\Sigma) = 2 - 2 g(\Sigma)$. The following was observed in \cite[Theorem 9]{HPRR-PeriodicIsoperimetricProblem} as a consequence of the Gauss-Bonnet theorem (recall that we are using the uniform measure on $\T^3(\beta)$, yielding a factor of $8 \beta$ relative to the formulation in \cite{HPRR-PeriodicIsoperimetricProblem}):
\begin{proposition}[Hauswirth--P\'erez--Romon--Ros] \label{prop:ODI}
For all $\bar v \in (p_i , p_{i+1})$, the conjectured profile $\I_m$ satisfies:
\[
\I_m^2 \I_m'' + \I_m (\I'_m)^2 - \frac{\pi}{2 \beta} \chi_i = 0 .
\]
If $\Sigma$ is the boundary of an isoperimetric minimizer of volume $\bar v \in (0,1)$, then the actual profile $\I$ satisfies at $\bar v$:
\begin{equation} \label{eq:prop-ODI}
\I^2 \I'' + \I (\I')^2 - \frac{\pi}{2 \beta} \chi(\Sigma) \leq 0 
\end{equation}
in the viscosity sense (recall Definition \ref{def:viscosity}). 
\end{proposition}
In view of all of the above, a simple application of the maximum principle as in \cite[Lemma 8]{HPRR-PeriodicIsoperimetricProblem} yields the following useful proposition; for completeness, we sketch a proof. 

\begin{proposition} \label{prop:max-principle}
Let $p_i \leq \bar v_1 \leq \bar v_2 \leq p_{i+1}$, where $\{p_i\}$ are as above. If $\I(\bar v_j) = \I_m(\bar v_j)$ for $j=1,2$, then necessarily $\I(\bar v) = \I_m(\bar v)$ for all $\bar v \in [\bar v_1,\bar v_2]$. Moreover, any minimizer of volume $\bar v \in (\bar v_1 , \bar v_2)$ must be enclosed by a round sphere if $i \in \{0,4\}$, a round cylinder about a shortest closed geodesic if $i \in \{1,3\}$, or two parallel totally geodesic tori if $i=2$. 
\end{proposition}
\begin{proof}[Sketch of proof]
Recall that $\I$ is continuous on $[0,1]$. Assume that $\min_{\bar v \in [v_1,v_2]} \I(\bar v) - \I_m(\bar v) < 0$. Since $\I(\bar v_j) = \I_m(\bar v_j)$, $j=1,2$, the minimum must be attained at some $\bar v \in (\bar v_1,\bar v_2)$. Let $\Sigma$ be the boundary of a minimizer of volume $\bar v$. Since $\I(\bar v) < \I_m(\bar v)$, we deduce from Proposition \ref{prop:genus} that $g(\Sigma) \geq 2$, and hence $\chi(\Sigma) \leq -2 < 0 \leq \chi_i$. 

By Proposition \ref{prop:ODI}, $\I$ satisfies (\ref{eq:prop-ODI}) at $\bar v$ in the viscosity sense. 
By replacing $\I$ by the function $\iota$ from Definition \ref{def:viscosity} if necessary, we may assume in the ensuing discussion that $\I$ is in $C^2(N_{\bar v})$ for some neighborhood $N_{\bar v}$ of $\bar v$ contained in $(\bar v_1,\bar v_2) \subset (p_i, p_{i+1})$, as $\iota - \I_m$ still attains its minimum on $N_{\bar v}$ at $\bar v$. 
As $\bar v \in (p_i , p_{i+1})$ is a minimum point, 
we have $\I'(\bar v) = \I_m'(\bar v)$ and $\I''(\bar v) \geq \I_m''(\bar v)$. Applying Proposition \ref{prop:ODI}, we deduce:
\[
\I''(\bar v) + \frac{\I'(\bar v)^2}{\I(\bar v)} \leq \frac{\frac{\pi}{2 \beta} \chi(\Sigma)}{\I(\bar v)^2} < \frac{\frac{\pi}{2 \beta} \chi_i}{\I_m(\bar v)^2} = \I_m''(\bar v) + \frac{\I_m'(\bar v)^2}{\I_m(\bar v)} \leq \I''(\bar v) + \frac{\I'(\bar v)^2}{\I(\bar v)} ,
\]
yielding a contradiction. Consequently $\I \geq \I_m$ on $[\bar v_1 , \bar v_2]$, but we also have $\I \leq \I_m$ trivially, yielding $\I = \I_m$ on $[\bar v_1 , \bar v_2]$. It then follows from Proposition \ref{prop:ODI} that $\chi(\Sigma) \geq \chi_i \geq 0$ for the boundary $\Sigma$ of any minimizer of volume $\bar v \in (\bar v_1 , \bar v_2)$. Therefore $g(\Sigma) \leq 1$, and we conclude by Proposition \ref{prop:genus} that it must be one of the three conjectured minimizers according to the value of $i \in \{0,\ldots,4\}$.
\end{proof}

Theorem \ref{thm:T3-ODE} now follows immediately.
\begin{proof}[Proof of Theorem \ref{thm:T3-ODE}]
Define $V_s$, $V_c$ and $V_p$ as the subsets of all $\bar v \in (0,1/2]$ so that there exists an isoperimetric minimizer $E$ in $\T^3(\beta)$ of (weighted) volume $\bar v$ enclosed by a round sphere, a round cylinder and two parallel totally geodesic tori, respectively. Denote $v_s := \sup V_s$, $v_{c-} := \inf V_c$, $v_{c+} := \sup V_s$ and $v_p := \inf V_p$. As mentioned in the Introduction, it is known that spheres and parallel tori are minimizing for $\bar v \in (0,\eps_s]$ and $\bar v \in [1/2-\eps_p,1/2]$, and hence $v_s > 0$ and $v_p < 1/2$; however, it may be that $V_c$ is empty,  in which case we set $v_{c-} = 1/\pi$ and $v_{c+} = \frac{4 \pi}{81} \beta^2$ (as this renders the assertion of Theorem \ref{thm:T3-ODE} regarding cylinders meaningless). Recalling the corresponding ranges of volumes where each of the candidates outperforms the other types, we have $0 < v_s \leq \frac{4 \pi}{81} \beta^2 \leq v_{c-} , v_{c+} \leq \frac{1}{\pi} \leq v_p < 1/2$. 

Compactness of the spaces of minimizers of the above three forms, continuity of their surface areas as a function of their volumes, and continuity of the isoperimetric profile $\I$ together ensure that the supremum and infimum in the above definitions of $v_s$ and $v_p$ are attained, and the same holds for $v_{c-},v_{c+}$ assuming that $V_c$ is non-empty.  

Proposition \ref{prop:max-principle} then implies that $V_s = (0,v_s]$, $V_c$ is either empty or coincides with $[v_{c-} , v_{c+}]$, and $V_p = [v_p,1/2]$; furthermore, in the relative interiors of $V_s$, $V_c$ and $V_p$ in $(0,1/2]$, a minimizer must be enclosed by a sphere, cylinder or parallel tori. 
Note that to obtain the latter for $V_p$, we should apply Proposition \ref{prop:max-principle} on the interval $[v_p , 1-v_p] \subset [p_2,p_3]$. 
This concludes the proof of the first assertion. 

Finally, if $\I(\bar v) = \I_m(\bar v)$ for both $\bar v \in \{\frac{4 \pi}{81} \beta^2,\frac{1}{\pi}\}$, then necessarily $v_s = v_{c-} = \frac{4 \pi}{81} \beta^2$ and $v_{c+} = v_p = \frac{1}{\pi}$, confirming that $\I \equiv \I_m$ on the entire $(0,1/2]$, and thus on the entire $(0,1)$ by symmetry. This establishes the second assertion and concludes the proof. 
\end{proof}

It remains to establish the final assertions of Theorem \ref{thm:Q3-main}. We proceed employing the notation from Theorem \ref{thm:T3-ODE} and its proof. 

\begin{proof}[Proof of part (\ref{it:Q3-2}) when $\bar v > \frac{4\pi}{81}$ and part (\ref{it:Q3-3}) of Theorem \ref{thm:Q3-main}]
Let us assume that $v_{c-} < v_{c+}$ (so in particular, $V_c \neq \emptyset$); by part (\ref{it:Q3-2}) this is guaranteed to be the case whenever $\beta < 0.919431$. If $v_{c+} = \frac{1}{\pi}$ then necessarily $v_p = \frac{1}{\pi}$ and $\I \equiv 1$ on $[\frac{1}{\pi} , \frac{1}{2}]$, so there is nothing to prove; let us therefore assume that $v_{c+} < \frac{1}{\pi}$. 

By concavity of $\I$, the left-derivative $\kappa_p := \I^{',-}(v_p)$ exists, and since $\I \equiv 1$ on the non-empty $(v_p,1/2)$ (or simply since $\I$ is symmetric about $1/2$), we have $\kappa_p \geq 0$. Similarly, the right-derivative $\kappa_{c+} := \I^{',+}(v_{c+})$ exists, and since $\I = \I_m$ on the non-empty $(v_{c-},v_{c+})$, we have $\kappa_{c+} \leq \I'_m(v_{c+})$.

By Proposition \ref{prop:genus}, we know that the genus of all minimizers of volume $\bar v \in (v_{c+},v_p)$ is at least $2$, and so by Proposition \ref{prop:ODI} the isoperimetric profile $\I$ satisfies on $(v_{c+},v_p)$ (in the viscosity sense):
\[
\I^2 \I'' + \I (\I')^2 \leq -\frac{\pi}{\beta}  ~,~ \I(v_p) = 1 ~,~ \I^{',-}(v_p) = \kappa_p .
\]
Defining $F$ to be solution to the ODE
\[
F^2 F'' + F (F')^2 = - \frac{\pi}{\beta} ~,~ F(v_p) = 1 ~,~ F'(v_p) = 0 (\leq \kappa_p) ,
\]
on the maximal interval $J$ where a positive solution exists, it follows by the maximum principle (see e.g. \cite[Lemma 8 (iv)]{HPRR-PeriodicIsoperimetricProblem}) that necessarily $\I \leq F$ on $[v_{c+},v_p] \subset J$. Note that $F \leq 1$, by comparing to the ODE with $0$ on the right-hand-side.
Trading off precision to gain simplicity, if we define $G = F^2/2$, our ODE becomes:
\[
G'' =  - \frac{\pi}{\beta} \frac{1}{F} \leq - \frac{\pi}{\beta} ~,~ G(v_p) = \frac{1}{2} ~,~ G'(v_p) = 0 ,
\]
and hence:
\[
F(\bar v) \leq F_{v_p}(\bar v) := \sqrt{1 - \frac{\pi}{\beta} (\bar v - v_p)^2} \;\;\; \forall \bar v \in J . 
\]
We conclude that $\I \leq F_{v_p}$ on $J$, which would be impossible if $F_{v_p}(\bar v) < \I_m(\bar v) = \sqrt{ \pi \bar v}$ for all $\bar v \in [0,\frac{1}{\pi}] \cap J$. Since changing $v_p$ only translates the graph of $F_{v_p}$, this proves that $v_p \leq v_{p,\max}$, where $v_{p,\max}$ is the critical value for which the graphs of $F_{v_{p,\max}}(\bar v)$ and $\sqrt{\pi \bar v}$ meet tangentially. In other words, $v_{p,\max}$ is defined by demanding that the quadratic equation:
\[
1 - \frac{\pi}{\beta} (x - v_{p,\max})^2 = \pi x 
\]
have a double root, or equivalently, that
\[
\frac{1}{\beta} y^2 + y + v_{p,\max} - \frac{1}{\pi} = 0
\]
have vanishing discriminant, yielding:
\[
v_{p,\max} = \frac{1}{\pi} + \frac{\beta}{4},
\]
and establishing part (\ref{it:Q3-3}). 
Note that the graphs of $F_{v_{p,\max}}(\bar v)$ and $\sqrt{\pi \bar v}$ meet tangentially at the double root $\frac{1}{\pi} - \frac{\beta}{4}$. 

An identical argument shows that $\I \leq F$ on $J \cap [v_{c+} , 1/2]$, where $F$ is now defined to be the solution to:
\[
F^2 F'' + F (F')^2 = - \frac{\pi}{\beta} ~,~ F(v_{c+}) = \I_m(v_{c+}) ~,~ F'(v_{c+}) = \I'_m(v_{c+}) (\geq \kappa_{c+}) ,
\]
on the maximal interval $J$ where a positive solution exists. Defining $G = F^2/2$, we see that as long as $F \leq 1$, $G$ satisfies
\[
G'' =  - \frac{\pi}{\beta} \frac{1}{F} \leq - \frac{\pi}{\beta} ~,~ G(v_{c+}) = \frac{\pi}{2} v_{c+}  ~,~ G'(v_{c+}) = \frac{\pi}{2} .
\]
Consequently, $F \leq F_{v_{c+}}$ as long as $F_{v_{c+}} \leq 1$, where:
\[
F_{v_{c+}}(\bar v) := \sqrt{ \pi v_{c+} + \pi (\bar v - v_{c+}) - \frac{\pi}{\beta} (\bar v - v_{c+})^2 } . 
\]
This would be impossible if $F_{v_{c+}} < 1$ on $J$, and the extremal case is when the graphs of $F_{v_{c+}}$ and $1$ meet tangentially. But this case exactly coincides with the case we've already examined above, and we deduce that this occurs at $v_{c+,\min} = \frac{1}{\pi} - \frac{\beta}{4}$. It remains to note that the graphs of $F_{v_{c+}}$ as a function of $v_{c+}$ are nested, and so if $v_{c+} < v_{c+,\min}$ we would have $F_{v_{c+}} < 1$ which is impossible. It follows that $v_{c+} \geq v_{c+,\min}$, establishing part (\ref{it:Q3-2}) and concluding the proof.
\end{proof}

\section{Counterexample on high-dimensional cubes} \label{sec:high-dim-cube}

In this section we establish Theorem \ref{thm:Qn} asserting the falsehood of Conjecture \ref{conj:Qn} regarding the $n$-dimensional cube $\Q^n$ for large enough  $n$. As already mentioned in the Introduction, the argument for demonstrating that the conjecture is false is the same as the one used by Pedrosa and Ritor\'e in \cite{PedrosaRitore-Products}, but requires some more computation. 
The idea is to show that when $n \geq 10$, a unit-radius neighborhood $B_1$ of a $0$-dimensional face (a vertex) of $\Q^n$ has strictly less surface area than all of the other tubular neighborhoods of $k$-dimensional faces for $k=1,\ldots,n-1$ (or their complements) of the same volume. However, $B_1$ cannot be an isoperimetric minimizer as it is tangential to the boundary of $\Q^n$; equivalently, in view of Remark \ref{rem:reflect-slab}, one may apply this argument to the $n$-dimensional flat torus $\T^n = \R^n / (2 \Z^n)$, and argue that since a geodesic ball of radius $1$ touches itself tangentially, it cannot be an isoperimetric minimizer (by e.g. \cite[Lemma 30.2]{MaggiBook}). 
 
It remains to verify that $B_1$ has strictly less surface area than all of the other tubular neighborhoods of $k$-dimensional faces (or their complements) of the same volume. As for $\Q^3$, when $\bar v \leq 1/2$ there is actually no need to test the complements of tubular neighborhoods, since a minimizer $E$ of volume $\bar v \leq  1/2$ will necessarily have non-negative mean-curvature by the concavity and symmetry of the isoperimetric profile $\I(\Q^n)$ (recall Proposition \ref{prop:I-properties}) in conjunction with Proposition \ref{prop:lambda-I} (which applies to $\Sigma = \partial^* E$ in all dimensions $n$). 
A radius $r \in [0,1]$ tubular neighborhood of a vertex of $\Q^k$ ($k \geq 1$) has volume and surface area given by
\[
V_k(r) = \frac{1}{2^k} \omega_k r^k ~,~ A_k(r) = \frac{1}{2^k} k \omega^{1/k}_k r^{k-1} ,
\]
where:
\[
\omega_k = \frac{\pi^{k/2}}{\Gamma(k/2 + 1)}
\]
is the volume of the $k$-dimensional unit-ball. Clearly, the same formulas apply to a radius $r \in [0,1]$ tubular neighborhood of an $(n-k)$-dimensional face of $\Q^n$ for all $n \geq k$. 
It follows that as long as $r \in [0,1]$, i.e.~as long as $\bar v \in [0,v_s(k)]$ where
\[
v_s(k) := \frac{\omega_k}{2^k} ,
\]
then a tubular neighborhood of an $(n-k)$-dimensional face of volume $\bar v$ has surface area $\I_m^{(k)}(\bar v)$, where:
\[
\I_m^{(k)}(\bar v) := \frac{k \omega_k^{1/k}}{2} {\bar v}^{\frac{k-1}{k}}.
\]
The isoperimetric conjecture for the $n$-dimensional cube $\Q^n$ thus predicts that its isoperimetric profile $\I(\Q^n)$ coincides on $[0,1/2]$ with:
\[
\I_m(\bar v) := \min \{ \I_m^{(k)}(\bar v) \; ; \; \bar v \in [0,v_s(k)] ~ ,~  k=1,\ldots,n \} ~,~ \bar v \in [0,1/2] ~ .
\]

\begin{lemma} \label{lem:omega-decreasing}
$\omega_n$ is strictly decreasing in $n \in \mathbb{N}$ for $n \geq 6$. $v_s(n)$ is strictly decreasing in $n \in \mathbb{N}$ and $v_s(4) < 1/2$. 
\end{lemma}
\begin{proof}
The first claim is elementary to verify. The second is trivial since $\pi < 4$, and hence $\pi^{n/2} / 2^n$ is strictly decreasing, while $\Gamma(n/2+1)$ is strictly increasing. 
\end{proof}

As explained above, to demonstrate that Conjecture \ref{conj:Qn} is false for $n \geq 10$, it remains to show that the unit-radius tubular neighborhood of a vertex of $\Q^n$ (having volume $v_s(n)<v_s(4) < 1/2$) has strictly smaller surface area than all of the other tubular neighborhoods of $(n-k)$-dimensional faces for $k=1,\ldots,n-1$ of volume $v_s(n)$. We will actually show that this sequence is strictly decreasing in $k$ when $n \geq 10$:

\begin{figure}
\begin{center}
        \includegraphics[scale=0.4]{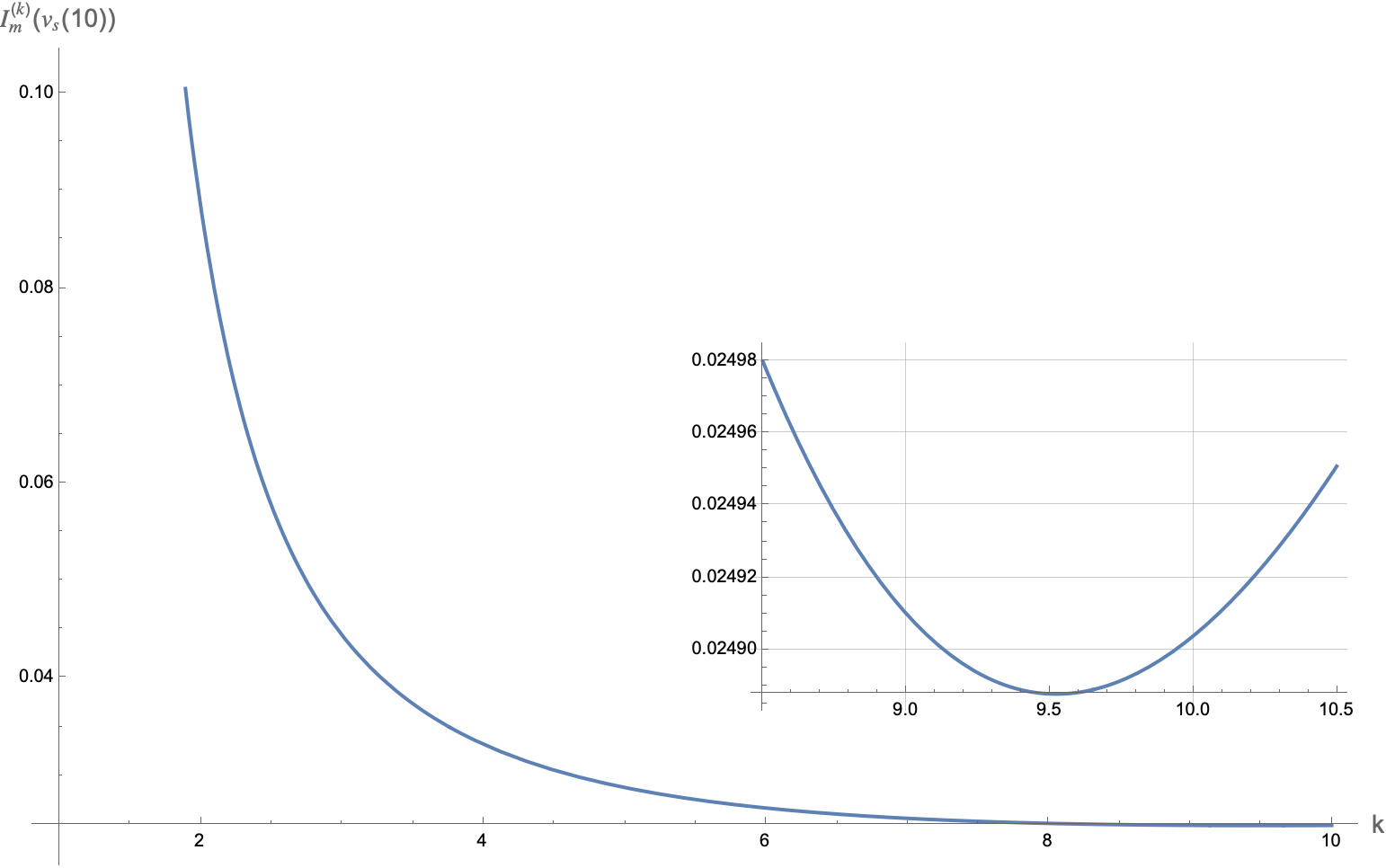}
        \includegraphics[scale=0.4]{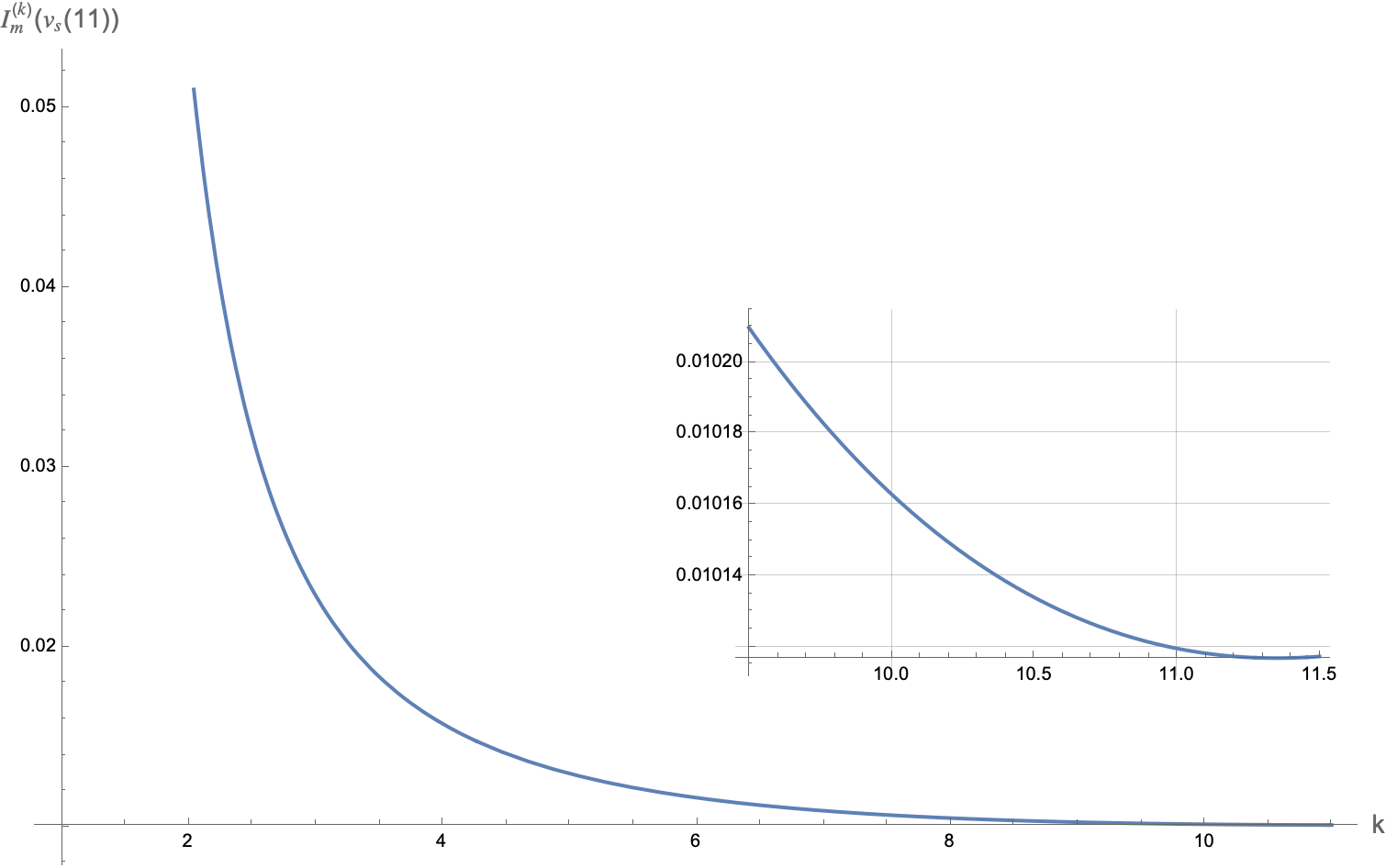}
     \end{center}
     \caption{
         \label{fig:high-dim-10}
         By direct verification, $k \mapsto \I_m^{(k)}(v_s(10))$ is strictly decreasing on the sequence of integers $\{1,\ldots,10\}$, although not on the interval $[9,10]$. $k \mapsto \I_m^{(k)}(v_s(11))$ is strictly decreasing on the sequence of integers $\{1,\ldots,11\}$ (and on the entire $[1,11]$ as well). 
     }
\end{figure}

\begin{lemma}
For all $n \geq 10$, $\{1,\ldots,n\} \ni k \mapsto \I_m^{(k)}(v_s(n))$ is strictly decreasing. This is false for $n \leq 9$. 
\end{lemma}
\begin{remark}
We restrict to integer values of $k$ since otherwise the claim would be false for $n=10$. 
\end{remark}
\begin{proof}
The negative claim for $n \leq 9$ and positive claims for $n=10$ and $n=11$ are verified by explicit computation; see Figure \ref{fig:high-dim-10}. To prove that the strict decrease remains true for all $n \geq 11$, we argue as follows. $\I_m^{(k)}(\bar v) = k (\omega_k / \bar v)^{1/k} \bar v$, so fixing $\bar v = v_s(n)$, the claim is equivalent to showing that $k (\omega_k / v_s(n))^{1/k}$ is strictly decreasing in $k \in \{1,\ldots,n\}$, i.e. that $F(k,n) := 2^{n/k} k (\omega_k / \omega_n)^{1/k}$ is strictly decreasing in $k \in \{1,\ldots,n\}$. We argue by induction on $n$, starting at $n=11$. Assume this is true for $n_0$, and write:
\[
\log F(k,n) = \frac{n}{k} \log 2 + \log k + \frac{1}{k} (\log \omega_k - \log \omega_n) . 
\]
Taking partial derivative in $n$, we see that:
\[
\partial_n \log F(k,n) = \frac{\log 2 - \frac{d}{dn} \log \omega_n}{k} . 
\]
Recall from Lemma \ref{lem:omega-decreasing}  that $\omega_n$ is strictly decreasing in $n$ when $n \geq 6$. Consequently, the numerator is positive, and so the derivative is positive and strictly decreasing in $k$. Integrating this on $[n_0,n_0+1]$ and using the induction hypothesis at $n_0$, we confirm that $\log F(k,n_0+1)$ remains strictly decreasing in $k \in \{1,\ldots,n_0\}$.

To extend this to $k=n_0+1$ and confirm that $\log F(n_0,n_0+1) > \log F(n_0+1,n_0+1) = \log 2(n_0+1)$, it is enough to show that $\partial_n \log F(n_0,n)  > \partial_n \log (2 n) = \frac{1}{n}$ for all $n \in [n_0,n_0+1]$. Therefore, it is enough to show that for $n \geq 11$:
\[
\log 2 - \frac{d}{dn} \log \omega_n > 1 ,
\]
or equivalently:
\[
\log 2 + \frac{d}{dn} \log \Gamma(n/2+1) - \frac{1}{2} \log \pi  > 1 . 
\]
This is verified by explicit computation for $n=11$, and since $\Gamma(x) = \int_0^\infty t^{x-1} e^{-t} dt$ is log-convex (e.g. by H\"older's inequality) on $(0,\infty)$, this remains true for all $n \geq 11$ as well. 
\end{proof}

\section{Gaussian Slabs} \label{sec:Gn}

Recall that $\G_T^n$ denotes the Gaussian slab of width $T > 0$ over the base $\G^{n-1} = (\R^{n-1} , \abs{\cdot}^2, \gamma^{n-1})$, where $\gamma^{n-1}$ denotes the standard Gaussian measure. Recall from the Introduction that the isoperimetric minimizers in $\G^{n-1}$ are half-planes, and therefore $\I(\G^{n-1}) = \I(\G^1) = \I_\gamma = \varphi_{\gamma} \circ \Phi_{\gamma}^{-1}$, where $\varphi_{\gamma}$ denotes the standard Gaussian density on $\R$ and $\Phi_{\gamma}(s) = \int_{-\infty}^s \varphi_{\gamma}(x) dx$. As the half-planes in $\G^{n-1}$ may be chosen to be nested, 
 it follows by Corollary \ref{cor:nested} that $\I_T := \I(\G_T^n)$ coincides with the based-induced profile $\I_T^b = \I(S_T(\I_\gamma))$. Since $\varphi_{\I_\gamma} = \varphi_{\gamma}$, we see that the model two-dimensional slab $S_T(\I_\gamma)$ coincides with $\G_T^2$, and we conclude that 
 $\I(\G_T^n) = \I(\G_T^2)$. This reduction from the case that the base is an $(n-1)$-dimensional Gaussian to the case that it is a one-dimensional Gaussian is well-known, and was already shown in the work of Fusco--Maggi--Pratelli \cite{FuscoMaggiPratelli-GaussianProduct}. It can also be directly obtained by employing Ehrhard symmetrization \cite{EhrhardPhiConcavity}, which is actually what the general machinery of Section \ref{sec:CMC} does in the Gaussian case. 

Our proof of Theorem \ref{thm:main-Gn} will be based on stability analysis and ODE arguments for the isoperimetric profile. Let $E$ denote a downward monotone minimizer in $\G^2_T = S_T(\I_\gamma)$ enclosed by $\Sigma = \partial E$, which by Proposition \ref{prop:regularity} will be $C^\infty$-smooth. Since the base profile $I_\gamma(\bar v) = \sqrt{2} v \sqrt{\log(1/v)} (1 + o(1))$ as $\bar v \rightarrow 0+$, Corollary \ref{cor:no-one-sided} states that $\Sigma$ cannot be a one-sided Gaussian unduloid, and so by Theorem \ref{thm:model-slab-main} it is either a (two-sided) unduloid, a horizontal line or a vertical one; in particular, unless it is vertical, $\Sigma$ is compact. Note that vertical lines have weighted perimeter $\frac{1}{T}$. 

The density of $\G^2_T$ is of the form $\exp(-W)$ with $W(t,s) = \frac{s^2}{2} + c$, and hence $\nabla^2 W = e_2 \otimes e_2$. Since $\G^2_T$ is geometrically flat, we see from (\ref{eq:LJac-locally-flat}) that for all $\theta \in \R^2$, the Jacobi operator $L_{Jac}$ on $\Sigma$ satisfies
\[
L_{Jac} \scalar{\theta,\n_{\Sigma}} = \nabla^2 W(\theta,\n_{\Sigma}) = \scalar{\theta,e_2}  \scalar{\n_{\Sigma},e_2}  .
\]
Applying this to the vertical direction $\theta = e_2$, we see that $\scalar{e_2,\n_{\Sigma}}$ is an eigenfunction of the Jacobi operator:
\begin{equation} \label{eq:eigenfunction}
L_{Jac} \scalar{e_2,\n_{\Sigma}} = \scalar{e_2,\n_{\Sigma}} . 
\end{equation}
Note that since $E$ is downward monotone, we have $\scalar{e_2 , \n_{\Sigma}} > 0$ unless $\Sigma$ is a vertical line (and then $\scalar{e_2 , \n_{\Sigma}} \equiv 0$). Also note that the vertical field $e_2$ is tangential to the boundary of $\G^2_T$, and (excluding vertical $\Sigma$'s) by truncating it outside a neighborhood of the compact $\Sigma$, we may assume that it is compactly supported, and therefore $\scalar{e_2 , \n_{\Sigma}} \in \scalar{C_c^\infty,\n_{\Sigma}}$ in the notation of Section \ref{sec:stability}. 
In view of Lemma \ref{lem:ODI}, we deduce the following crucial:
\begin{proposition} \label{prop:GT-ODI}
The isoperimetric profile $\I_T$ of $\G^n_T$ satisfies the following differential inequality at $\bar v \in (0,1)$ in the viscosity sense whenever $\I_T(\bar v) < \frac{1}{T}$:
\[
\I_T(\bar v)\I_T''(\bar v) \leq -1 ,
\]
with strict inequality unless all minimizers in $\G^2_T$ of weighted volume $\bar v$ are (up to null-sets) horizontal half-planes. \\
In particular, if $0 \leq \bar v_0 < \bar v_1 \leq 1/2$ are so that $\I_T(\bar v) < \frac{1}{T}$ for all $\bar v \in (\bar v_0 , \bar v_1)$, and $\I_m$ is a smooth function so that $\I_m(\bar v_i) = \I_T(\bar v_i)$, $i=0,1$, and $\I_m \I_m'' = -1$ on $(\bar v_0, ,\bar v_1)$, then $\I_T \geq \I_m$ on the entire $[\bar v_0 , \bar v_1]$. 
\end{proposition}
\begin{proof}
The assumption that $\I_T(\bar v) < \frac{1}{T}$ guarantees that a minimizer $E$ of weighted volume $\bar v$ in $\G^2_T$ is enclosed by a non-vertical $\Sigma_{\bar v} = \overline{\partial^* E}$, and hence  $\int_{\Sigma_{\bar v}} u  \, d\sigma_T > 0$ for $u =\scalar{e_2,\n_{\Sigma_{\bar v}}} \in \scalar{C_c^\infty,\n_{\Sigma}}$. Using this test-function in (\ref{eq:ODI-gen}), applying (\ref{eq:eigenfunction}) and Cauchy-Schwarz:
\begin{equation} \label{eq:CS}
\int_{\Sigma_{\bar v}} d\sigma_T \int_{\Sigma_{\bar v}} u^2 d\sigma_T \geq \brac{\int_{\Sigma_{\bar v}} u d\sigma_T}^2 ,
\end{equation}
we deduce that $\I_T(\bar v) \I_T''(\bar v) \leq -1$ (in the viscosity sense). The inequality is in fact strict unless (at the very least) for all $E$ as above there is equality in (\ref{eq:CS}), i.e.~unless $u$ is constant on $\Sigma_{\bar v}$; since $\Sigma_{\bar v}$ meets $\partial G^2_T$ perpendicularly, it follows that $\Sigma_{\bar v}$ must be horizontal, and hence all minimizers of weighted volume $\bar v$ must be horizontal half-planes (up to null-sets).  
The ``in particular" part follows from the maximum principle argument of Lemma \ref{lem:ODI}.
\end{proof}

We are now ready to establish Theorem \ref{thm:main-Gn}.

\begin{proof}[Proof of Theorem \ref{thm:main-Gn}]
Let $T > \sqrt{2 \pi}$. Define $v_v := \min \{ \bar v \in (0,1/2] \; ; \; \I_T(\bar v) = \frac{1}{T} \}$; since $\I_T(1/2) = \frac{1}{T}$ and $\I_T$ is continuous the minimum is over a non-empty set and is attained. Similarly, define $v_h := \max \{ \bar v \in [0,1/2] \; ; \; \I_T(\bar v) = \I_{\gamma}(\bar v) \}$.
 
Since $\I_\gamma(v_h) = \I_{T}(v_h) \leq \frac{1}{T}$, $\I_{\gamma}$ is strictly increasing on $[0,1/2]$, and $\I_T \leq \I_{\gamma}$ by testing horizontal half-planes, it follows that $\I_T < \frac{1}{T}$ on $[0,v_h)$. In addition, $\I_T(0) = \I_{\gamma}(0) = 0$ and $\I_{\gamma} \I_{\gamma}'' = -1$, and so Proposition \ref{prop:GT-ODI} implies that $\I_T \geq \I_{\gamma}$ on the entire $[0,v_h]$. But as $\I_T \leq \I_\gamma$, it follows that $\I_T = \I_{\gamma}$ on the entire interval, implying that horizontal half-planes are indeed minimizers (in $\G^2_T$ and $\G^n_T$). 

On the other end, $\I_T(v_v) = \I_T(1/2) = \frac{1}{T}$ (and in fact also $\I_T(1-v_v) = \frac{1}{T}$ by symmetry). The concavity of $\I_T$ (recall Proposition \ref{prop:I-properties}) implies that $\I_T \geq \frac{1}{T}$ on the entire $[v_v,1/2]$. But as $\I_T \leq \frac{1}{T}$ by testing vertical half-planes, it follows that $\I_T \equiv \frac{1}{T}$ on the entire interval,  implying that vertical half-planes are indeed minimizers (in $\G^2_T$ and $\G^n_T$). 

By definition, on $(v_h,v_v)$ we have $\I_T(\bar v) < \min(I_{\gamma},\frac{1}{T})$, and therefore neither horizontal nor vertical half-planes are minimizers. Since we've also already disqualified one-sided generalized unduloids (recall Corollary \ref{cor:no-one-sided}), a minimizer must be enclosed by a (two-sided) Gaussian unduloid, whose explicit description in $\G^2_T$ is given by (\ref{eq:f-formula}), which immediately extends to $\G^n_T$ -- see Figure \ref{fig:Gaussian-unduloids}. In addition, Proposition \ref{prop:GT-ODI} implies that $\I_T \I_T'' < -1$ in the viscosity sense. 

By Lemma \ref{lem:monotone} we know that $\I_T$ is pointwise non-increasing and that $T \I_T$ is pointwise non-decreasing in $T$, respectively implying that $v_h$ and $v_v$ are non-increasing in $T$. 

When $T > \pi$, since $\I_{\gamma} \I_{\gamma}'' = -1$ we know by Lemma \ref{lem:horizontal-stable} that all horizontal lines in $\G^2_T$ are unstable. In particular, (non-empty) horizontal half-planes can never be minimizing and therefore $I_T(\bar v) < I_{\gamma}(\bar v)$ for all $\bar v \in (0,1)$. In particular, $v_h = 0$ and $\frac{1}{T} = I_T(v_v) < \I_\gamma(v_v)$. 

It remains to establish that $v_v \leq \frac{\sqrt{2 \pi}}{2 T}$, which is done in the subsequent lemma. 
\end{proof}

\begin{figure}
\begin{center}
        \includegraphics[scale=0.4]{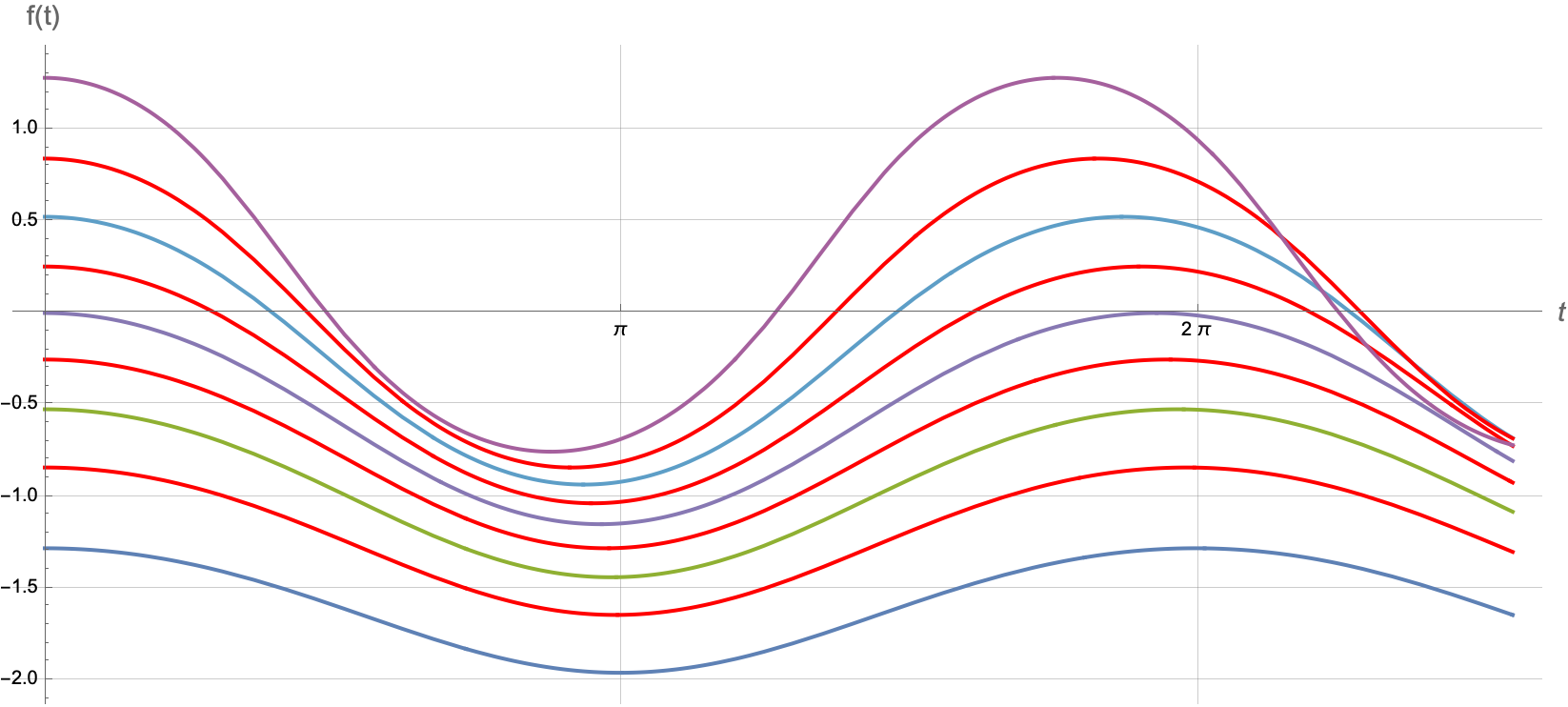}
     \end{center}
     \caption{
         \label{fig:Gaussian-unduloids}
         Gaussian unduloids on $G^2_T$ with various values of $v_1 \in (0,1)$ and $v_0 = v_1/4$. Every unduloid has its own half-period $T(v_0,v_1)$, and only those with $T(v_0,v_1) = T$ can be minimizers, but here we depict all of them and extend them periodically to fill the entire plot. 
     }
\end{figure}

\begin{lemma}
For any $T > \sqrt{2 \pi}$, on $\G^n_T$ we have $v_v \leq v_v^+ := \frac{\sqrt{2 \pi}}{2 T}$ and:
\[
\I_T \geq \I_{\gamma,T,v_v} \text{ on $[0,1]$},
\]
where given $w \in (0,1/2]$ so that $\I_T(w) = \frac{1}{T}$, $\I_{\gamma,T,w}$ is defined by first uniquely determining $\delta_w \in (w,1]$ so that $\delta_w \I_{\gamma}(w/\delta_w) = \frac{1}{T}$, and then setting:
\[
\I_{\gamma,T, w} (\bar v) := \begin{cases} 
\delta_w \I_{\gamma}(\bar v/\delta_w) & \bar v \in [0,w] \\
\frac{1}{T} & \bar v \in [w, 1 - w] \\
\delta_w\I_{\gamma}((1-\bar v)/\delta_w) & \bar v \in [1-w,1] 
\end{cases} . 
\]
In particular, $\delta_{v_v^+} = \frac{\sqrt{2 \pi}}{T}$, and $\I_T$ is lower-bounded by the $C^{1,1}$ function $\I_{\gamma,T,v_v^+}$ (see Figure \ref{fig:GaussianBounds}).
\end{lemma}

\begin{proof}
Let $w \in (0,1/2]$ so that $\I_T(w) = \frac{1}{T}$.
Since $\I_\gamma$ is strictly concave, the function $\delta \mapsto \delta \I_\gamma(w/\delta)$ is increasing from $0$ on $[w,\infty)$. 
Since at $\delta=1$ we have $\I_\gamma(w) \geq \I_T(w) =  \frac{1}{T}$, we deduce the existence of a unique $\delta_w \in (w,1]$ so that $\delta_w \I_\gamma(w/\delta_w) = \frac{1}{T}$. 

Recall that $v_v$ is characterized as the minimal $w$ so that $\I_T \equiv \frac{1}{T}$ on $[w,1-w]$. Denoting $\I_m(\bar v) := \delta_{v_v} I_\gamma(\bar v/\delta_{v_v})$, observe that $\I_m \I_m'' = -1$ on $[0,v_v]$, $\I_m(0) = \I_T(0) = 0$ and $\I_m(v_v) = \I_T(v_v) = \frac{1}{T}$. As $\I_T(\bar v) < \frac{1}{T}$ for all $\bar v \in [0,v_v)$, it follows by Proposition \ref{prop:GT-ODI} that $\I_T \geq \I_m$ on the entire $[0,v_v]$, and we deduce by symmetry that $\I_T \geq \I_{\gamma,T,v_v}$ on $[0,1]$. 

Now observe that if $v_v / \delta_{v_v} > 1/2$, we would have $\I_{\gamma,T,v_v}(\delta_{v_v} / 2) >  \I_{\gamma,T,v_v}(v_v) = \frac{1}{T}$,  in contradiction to $\frac{1}{T} \geq \I_T \geq \I_{\gamma,T,v_v}$. Consequently $\delta_{v_v} \geq 2 v_v$, and since $\delta \mapsto \delta \I_\gamma(v_v/\delta)$ is increasing on $[v_v,\infty)$, comparing $\delta \in \{ 2 v_v , \delta_{v_v} \}$ we deduce:
\[
2 v_v I_\gamma(1/2) \leq \delta_{v_v} I_\gamma(v_v/\delta_{v_v}) = \frac{1}{T} ,
\]
and therefore $v_v \leq v_v^+ := \frac{\sqrt{2 \pi}}{2 T}$ as asserted. 

Finally, $\delta_{v_v^+} = 2 v_v^+$, and it is easy to check that $\I_{\gamma,T,w}$ is pointwise non-increasing on $w \in [v_v,v_v^+]$, implying that $\I_T \geq \I_{\gamma,T,v_v^+}$. 
\end{proof}

\begin{figure}
\begin{center}
        \includegraphics[scale=0.4]{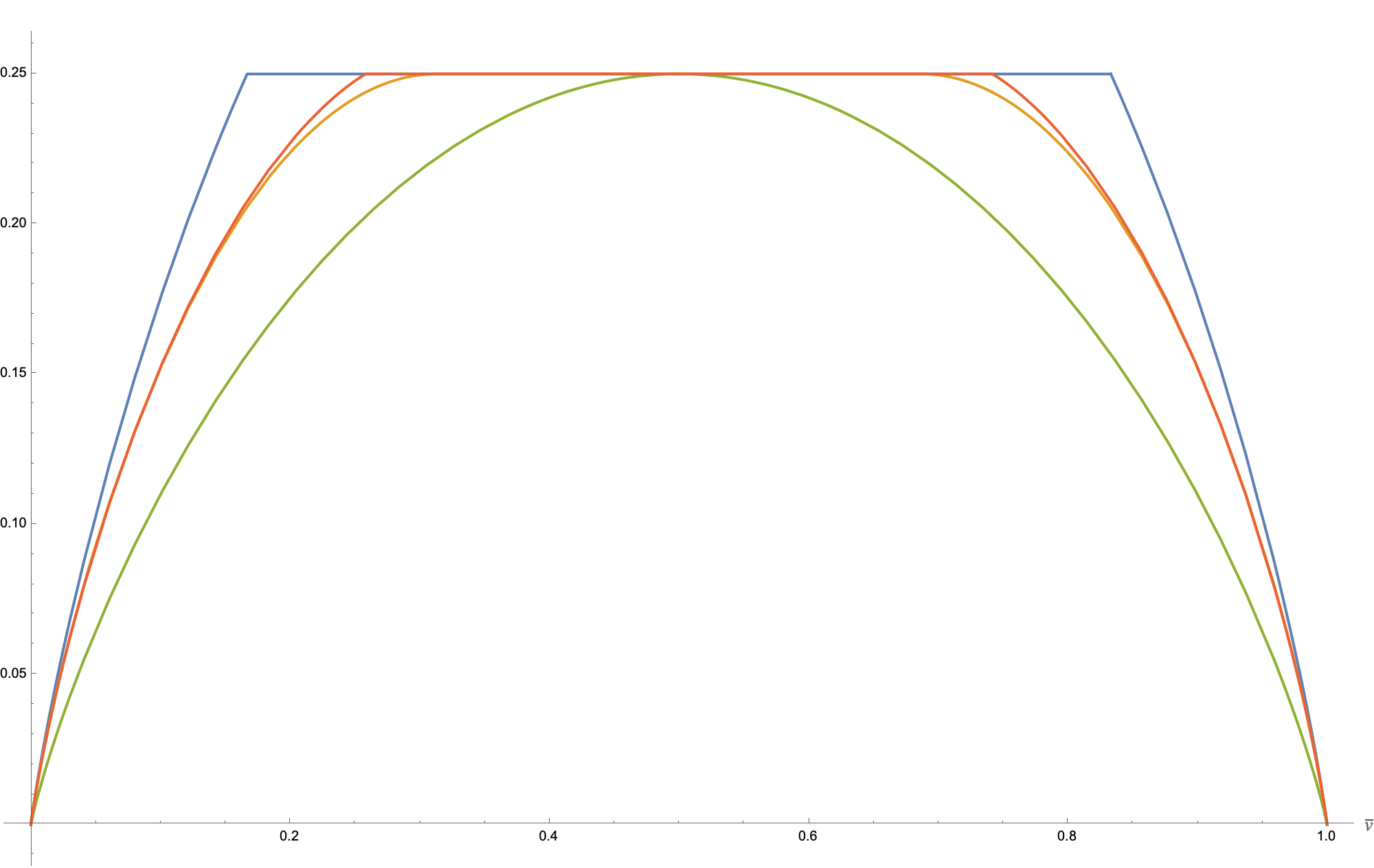}
     \end{center}
     \caption{
         \label{fig:GaussianBounds}
         On $G^2_T$ with $T=4$: trivial upper bound for $\I_T$ is $\min(\I_\gamma , \frac{1}{T})$ in blue; trivial lower bound for $\I_T$ is $\frac{\sqrt{2 \pi}}{T} \I_\gamma$ in green; lower bound $\I_{\gamma,T,v_v^+}$ in yellow; lower bound $A_{-,T}$ in red. 
     }
\end{figure}

\begin{remark}
An interesting challenge is to get an explicit formula for $v_v$. Certainly the upper bound $v_v \leq \frac{\sqrt{2 \pi}}{2 T}$ can be improved for large $T$. For example, recalling Lemma \ref{lem:VA}, we know that the weighted perimeter $A_T(v_0,v_1)$ of a generalized Gaussian unduloid $\Sigma$ in $\G^2_T$ with parameters $0 < v_0 < v_1 < 1$ enclosing a set $E$ of weighted volume $\bar v$, is lower bounded by:
\[
A_T(v_0,v_1) \geq A_{-,T}(\bar v, v_0,v_1) := \frac{\ell(\bar v,v_0,v_1) + \sqrt{\ell(\bar v,v_0,v_1)^2 + 4 \brac{\frac{v_1-v_0}{T}}^2}}{2} ,
\]
where:
\[
\ell(v,v_0,v_1) := \frac{v_1-v}{v_1-v_0} I_\gamma(v_0) + \frac{v-v_0}{v_1-v_0} I_\gamma(v_1) . 
\]
By (\ref{eq:V-formula}) and (\ref{eq:T-formula}), $\bar v = V(E)$ is the weighted average of values $v \in [v_0,v_1]$, and hence $v_0 < \bar v < v_1$. Clearly when $v_0,v_1 \rightarrow \bar v$ the unduloid converges to a horizontal line and indeed $A_{-,T}(\bar v , v_0,v_1)$ converges to its corresponding weighted perimeter $\I_\gamma(\bar v)$; on the other extreme, when $v_0 \rightarrow 0$ and $v_1 \rightarrow 1$ the lower bound converges to $\frac{1}{T}$, the weighted perimeter of a vertical line. Consequently, we have for all $\bar v \in (0,1)$:
\[
\I_T(\bar v) \geq A_{-,T}(\bar v) := \inf \{ A_{-,T}(\bar v,v_0,v_1) \; ; \; 0 < v_0 < \bar v < v_1 < 1 \} . 
\]
Defining:
\[
v_v^{++} := \inf \{ \bar v \in (0,1/2] \; ; \; A_{-,T}(\bar v) = \frac{1}{T} \} ,
\]
it follows that $v_v \leq v_v^{++}$. Numerical evidence suggests that we always have $A_{-,T} \geq \I_{\gamma,T,v_v^+}$ and $v_v^{++} < \frac{\sqrt{2 \pi}}{2 T}$, strictly improving our upper bound on $v_v$ from the previous lemma; see Figure \ref{fig:GaussianBounds}.
\end{remark}

\appendix

\section{Appendix}

The appendix is dedicated to providing a proof of Proposition \ref{prop:main-computation}, which we repeat here for convenience:

\begin{proposition} \label{prop:Appendix}
For all $\beta > 0$, if $\Sigma$ is a generalized unduloid in $S_\beta(\I_{\T^2})$ with parameters $0 < v_0 < v_1 < 1$, having weighted mean-curvature $\lambda \geq 0.8$ and enclosing weighted volume $\bar v \leq \frac{4 \pi}{81}$, then necessarily $v_1 \leq \frac{1}{\pi}$. 
\end{proposition}

In other words, our goal is to show that whenever $v_1 > \frac{1}{\pi}$ and $\lambda \geq 0.8$ then:
\[
V_\beta(v_0,v_1) - \frac{4 \pi}{81} > 0 ,
\]
where the weighted volume $V_\beta(v_0,v_1)$ of a generalized unduloid in $S_\beta(\I_{\T^2})$ is given in (\ref{eq:V-formula}). Since a generalized unduloid on $S_\beta(\I_{\T^2})$ has $T(v_0,v_1) = \beta$ with $T(v_0,v_1)$ given by (\ref{eq:T-formula}), this boils down to showing that whenever $v_1 > \frac{1}{\pi}$ and $\lambda \geq 0.8$,
\begin{equation} \label{eq:Q-formula}
Q(v_0,v_1) := \int_{v_0}^{v_1} \frac{v - 4 \pi / 81}{\I_{\T^2}(v) \sqrt{(\I_{\T^2}(v)/\ell_{v_0,v_1}(v))^2 - 1} } dv > 0   . 
\end{equation}
Clearly there is nothing to check if $v_0 \geq \frac{4\pi}{81}$, and so we may assume:
\begin{equation} \label{eq:range}
0 < v_0 < \frac{4\pi}{81} < \frac{1}{\pi} < v_1 < 1 . 
\end{equation}

Unfortunately, we could not find a more elegant argument for establishing the positivity of $Q(v_0,v_1)$ in the range (\ref{eq:range}) other than simply by brute force numerical verification. Recalling the definitions of:
\[
 \I_{\T^2}(v) = \min(\sqrt{\pi v} , 1 , \sqrt{\pi (1-v)}) ~,~ \ell_{v_0,v_1}(v) = \frac{v_1-v}{v_1-v_0} \I_{\T^2}(v_0) + \frac{v-v_0}{v_1-v_0} \I_{\T^2}(v_1) ,
\]
the integration in (\ref{eq:Q-formula}) involves several elliptic integrals.  Recall that given $m \in [0,1]$, $F$ and $E$ denote the elliptic integrals of the first and second kinds, respectively, defined as:
 \begin{align*}
 F(x,m) & := \int_0^x \frac{dt}{\sqrt{(1-t^2)(1-m t^2)}}~,~ & K(m) := F(1,m) ,\\
 E(x,m) & := \int_0^x \frac{\sqrt{1 - m t^2}}{\sqrt{1-t^2}} dt  ~,~ & E(m) := E(1,m)  .
    \end{align*}

Let us first treat the case when $v_1 \in (\frac{1}{\pi} , 1 - \frac{1}{\pi}]$, saving the case when $v_1 \in (1-\frac{1}{\pi},1)$ for later. In the former case, we will not need to use the assumption that $\lambda \geq 0.8$. To facilitate the computation and analysis, we divide the integration into three intervals: $[v_0,4 \pi / 81]$, $[4 \pi/81,1/\pi]$ and $[1/\pi,v_1]$. Since $\ell_{v_0,v_1}$ defines the chord between $v_0$ and $v_1$ of the concave $\I_{\T^2}$, we have $\ell_{0,v_1} \leq \ell_{v_0,v_1} \leq \ell_{v_0,1/\pi}$ on the intersection of their corresponding domains. Noting that the integrand is negative in the first interval and positive in the other two, we lower bound $Q$ as follows:
\begin{equation} \label{eq:QP}
Q(v_0,v_1) \geq P(v_0,v_1) := P_1(v_0) + P_2(v_1) + P_3(v_0,v_1) ,
\end{equation}
where:
\begin{align}
\nonumber P_1(v_0) & = \int_{v_0}^{\frac{4\pi}{81}} \frac{v - 4 \pi / 81}{\I_{\T^2}(v) \sqrt{(\I_{\T^2}(v)/\ell_{v_0,1/\pi}(v))^2 - 1} } dv ,\\
\label{eq:P2-int} P_2(v_1) & = \int_{\frac{4\pi}{81}}^{\frac{1}{\pi}}  \frac{v - 4 \pi / 81}{\I_{\T^2}(v) \sqrt{(\I_{\T^2}(v)/\ell_{0,v_1}(v))^2 - 1} } dv ,\\
\label{eq:P3-int} P_3(v_0,v_1) &= \int_{\frac{1}{\pi}}^{v_1}  \frac{v - 4 \pi / 81}{\I_{\T^2}(v) \sqrt{(\I_{\T^2}(v)/\ell_{v_0,v_1}(v))^2 - 1} } dv .
\end{align}
Plugging in $\I_{\T^2}(v)^2 = \pi v$ in the first two cases and $\I_{\T^2}(v) = 1$ in the third, and using that for all $v_1 \in [1/\pi,1-1/\pi]$
\[
\ell_{v_0,v_1}(v) = \frac{ v_1 \sqrt{\pi v_0} - v_0  + (1-\sqrt{\pi v_0}) v}{v_1 - v_0} ,
\]
and in particular $\ell_{0,v_1}(v) = v/ v_1$, we obtain by direct computation:
\begin{align*}
P_1(v_0) &=  - \frac{4}{27} \sqrt{\frac{1}{\pi }-\frac{4 \pi }{81} }
   \sqrt{\frac{4 \pi }{81}-v_0}  \\
   &   - 2 \left(\frac{v_0}{3 \pi }+\frac{4 \sqrt{\pi v_0}}{81}\right) \brac{K(1 - \pi v_0) - F\left(\frac{\sqrt{\frac{1}{\pi }-\frac{4 \pi }{81}}}{\sqrt{\frac{1}{\pi
   }-v_0}}, 1 - \pi v_0  \right)} \\
   &  + \frac{2}{\pi} \left(\frac{2 \left(\sqrt{\pi v_0 }+1\right)^2}{3 \pi
   }-\frac{\sqrt{v_0}}{3 \sqrt{\pi }}-\frac{4 \pi }{81}\right) \brac{E(1 - \pi v_0) - E\left(\frac{\sqrt{\frac{1}{\pi }-\frac{4 \pi }{81}}}{\sqrt{\frac{1}{\pi
   }-v_0}}, 1 - \pi v_0  \right)}  ,\\
P_2(v_1) & = \frac{2}{81} \brac{ \left(54  v_1^2 - \frac{8}{3}\right) \sqrt{\pi ^2 v_1^2-\frac{4
   \pi ^2}{81}} -  \left(54  v_1^2-4 + \frac{27}{\pi^2} \right) \sqrt{\pi ^2
   v_1^2-1} } , \\
P_3(v_0,v_1) &= 
\left(\frac{\sqrt{\pi v_0 } v_1 - v_0}{1 - \sqrt{\pi v_0 }}+
   v_1\right)^2 \tan^{-1}\left(\frac{\sqrt{v_1-\frac{1}{\pi
   }}}{\sqrt{\frac{2 \left(\sqrt{\pi v_0} v_1 - v_0\right)}{1 - \sqrt{\pi v_0}}+v_1+\frac{1}{\pi }}}\right) \\
   & +\frac{1}{2} \left(\frac{1}{\pi } -\frac{8 \pi }{81} -\frac{\sqrt{\pi v_0} v_1-v_0}{1 - \sqrt{\pi v_0}} \right) \sqrt{\left(v_1-\frac{1}{\pi
   }\right) \left(\frac{2 \left(\sqrt{\pi v_0} v_1 - v_0\right)}{1 - \sqrt{\pi v_0}}+v_1+\frac{1}{\pi }\right)} .
\end{align*}

\begin{figure}
\begin{center}
        \includegraphics[scale=0.5]{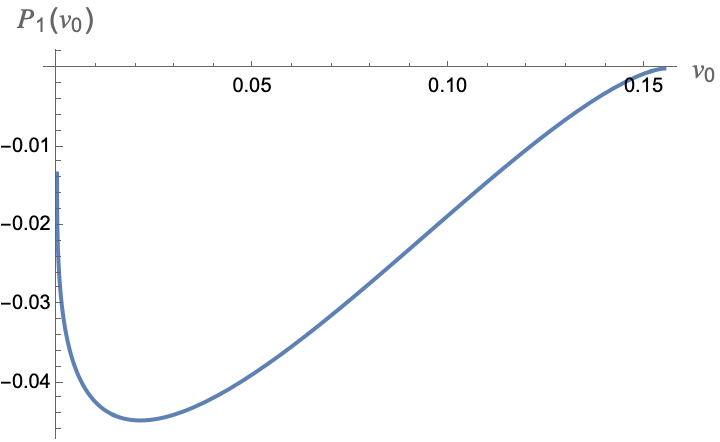}
          \includegraphics[scale=0.5]{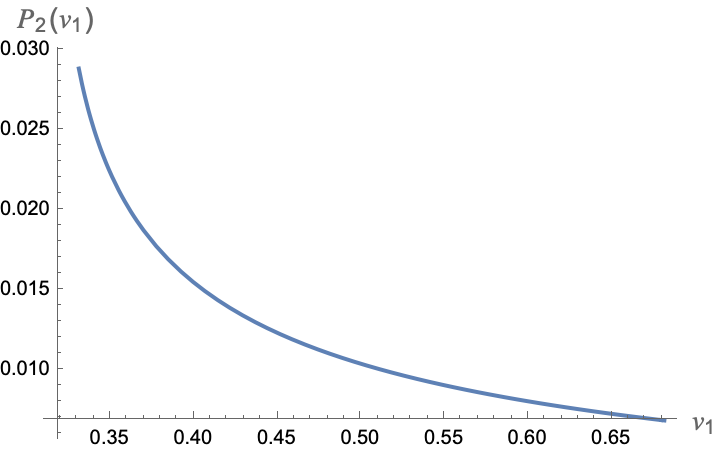}
         \includegraphics[scale=0.5]{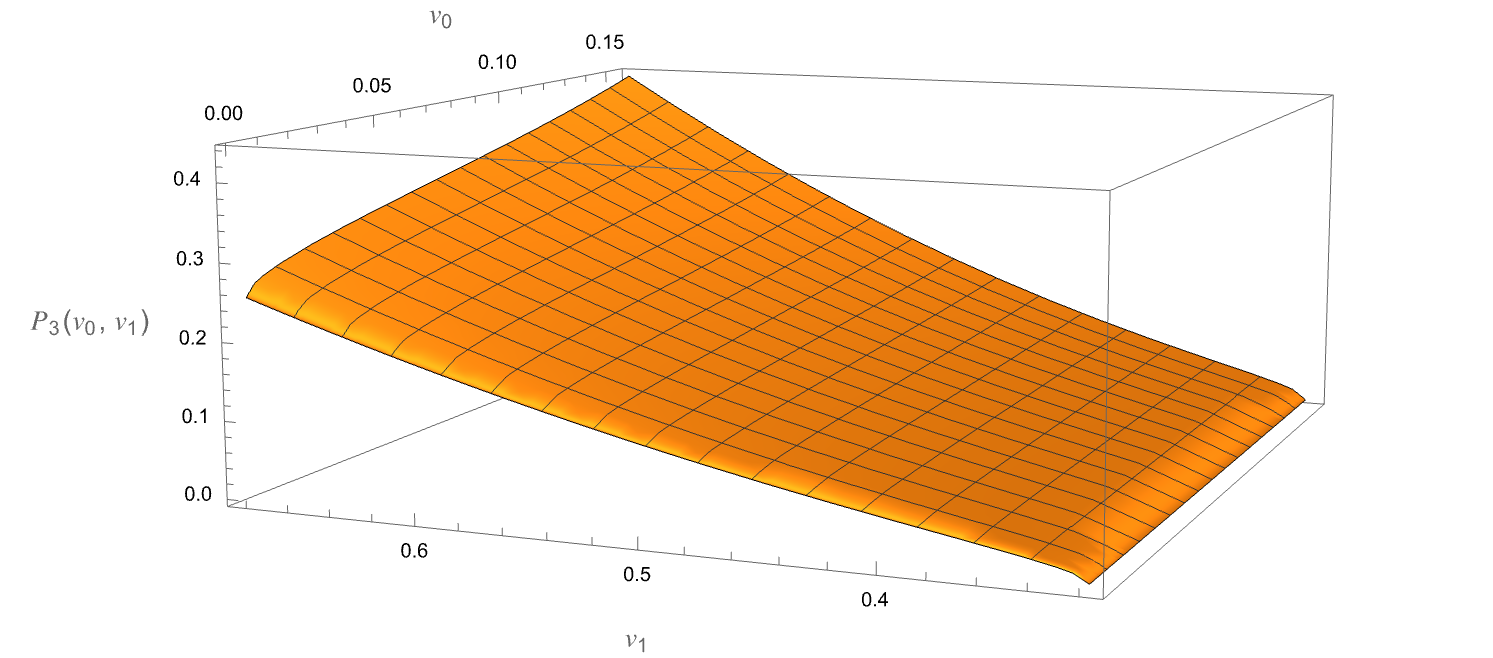} 
     \end{center}
     \caption{
         \label{fig:p123-plots}
         Plots of $P_1(v_0)$ for $v_0 \in (0,\frac{4\pi}{81})$, of $P_2(v_1)$ for $v_1 \in (\frac{1}{\pi},1-\frac{1}{\pi}]$ and $P_3(v_0,v_1)$ in the Cartesian product. 
              }
\end{figure}

\begin{figure}
\begin{center}
        \includegraphics[scale=0.3]{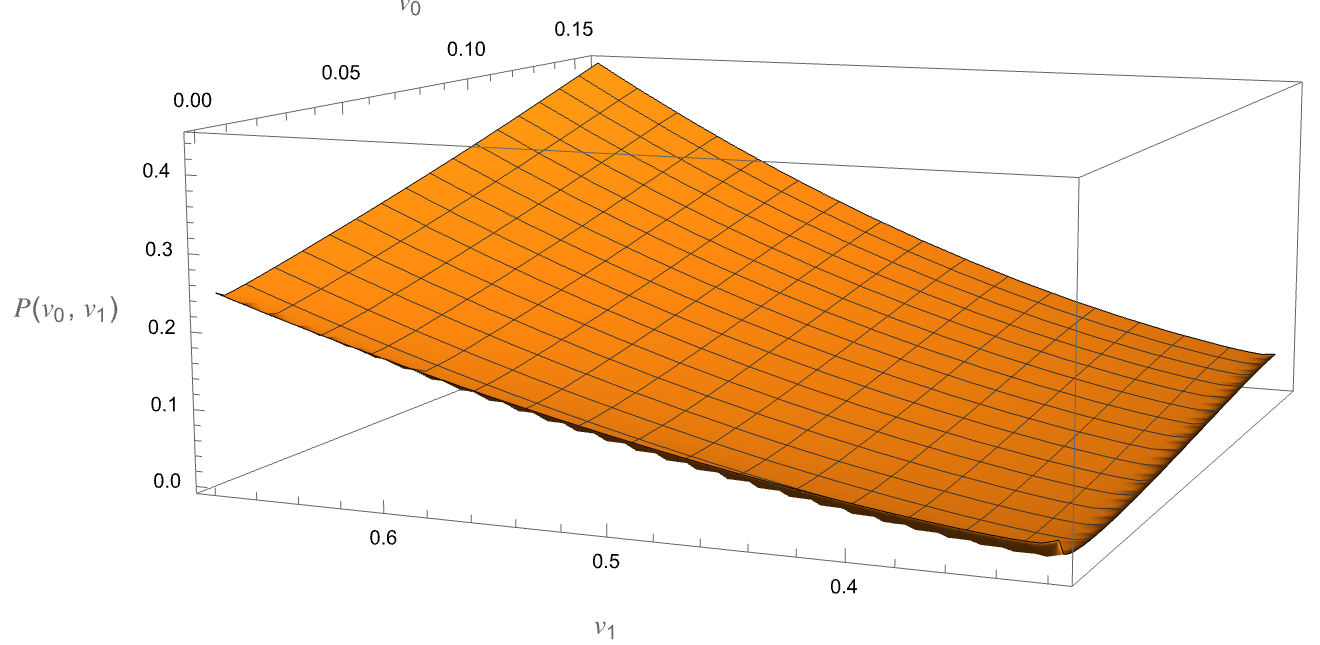}
                \includegraphics[scale=0.3]{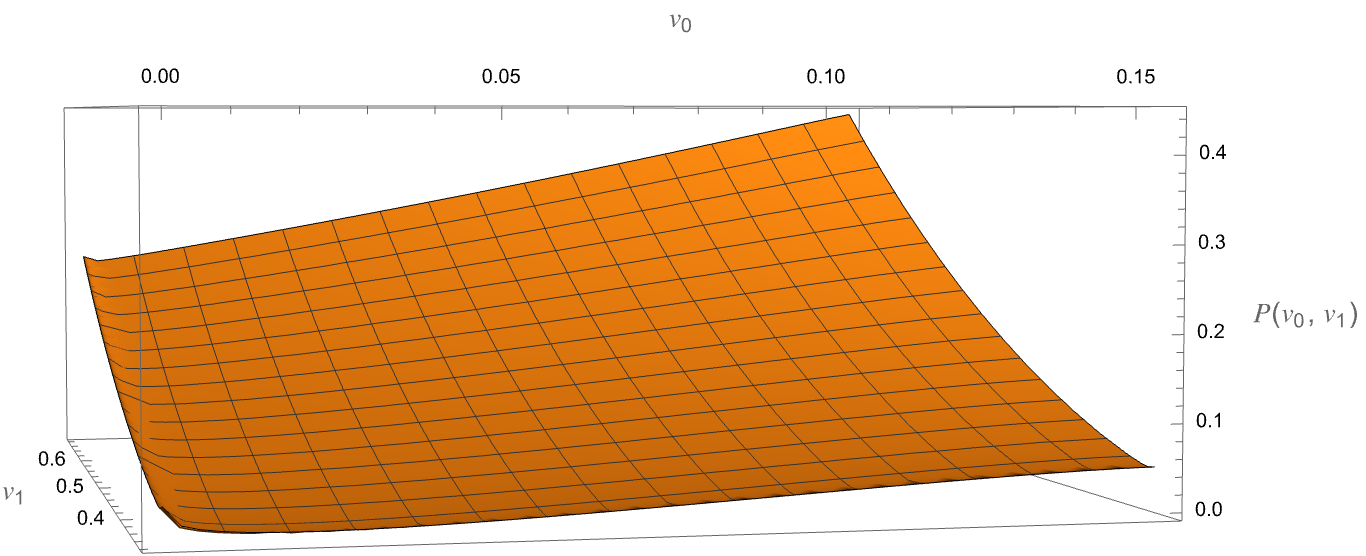}
     \end{center}
     \caption{
         \label{fig:P-plot}
         Plots of $P(v_0,v_1)$ for $(v_0,v_1) \in (0,\frac{4\pi}{81}) \times (\frac{1}{\pi},1-\frac{1}{\pi}]$ from different angles. 
              }
\end{figure}

See Figure \ref{fig:p123-plots} for plots of $P_1(v_0)$, $P_2(v_1)$ and $P_3(v_0,v_1)$. 
It is not hard to check that $P_1(v_0)$ is locally Lipschitz on $(0,\frac{4\pi}{81}]$ and behaves like $- c \sqrt{v_0} \log 1/v_0$ near $v_0 = 0$, 
and that $P_2(v_1)$ is smooth on $(\frac{1}{\pi},1 - \frac{1}{\pi}]$ and $1/2$-H\"older near $v_1 = \frac{1}{\pi}$. Similarly, $P_3(v_0,v_1)$ is smooth on $\mathcal{R} := (0,\frac{4\pi}{81}] \times (\frac{1}{\pi} , 1 - \frac{1}{\pi}]$ and $1/2$-H\"older near $v_0 = 0$ and $v_1 = \frac{1}{\pi}$. 
It turns out that $P(v_0,v_1) > 0$ on the entire $\mathcal{R}$ -- see Figure \ref{fig:P-plot}. According to Mathematica, $\min_{(v_0,v_1) \in \mathcal{R}} P(v_0,v_1) \simeq 0.000844955$, attained at $(v_0,v_1) \simeq (0.0201698, 0.32374)$. 
A rigorous justification of the positivity on $\mathcal{R}$ will be carried out in the subsequent subsections. 
We will first rigorously show that on $\mathcal{R}$ we have
\[
\frac{\partial P}{\partial v_0} \geq - c_1 \frac{\log{1/v_0}}{\sqrt{v_0}} ~,~ \frac{\partial P}{\partial v_1} \geq -c_2 \frac{1}{\sqrt{v_1 - 1/\pi}} 
\]
 for appropriate explicit constants $c_1 ,c_2 > 0$. Using this, we will construct an appropriately fine two-dimensional mesh $\mathcal{M} = \{(v^i_0,v^j_1)\}_{i=1,\ldots,N_0,j=1,\ldots,N_1}$ ensuring that 
 \[
\min_{(v_0,v_1) \in \mathcal{R}} P(v_0,v_1) \geq \min_{(v_0,v_1) \in \mathcal{M}} P(v_0,v_1) - 2\eps,
\]
for values of $\eps,\delta >0$ so that $2\eps + \delta < 0.000844955$. Finally, we will explicitly numerically compute $P(v_0,v_1)$ on this mesh with a working precision ensuring an error of at most $\delta$, thereby rigorously establishing that $P(v_0,v_1) > 0$ on $\mathcal{R}$. 

\begin{remark}
Our original numerical computation employed Wolfram's Mathematica 13.3 \cite{Mathematica-13.3}, whose manuals assert an arbitrary working precision when calculating elliptic integrals, as well as the other non-algebraic functions involved in the computation. However, as pointed out to us by a referee, it is better practice to use an open-source software for a rigorous justification. Consequently, we have also repeated the computation using FLINT \cite{Flint}, specifically its arbitrary precision ball arithmetic (Arb) \cite{Arb-manual}, which can compute incomplete elliptic integrals with rigorous precision intervals. Both softwares agreed that the minimum $\min_{(v_0,v_1) \in \mathcal{M}}  P(v_0,v_1)$ (taking into account the precision interval) is attained on exactly the same point on the mesh $\mathcal{M}$. The source code, written using the SageMath 10.6 environment \cite{Sagemath-10.6}, is available upon request. 
\end{remark}

Lastly, to handle the range $v_1 \in (1-\frac{1}{\pi},1)$, we shall use for simplicity the assumption that $\lambda \geq 0.8$ (even though it is not really required to establish that $Q(v_0,v_1) > 0$). Recalling that $\ell_{v_0,v_1}(v) = \I_{\T^2}(v_0) + \lambda (v-v_0)$ and that $\I_{\T^2}$ is concave, we see that $\ell_{v_0,v_1}(v)$ is minimized for a given $v$ and $\lambda$ when $v_0 = 0$, and hence $\ell_{v_0,v_1}(v) \geq \lambda v \geq 0.8 v$. Using this to lower bound the contribution of the integral in (\ref{eq:Q-formula}) on $[\frac{1}{\pi} , 1 - \frac{1}{\pi}]$, we obtain:
\begin{equation} \label{eq:QP4}
Q(v_0,v_1) \geq P_1(v_0) + P_4  ,
\end{equation}
where, denoting $c = \frac{1}{0.8}$, we have:
\begin{align}
\nonumber P_4 &= \int_{1/\pi}^{1-1/\pi} \frac{v - 4 \pi / 81}{\sqrt{c^2/v^2 - 1} } dv \\
& = \left . \brac{\frac{4 \pi}{81} - \frac{v}{2}} \sqrt{c^2 - v^2} + 
\label{eq:P4} \frac{c^2}{2} \tan^{-1}\brac{\frac{v}{\sqrt{c^2 - v^2}}} \right |_{v=\frac{1}{\pi}}^{v = 1 - \frac{1}{\pi}} \simeq 0.0597521 . 
\end{align}
On the other hand, according to Mathematica, the minimum of $P_1(v_0)$ on $v_0 \in (0,\frac{4\pi}{81})$ is approximately $-0.0447304$ (attained at $v_0 \simeq 0.0212857$). A rigorous justification that $P_1(v_0) >  -0.046$ will be carried out in the last subsection; as before, we will use a brute force computation on a fine mesh. We will thus conclude that $Q(v_0,v_1) > 0$ also when $v_1 > 1 - 1/\pi$ (and $\lambda \geq 0.8$). Up to justification of the various estimates and numerical computations, this concludes the proof of Proposition \ref{prop:Appendix}. The rest of the appendix is dedicated to a rigorous verification of the above arguments.

\subsection{$P_1(v_0)$}

Let us denote:
\begin{align*}
F_-(x,m) & := K(m) - F(x,m) = \int_{x}^1 \frac{dt}{\sqrt{(1-t^2)(1-m t^2)}}  ,\\
E_-(x,m) & := E(m) - E(x,m) = \int_{x}^1 \frac{\sqrt{1 - m t^2}}{\sqrt{1-t^2}} dt  .
\end{align*}

Directly differentiating, we calculate:
\begin{align*}
P_1'(v_0)  = \frac{1}{\sqrt{\pi v_0} (1 - \sqrt{\pi v_0}) } & \left ( \frac{2 \pi}{9}  \sqrt{\frac{1}{\pi} - \frac{4 \pi}{81}} \sqrt{\frac{4 \pi}{81} - v_0} \right . \\
& - \left . \brac{\frac{4 \pi}{81} - v_0} F_{-} \brac{\frac{\sqrt{\frac{1}{\pi }-\frac{4 \pi }{81}}}{\sqrt{\frac{1}{\pi}-v_0}} , 1-\pi  v_0} \right . \\
   &+  \left .  \brac{\frac{4 \pi}{81} +\frac{1}{\pi} -2  v_0}
   E_- \brac{\frac{\sqrt{\frac{1}{\pi}-\frac{4 \pi }{81}}}{\sqrt{\frac{1}{\pi }-v_0}} , 1-\pi v_0 ) } \right ) .
\end{align*}
Note that the first and third terms on the right-hand-side are non-negative for all $v_0 \in (0,\frac{4 \pi}{81})$, and so to lower-bound $P_1'(v_0)$ we just need to treat the second term above.

 \begin{lemma}
 For all $v_0 \in (0, \frac{4 \pi}{81})$:
 \[
 F_- \brac{\frac{\sqrt{\frac{1}{\pi}-\frac{4 \pi }{81}}}{\sqrt{\frac{1}{\pi }-v_0}} , 1-\pi v_0  } \leq \frac{\log(\frac{4 \pi}{9 \sqrt{\pi v_0}})}{\sqrt{1 - \frac{4 \pi^2}{81}}} . 
\]
\end{lemma}
\begin{proof}
Defining $\lambda = \frac{\sqrt{\frac{1}{\pi}-\frac{4 \pi }{81}}}{\sqrt{\frac{1}{\pi }-v_0}}$ and $m = 1-\pi v_0$, we have after changing variables $u^2 = 1 - t^2$:
\[
F_-(\lambda,m) = \int_{\lambda}^1 \frac{dt}{\sqrt{(1-t^2)(1-m t^2)}}= \int_{0}^{\sqrt{1-\lambda^2}} \frac{du}{\sqrt{(1-u^2)(1-m+ m u^2)}} . 
\]
Setting $s^2 = 1-m+ m u^2$, recalling the definition of $m$ and $\lambda$, and noting that $\sqrt{1-m + m (1-\lambda^2)} = \frac{2 \pi}{9}$, we obtain:
\[
= \int_{\sqrt{\pi v_0}}^{\frac{2\pi}{9}} \frac{ds}{\sqrt{(s^2 - \pi v_0)(1-s^2)}} . 
\]
Finally, setting $s = \sqrt{\pi v_0} z$, we conclude:
\[
= \int_1^{\frac{2 \pi}{9 \sqrt{\pi v_0}}} \frac{dz}{\sqrt{(z^2 - 1)(1 - \pi v_0 z^2)}} . 
\]
Denoting $a = \frac{2 \pi}{9 \sqrt{\pi v_0}}$, we therefore bound:
\[
F_-(\lambda,m) \leq \frac{1}{\sqrt{1 - \frac{4 \pi^2}{81}}} \int_1^{a} \frac{dz}{\sqrt{z^2 - 1}} = \frac{\log( a + \sqrt{a^2-1})}{\sqrt{1 - \frac{4 \pi^2}{81}}} \leq \frac{\log(2a)}{\sqrt{1 - \frac{4 \pi^2}{81}}}  . 
\]
\end{proof}

\begin{lemma}
\[
\max_{v_0 \in [0,\frac{4\pi}{81}]} \frac{\frac{4\pi}{81} - v_0}{1 - \sqrt{\pi v_0}} = \frac{\frac{4 \pi}{81} - \frac{1}{\pi} \brac{ 1  - \sqrt{1-\frac{4 \pi ^2}{81}}}^2 }{\sqrt{1 - \frac{4\pi^2}{81}}} . 
\]
\end{lemma}
\begin{proof}
Set $v_0 = x^2$ and differentiate. The maximum is attained at $v_0 = \frac{1}{\pi} \brac{ 1  - \sqrt{1-\frac{4 \pi ^2}{81}}}^2$. 
\end{proof}
\begin{corollary}
\[
\frac{1}{\sqrt{\pi} \sqrt{1 - \frac{4\pi^2}{81}}} \max_{v_0 \in [0,\frac{4\pi}{81}]} \frac{\frac{4\pi}{81} - v_0}{1 - \sqrt{\pi v_0}} \simeq 0.142487 < \frac{1}{7} . 
\]
\end{corollary}

Combining the previous estimates, we obtain:
\begin{lemma}
For all $v_0 \in (0,\frac{4\pi}{81})$:
\[
P_1'(v_0) \geq - \frac{\log\brac{\frac{4 \pi}{9 \sqrt{\pi v_0}}}}{7 \sqrt{v_0}} . 
\]
\end{lemma}

 \subsection{$P_2(v_1)$}
 
 \begin{lemma}
 $P_2'(v_1) < 0$ and $\sqrt{\pi ^2 v_1^2-1} |P_2'(v_1)|$ is decreasing on $v_1 \in [\frac{1}{\pi} , \infty)$. Therefore: \[
\max_{v_1 \in [\frac{1}{\pi} , 1 - \frac{1}{\pi}]} \sqrt{\pi ^2
   v_1^2-1} \abs{P_2'(v_1)} = 2\brac{\frac{1}{\pi} - \frac{4 \pi}{81}} \simeq 0.326339 < \frac{1}{3} .
\]
In particular, for all $v_1 \in [\frac{1}{\pi} , 1 - \frac{1}{\pi}]$:
\[
0 > P_2'(v_1) > -\frac{1}{3 \sqrt{2 \pi}} \frac{1}{\sqrt{v_1 - \frac{1}{\pi}}} .
\]
\end{lemma}
\begin{proof}
By concavity of $\I_{\T^2}$ (in fact, strict concavity at $0$), $\ell_{0,v_1}$ is pointwise decreasing in $v_1$. Consequently, $P_2(v_1)$ is immediately seen to be decreasing from its integral definition (\ref{eq:P2-int}).
 Directly differentiating, we calculate:
 \begin{equation} \label{eq:P2d}
 P_2'(v_1) = \frac{4 v_1 \brac{\frac{1}{2} + \frac{2 \pi^2}{81} + \sqrt{\pi^2 v_1^2-\frac{4 \pi^2}{81}} \sqrt{\pi ^2 v_1^2-1} - \pi^2 v_1^2 }}{\sqrt{\pi ^2
   v_1^2-1}} .
\end{equation}
Denoting $x = \pi v_1$ and $b = \frac{4 \pi^2}{81} < 1$, the monotonicity claim then boils down to showing that:
\[
x \brac{x^2 - \frac{1 + b}{2} - \sqrt{x^2 - b} \sqrt{x^2 - 1}}
\]
is decreasing on $[1,\infty)$. Denoting $y = x^2 - \frac{1+b}{2}$, note that $(x^2 - b) (x^2-1) = (y+\frac{1-b}{2})(y-\frac{1-b}{2}) = y^2 - c^2$ where $c = \frac{1-b}{2}$, and so the task is equivalent to showing that:
\[
\sqrt{y + \frac{1+b}{2}} \brac{y - \sqrt{y^2 - c^2}} = \sqrt{\frac{1}{y} + \frac{1+b}{2 y^2} } \brac{y^2 - \sqrt{y^4 - c^2 y^2}} 
\]
is decreasing on $[c,\infty)$. The first term on the right is clearly decreasing, and the second term is decreasing iff (denoting $z = y^2$)
\[
z - \sqrt{z^2 - c^2 z} 
\]
is decreasing on $[c^2,\infty)$. The latter is directly checked by differentiation. 

Consequently, the maximum of $\sqrt{\pi ^2  v_1^2-1} \abs{P_2'(v_1)}$ over  $[\frac{1}{\pi} , 1 - \frac{1}{\pi}]$ is attained at $v_1 = \frac{1}{\pi}$, yielding the first asserted bound by (\ref{eq:P2d}). Using that $\sqrt{ \pi v_1 + 1} \geq \sqrt{2}$ when $v_1 \geq \frac{1}{\pi}$, the second bound follows. 
\end{proof}
 
 \subsection{$P_3(v_0,v_1)$}
 
\begin{lemma}
For all $v_0 \in [0,\frac{4\pi}{81}]$ and $v_1 \in [\frac{1}{\pi} , 1 - \frac{1}{\pi}]$, $\frac{\partial P_3(v_0,v_1)}{\partial v_0} \geq 0$ and $\frac{\partial P_3(v_0,v_1)}{\partial v_1} \geq 0$.
\end{lemma}
\begin{proof} 
Recall that by concavity of $\I_{\T^2}$, $\ell_{v_0,v_1}$ is pointwise non-decreasing in $v_0$. Consequently, $P_3$ is immediately seen to be non-decreasing in $v_0$ from its integral definition (\ref{eq:P3-int}), and hence $\frac{\partial P_3}{\partial v_0} \geq 0$. To see the monotonicity in $v_1$ we proceed as follows. 

Define:
 \[
 z = z(v_0,v_1) = \frac{\sqrt{\pi v_0 } v_1 - v_0}{1 - \sqrt{\pi v_0 }} ,
 \]
 and calculate:
 \[
\frac{\partial z}{\partial v_0} = \frac{ \sqrt{\pi} v_1 - (2 \sqrt{v_0} - \sqrt{\pi} v_0)}{ 2 (1-\sqrt{\pi v_0})^2 \sqrt{v_0}} ~,~ \frac{\partial z}{\partial v_1} =  \frac{\sqrt{\pi v_0}}{1 - \sqrt{\pi v_0}} .
 \]
Consequently, whenever $v_0 \in [0,\frac{1}{\pi})$ and $v_1 \geq \frac{1}{\pi}$, we see that $z \geq 0$, $\frac{\partial z}{\partial v_0} \geq 0$ and $\frac{\partial z}{\partial v_1}\geq 0$ (with strict inequality when $v_0 > 0$).
 
 Expressing $P_3$ as a function of $z$ and $v_1$, we have:
 \[
 P_3(z,v_1) = (z +v_1)^2 \tan^{-1}\left(\frac{\sqrt{v_1-\frac{1}{\pi
   }}}{\sqrt{2 z+v_1+\frac{1}{\pi }}}\right) 
   +\frac{1}{2} \left(\frac{1}{\pi } -\frac{8 \pi }{81} -z \right) \sqrt{\left(v_1-\frac{1}{\pi
   }\right) \left(2 z +v_1+\frac{1}{\pi }\right)} .
 \]
 A calculation verifies:
 \begin{align*}
\frac{\partial P_3(z,v_1)}{\partial z} & =  2 (z+v_1) \tan^{-1}\left(\frac{\sqrt{v_1-\frac{1}{\pi
   }}}{\sqrt{2 z+v_1+\frac{1}{\pi }}}\right)  -  \brac{2 z +v_1 + \frac{4 \pi}{81} } \frac{\sqrt{v_1-\frac{1}{\pi
   }}}{\sqrt{2 z+v_1+\frac{1}{\pi }}} ,\\
   \frac{\partial P_3(z,v_1)}{\partial v_1} & =  (z+v_1) \brac{2  \tan^{-1}\left(\frac{\sqrt{v_1-\frac{1}{\pi
   }}}{\sqrt{2 z+v_1+\frac{1}{\pi }}}\right) + \frac{\frac{1}{\pi} - \frac{4\pi}{81}}{ \sqrt{v_1 - \frac{1}{\pi}}\sqrt{2 z+v_1+\frac{1}{\pi }}  }} .  
 \end{align*}
 Since $z \geq 0$, we immediately see that $\frac{\partial P_3(z,v_1)}{\partial v_1} \geq 0$. In addition, since
 \[
 0 \leq  \frac{\partial P_3(v_0,v_1)}{\partial v_0} = \frac{\partial P_3(z,v_1)}{\partial z} \frac{\partial z}{\partial v_0} 
 \]
 and $\frac{\partial z}{\partial v_0} > 0$ if $v_0 > 0$, we deduce that $\frac{\partial P_3(z,v_1)}{\partial z} \geq 0$ (also when $v_0 = 0$ by continuity). It follows that:
 \[
 \frac{\partial P_3(v_0,v_1)}{\partial v_1} = \frac{\partial P_3(z,v_1)}{\partial z} \frac{\partial z}{\partial v_1} + \frac{\partial P_3(z,v_1)}{\partial v_1} \geq 0,
 \]
 as asserted.  

 \end{proof}
 
 \subsection{Concluding when $v_1 \in [\frac{1}{\pi} , 1 - \frac{1}{\pi}]$}
 
Combining all of the prior estimates, we obtain:
\begin{proposition}
For all $v_0 \in (0,\frac{4\pi}{81}]$ and $v_1 \in (\frac{1}{\pi},1-\frac{1}{\pi}]$:
\begin{equation} \label{eq:Holder}
\frac{\partial P(v_0,v_1)}{\partial v_0} \geq - D_0(v_0)  ~,~ \frac{\partial P(v_0,v_1)}{\partial v_1} \geq  -D_1(v_1) , 
\end{equation}
where:
\[
D_0(v_0) := \frac{\log\brac{\frac{4 \pi}{9 \sqrt{\pi v_0}}}}{7 \sqrt{v_0}} ~,~ D_1(v_1) := \frac{1}{3 \sqrt{2 \pi}} \frac{1}{\sqrt{v_1 - \frac{1}{\pi}}}  . 
\]
Integrating, we obtain the following monotone increasing functions:
\[
L_0(v_0) := \frac{\sqrt{v_0}}{7} \log\brac{\frac{16 e^2 \pi}{81 v_0}} ~,~ L_1(v_1) := \frac{\sqrt{2}}{3 \sqrt{\pi}} \sqrt{v_1 - \frac{1}{\pi}} .
\]
Setting
\[
L(v_0,v_1) = L_0(v_0) + L_1(v_1) ,
\]
it follows that for all $0 \leq v_0 \leq v_0' \leq \frac{4\pi}{81}$ and $\frac{1}{\pi} \leq v_1 \leq v_1' \leq 1 - \frac{1}{\pi}$,
\[
P(v_0', v_1') - P(v_0,v_1) \geq - (L(v_0',v_1') - L(v_0,v_1)) . 
\]
\end{proposition}

Define:
\begin{align*}
A_0 & = L_0\brac{\frac{4 \pi}{81}} = \frac{1}{7} \cdot \frac{2}{9} \sqrt{\pi} \log(4 e^2) \simeq 0.190541, \\
A_1 & = L_1\brac{1 -\frac{1}{\pi}} = \frac{\sqrt{2}}{3 \sqrt{\pi}} \sqrt{1 - \frac{2}{\pi}} \simeq 0.160324 ,
\end{align*}
and set:
\[
\eps := 3 \cdot 10^{-4} ~,~ N_0 := \lceil A_0/\eps \rceil = 636 ~,~ N_1 := \lceil A_1/\eps \rceil = 535 .
\]
We can now define the $636 \times 535$ mesh $\mathcal{M}$ as follows:
\begin{align*}
& s^i_0 := (i-1) \eps ~,~ v^i_0 := L_0^{-1}(s^i_0) ~,~  i = 1,\ldots,N_0 ,\\
& s^j_1 := (j-1)  \eps ~,~ v^j_1 := L_1^{-1}(s^j_1) ~,~ j = 1,\ldots,N_1 , \\
& \mathcal{M} = \mathcal{M}_0 \times \mathcal{M}_1 ~,~ \mathcal{M}_0 = \{ v^i_0 \}_{i=1,\ldots,N_0} ~,~ 
\mathcal{M}_1 = \{ v^j_1 \}_{j=1,\ldots,N_1} . 
\end{align*}
It follows that:
\begin{equation} \label{eq:2eps}
\min_{(v_0,v_1) \in [0, \frac{4\pi}{81}] \times [\frac{1}{\pi} , 1 - \frac{1}{\pi}]} P(v_0,v_1) \geq \min_{(v_0,v_1) \in \mathcal{M}} P(v_0,v_1) - 2 \eps ,
\end{equation}
and it remains to numerically verify that the right-hand-side is positive.  

In practice, since the inverse of $L_0$ does not have a closed form, we prefer to avoid any issues with estimating the numerical precision involved in inverting it, which would then need to be corrected by using (\ref{eq:Holder}) again. Instead, we define $\mathcal{M}_0$ in a different manner than above. We first set $v^1_0 = 0$ and $v^2_0 = 10^{-8}$, since $L_0(v^2_0) \simeq 0.0002849 < \eps$, and this is well within Mathematica or FLINT's default precision. We then recursively define:
\[
v^{i+1}_0 := \textsf{Compute}\brac{ v^i_0 + \frac{0.99 \eps}{D_0(v^i_0)} } ~,~ i \geq 2 ,
\]
until the first time $v^{i+1}_0$ exceeds $\frac{4\pi}{81}$, at which point we stop and set $N_0 = i$. 
The computation of the right-hand-side in FLINT, denoted by $\textsf{Compute}$, is performed using the default 53 bit precision; since
$0.25 \leq D_0(v^i_0) \leq 12817$, $\frac{\eps}{D_0(v^i_0)} \geq 2 \cdot 10^{-8}$ and so this ensures that:
\[
\textsf{Compute}\brac{ v^i_0 + \frac{0.99 \eps}{D_0(v^i_0)} } \leq v^i_0 + \frac{\eps}{D_0(v^i_0)} .
\]
This produces a collection of $N_0=647$ points $\mathcal{M}_0 = \{v^i_0\}_{i=1,\ldots,N_0}$, and we are guaranteed that 
\[
L_0(v_0^{i+1}) - L_0(v_0^{i}) = \int_{v_0^i}^{v_0^{i+1}} D_0(v) dv \leq D_0(v^i_0) (v^{i+1}_0 - v^i_0) \leq \eps , \]
where we've used that $D_0(v_0)$ is decreasing on $[0,\frac{4\pi}{81}]$ (as checked by direct differentiation). 
This simple procedure thus avoids any error accumulation in the computations. 
The latter issue does not appear for the inverse of $L_1$, since $L_1^{-1}(s) =  \frac{1}{\pi} + \frac{9 \pi}{2} s^2$ and both Mathematica and FLINT maintain infinite precision in algebraic computations. We thus obtain a $647 \times 535$ mesh $\mathcal{M} = \mathcal{M}_0 \times \mathcal{M}_1$.

According to FLINT, the minimum of $P$ over our mesh satisfies:
\[
\min_{(v_0,v_1) \in \mathcal{M}} P(v_0,v_1) \geq 0.0008450618226248 ,
\]
with the smallest lower-bound  being attained at $(v_0,v_1) \in \mathcal{M}$ with $v_0 \simeq 0.020145278695067$ 
and $v_1 = \frac{1}{\pi} + \frac{9 \pi}{2} \brac{\frac{39}{2000}}^2$. Recalling (\ref{eq:QP}), (\ref{eq:2eps}) and that $\eps = 0.0003$, 
this rigorously establishes that
$\min_{(v_0,v_1) \in [0, \frac{4\pi}{81}] \times [\frac{1}{\pi} , 1 - \frac{1}{\pi}]} Q(v_0,v_1) > 0$. It remains to verify the positivity of $Q$ when $v_1 > 1 - \frac{1}{\pi}$ and $\lambda \geq 0.8$.

\subsection{Concluding when $v_1 \in (1 - \frac{1}{\pi},1)$}

As already established in (\ref{eq:QP4}), to handle the case when $v_1 \in (1 - \frac{1}{\pi},1)$ it remains to get a good bound on $\min_{v_0 \in [0,\frac{4\pi}{81}]} P_1(v_0)$. Reusing the mesh $\mathcal{M}_0$ from the previous subsection, we are guaranteed:
\[
\min_{v_0 \in [0,\frac{4\pi}{81}]} P_1(v_0) \geq \min_{v_0 \in \mathcal{M}_0} P_1(v_0) - \eps . 
\]
According to FLINT, the minimum of $P_1$ over $\mathcal{M}_0$ satisfies:
\[
\min_{v_0 \in \mathcal{M}_0} P_1(v_0) \geq   -0.04473032245783 ,
\]
with the smallest lower-bound being attained at $v_0 \in \mathcal{M}_0$ with $v_0 \simeq 0.0213752137750056$.
Since $\eps = 0.0003$, it follows that
\[
 \min_{v_0 \in [0,\frac{4\pi}{81}]} P_1(v_0) \geq -0.046 .
\]
Together with (\ref{eq:QP}), (\ref{eq:QP4}) and (\ref{eq:P4}), this confirms that $\min_{(v_0,v_1) \in [0,\frac{4\pi}{81}] \times (1-\frac{1}{\pi},1)} Q(v_0,v_1) > 0$ (whenever $\lambda \geq 0.8$), and therefore concludes the proof of Proposition \ref{prop:Appendix}.

\bibliographystyle{plain}
\bibliography{../../../ConvexBib}

\def\cprime{$'$} \def\textasciitilde{$\sim$}
\begin{thebibliography}{10}

\bibitem{AcerbiFuscoMorini}
E.~Acerbi, N.~Fusco, and M.~Morini.
\newblock Minimality via second variation for a nonlocal isoperimetric problem.
\newblock {\em Comm. Math. Phys.}, 322(2):515--557, 2013.

\bibitem{BakryEmery}
D.~Bakry and M.~{\'E}mery.
\newblock Diffusions hypercontractives.
\newblock In {\em S\'eminaire de probabilit\'es, XIX, 1983/84}, volume 1123 of
  {\em Lecture Notes in Math.}, pages 177--206. Springer, Berlin, 1985.

\bibitem{BarbosaDoCarmo-StabilityInRn}
J.~L. Barbosa and M.~do~Carmo.
\newblock Stability of hypersurfaces with constant mean curvature.
\newblock {\em Math. Z.}, 185(3):339--353, 1984.

\bibitem{BarchiesiCagnettiFusco}
M.~Barchiesi, F.~Cagnetti, and N.~Fusco.
\newblock Stability of the {S}teiner symmetrization of convex sets.
\newblock {\em J. Eur. Math. Soc. (JEMS)}, 15(4):1245--1278, 2013.

\bibitem{BartheTensorizationGAFA}
F.~Barthe.
\newblock Log-concave and spherical models in isoperimetry.
\newblock {\em Geom. Funct. Anal.}, 12(1):32--55, 2002.

\bibitem{BartheMaureyIsoperimetricInqs}
F.~Barthe and B.~Maurey.
\newblock Some remarks on isoperimetry of {G}aussian type.
\newblock {\em Ann. Inst. H. Poincar\'e Probab. Statist.}, 36(4):419--434,
  2000.

\bibitem{BavardPansu}
C.~Bavard and P.~Pansu.
\newblock Sur le volume minimal de $\bold {R}\sp 2$.
\newblock {\em Ann. Sci. \'Ecole Norm. Sup.}, 19(4):479--490, 1986.

\bibitem{BayleThesis}
V.~Bayle.
\newblock {\em Propri\'et\'es de concavit\'e du profil isop\'erim\'etrique et
  applications}.
\newblock PhD thesis, Institut Joseph Fourier, Grenoble, 2004.

\bibitem{BayleRosales}
V.~Bayle and C.~Rosales.
\newblock Some isoperimetric comparison theorems for convex bodies in
  {R}iemannian manifolds.
\newblock {\em Indiana Univ. Math. J.}, 54(5):1371--1394, 2005.

\bibitem{BobkovExtremalHalfSpaces}
S.~Bobkov.
\newblock Extremal properties of half-spaces for log-concave distributions.
\newblock {\em Ann. Probab.}, 24(1):35--48, 1996.

\bibitem{BobkovGaussianIsopInqViaCube}
S.~G. Bobkov.
\newblock An isoperimetric inequality on the discrete cube, and an elementary
  proof of the isoperimetric inequality in {G}auss space.
\newblock {\em Ann. Probab.}, 25(1):206--214, 1997.

\bibitem{Borell-GaussianIsoperimetry}
Ch. Borell.
\newblock The {B}runn--{M}inkowski inequality in {G}auss spaces.
\newblock {\em Invent. Math.}, 30:207--216, 1975.

\bibitem{BrezisBruckstein}
H.~Brezis and A.~Bruckstein.
\newblock A sharp relative isoperimetric inequality for the square.
\newblock {\em C. R. Math. Acad. Sci. Paris}, 359:1191--1199, 2021.

\bibitem{Castro-Products}
K.~Castro.
\newblock Isoperimetric inequalities in cylinders with density.
\newblock {\em Nonlinear Anal.}, 217:Paper No. 112726, 10, 2022.

\bibitem{ChambersEtAl-IsoperimetryOnCube}
G.~Chambers and L.~Mouill\'e.
\newblock On the relative isoperimetric problem for the cube.
\newblock arxiv.org/abs/2302.04382, 2023.

\bibitem{CFMP-GaussianIsoperimetricStability}
A.~Cianchi, N.~Fusco, F.~Maggi, and A.~Pratelli.
\newblock On the isoperimetric deficit in {G}auss space.
\newblock {\em Amer. J. Math.}, 133(1):131--186, 2011.

\bibitem{EhrhardPhiConcavity}
A.~Ehrhard.
\newblock Sym\'etrisation dans l'espace de {G}auss.
\newblock {\em Math. Scand.}, 53(2):281--301, 1983.

\bibitem{FuscoMaggiPratelli-GaussianProduct}
N.~Fusco, F.~Maggi, and A.~Pratelli.
\newblock On the isoperimetric problem with respect to a mixed
  {E}uclidean-{G}aussian density.
\newblock {\em J. Funct. Anal.}, 260(12):3678--3717, 2011.

\bibitem{Glaudo-IsoperimetryOnCube}
F.~Glaudo.
\newblock On the isoperimetric profile of the hypercube.
\newblock arxiv.org/abs/2305.20051, 2023.

\bibitem{Gonzalo-Products}
J.~Gonzalo.
\newblock Soap bubbles and isoperimetric regions in the product of a closed
  manifold with {E}uclidean space.
\newblock arxiv.org/abs/1312.6311, 2013.

\bibitem{Gruter}
M.~Gr{\"u}ter.
\newblock Boundary regularity for solutions of a partitioning problem.
\newblock {\em Arch. Rational Mech. Anal.}, 97(3):261--270, 1987.

\bibitem{HadwigerCube}
H.~Hadwiger.
\newblock Gitterperiodische {P}unktmengen und {I}soperimetrie.
\newblock {\em Monatsh. Math.}, 76:410--418, 1972.

\bibitem{HPRR-PeriodicIsoperimetricProblem}
L.~Hauswirth, J.~P\'{e}rez, P.~Romon, and A.~Ros.
\newblock The periodic isoperimetric problem.
\newblock {\em Trans. Amer. Math. Soc.}, 356(5):2025--2047, 2004.

\bibitem{Howards-BScThesis}
H.~Howards.
\newblock {\em Soap Bubbles on Surfaces}.
\newblock Undergraduate {T}hesis, Williams College, March 1992.

\bibitem{HHM-Surfaces}
H.~Howards, M.~Hutchings, and F.~Morgan.
\newblock The isoperimetric problem on surfaces.
\newblock {\em Amer. Math. Monthly}, 106(5):430--439, 1999.

\bibitem{Arb-manual}
F.~Johansson.
\newblock {A}rb: Efficient arbitrary-precision midpoint-radius interval
  arithmetic.
\newblock {\em {IEEE} Transactions on Computers}, 66(8):1281--1292, 2017.

\bibitem{KMS-LimitOfCapillarity}
D.~King, F.~Maggi, and S.~Stuvard.
\newblock Plateau's problem as a singular limit of capillarity problems.
\newblock {\em Comm. Pure Appl. Math.}, 75(5):895--969, 2022.

\bibitem{Koiso}
M.~Koiso.
\newblock Deformation and stability of surfaces with constant mean curvature.
\newblock {\em Tohoku Math. J. (2)}, 54(1):145--159, 2002.

\bibitem{KoisoMiyamoto}
M.~Koiso and U.~Miyamoto.
\newblock Stability of hypersurfaces with constant mean curvature trapped
  between two parallel hyperplanes.
\newblock {\em Jpn. J. Ind. Appl. Math.}, 41(1):233--268, 2024.

\bibitem{Kuwert}
E.~Kuwert.
\newblock Note on the isoperimetric profile of a convex body.
\newblock In {\em Geometric analysis and nonlinear partial differential
  equations}, pages 195--200. Springer, Berlin, 2003.

\bibitem{MaggiBook}
F.~Maggi.
\newblock {\em Sets of finite perimeter and geometric variational problems: an
  introduction to {G}eometric {M}easure {T}heory}, volume 135 of {\em Cambridge
  Studies in Advanced Mathematics}.
\newblock Cambridge University Press, Cambridge, 2012.

\bibitem{EMilman-RoleOfConvexity}
E.~Milman.
\newblock On the role of convexity in isoperimetry, spectral-gap and
  concentration.
\newblock {\em Invent. Math.}, 177(1):1--43, 2009.

\bibitem{EMilman-SpectrumAndContractions}
E.~Milman.
\newblock Spectral estimates, contractions and hypercontractivity.
\newblock {\em J. Spectr. Theory}, 8(2):669--714, 2018.

\bibitem{EMilmanNeeman-GaussianMultiBubble}
E.~Milman and J.~Neeman.
\newblock The {G}aussian double-bubble and multi-bubble conjectures.
\newblock {\em Ann. of Math. (2)}, 195(1):89--206, 2022.

\bibitem{EMilmanNeeman-QuintupleBubble}
E.~Milman and J.~Neeman.
\newblock Plateau bubbles and the quintuple-bubble conjecture.
\newblock arxiv.org:2307.08164, 2023.

\bibitem{EMilmanNeeman-TripleAndQuadruple}
E.~Milman and J.~Neeman.
\newblock The structure of isoperimetric bubbles on $\mathbb{R}^n$ and
  $\mathbb{S}^n$.
\newblock {\em Acta Math.}, 234(1):71--188, 2025.

\bibitem{MorganRegularityOfMinimizers}
F.~Morgan.
\newblock Regularity of isoperimetric hypersurfaces in {R}iemannian manifolds.
\newblock {\em Trans. Amer. Math. Soc.}, 355(12):5041--5052 (electronic), 2003.

\bibitem{Morgan-Products}
F.~Morgan.
\newblock Isoperimetric estimates in products.
\newblock {\em Ann. Global Anal. Geom.}, 30(1):73--79, 2006.

\bibitem{Morgan-Polytopes}
F.~Morgan.
\newblock In polytopes, small balls about some vertex minimize perimeter.
\newblock {\em J. Geom. Anal.}, 17(1):97--106, 2007.

\bibitem{MorganBook4Ed}
F.~Morgan.
\newblock {\em Geometric measure theory (a beginner's guide)}.
\newblock Elsevier/Academic Press, Amsterdam, fourth edition, 2009.

\bibitem{MorganJohnson}
F.~Morgan and D.~L. Johnson.
\newblock Some sharp isoperimetric theorems for {R}iemannian manifolds.
\newblock {\em Indiana Univ. Math. J.}, 49(3):1017--1041, 2000.

\bibitem{Pedrosa-SphericalCylinders}
R.~H.~L. Pedrosa.
\newblock The isoperimetric problem in spherical cylinders.
\newblock {\em Ann. Global Anal. Geom.}, 26(4):333--354, 2004.

\bibitem{PedrosaRitore-Products}
R.~H.~L. Pedrosa and M.~Ritor\'{e}.
\newblock Isoperimetric domains in the {R}iemannian product of a circle with a
  simply connected space form and applications to free boundary problems.
\newblock {\em Indiana Univ. Math. J.}, 48(4):1357--1394, 1999.

\bibitem{Ritore-Genus4}
M.~Ritor\'{e}.
\newblock Index one minimal surfaces in flat three space forms.
\newblock {\em Indiana Univ. Math. J.}, 46(4):1137--1153, 1997.

\bibitem{Ritore-IsoperimetricBook}
M.~Ritor\'{e}.
\newblock {\em Isoperimetric inequalities in {R}iemannian manifolds}, volume
  348 of {\em Progress in Mathematics}.
\newblock Birkh\"{a}user/Springer, Cham, [2023] \copyright 2023.

\bibitem{RitoreRos-RP3}
M.~Ritor\'{e} and A.~Ros.
\newblock Stable constant mean curvature tori and the isoperimetric problem in
  three space forms.
\newblock {\em Comment. Math. Helv.}, 67(2):293--305, 1992.

\bibitem{RitoreRos-CompactnessOfStableCMC}
M.~Ritor\'{e} and A.~Ros.
\newblock The spaces of index one minimal surfaces and stable constant mean
  curvature surfaces embedded in flat three manifolds.
\newblock {\em Trans. Amer. Math. Soc.}, 348(1):391--410, 1996.

\bibitem{RitoreVernadakis-Products}
M.~Ritor\'{e} and E.~Vernadakis.
\newblock Large isoperimetric regions in the product of a compact manifold with
  {E}uclidean space.
\newblock {\em Adv. Math.}, 306:958--972, 2017.

\bibitem{RosIsoperimetricProblemNotes}
A.~Ros.
\newblock The isoperimetric problem.
\newblock In {\em Global theory of minimal surfaces}, volume~2 of {\em Clay
  Math. Proc.}, pages 175--209. Amer. Math. Soc., Providence, RI, 2005.

\bibitem{Ros-StableGenus3}
A.~Ros.
\newblock One-sided complete stable minimal surfaces.
\newblock {\em J. Differential Geom.}, 74(1):69--92, 2006.

\bibitem{RCBMIsopInqsForLogConvexDensities}
C.~Rosales, A.~Ca{\~n}ete, V.~Bayle, and F.~Morgan.
\newblock On the isoperimetric problem in {E}uclidean space with density.
\newblock {\em Calc. Var. Partial Differential Equations}, 31(1):27--46, 2008.

\bibitem{SternbergZumbrun}
P.~Sternberg and K.~Zumbrun.
\newblock On the connectivity of boundaries of sets minimizing perimeter
  subject to a volume constraint.
\newblock {\em Comm. Anal. Geom.}, 7(1):199--220, 1999.

\bibitem{SudakovTsirelson}
V.~N. Sudakov and B.~S. Cirel{\cprime}son~[Tsirelson].
\newblock Extremal properties of half-spaces for spherically invariant
  measures.
\newblock {\em Zap. Nau\v cn. Sem. Leningrad. Otdel. Mat. Inst. Steklov.
  (LOMI)}, 41:14--24, 165, 1974.
\newblock Problems in the theory of probability distributions, II.

\bibitem{Flint}
The~{FLINT} team.
\newblock {\em {FLINT}: {F}ast {L}ibrary for {N}umber {T}heory}, 2025.
\newblock Version 3.2.1, \url{https://flintlib.org}.

\bibitem{Sagemath-10.6}
{The Sage Developers}.
\newblock {\em {SageMath, the Sage Mathematics Software System}}, 2025.
\newblock Version 10.6, \url{https://www.sagemath.org}.

\bibitem{Mathematica-13.3}
{Wolfram Research, Inc.}
\newblock {\em Mathematica}, 2024.
\newblock Version 13.3, \url{https://www.wolfram.com}.

\end{thebibliography}

\end{document}